\documentclass[a4paper, 12pt, american,reqno]{amsart}
\usepackage[T1]{fontenc}
\usepackage[utf8]{inputenc}
\usepackage[active]{srcltx}
\usepackage{color}
\usepackage{bbm}
\usepackage{dsfont}
\usepackage{amsmath}
\usepackage{amsthm}
\usepackage{amssymb}
\usepackage{esint}

\makeatletter
\numberwithin{equation}{section}
\numberwithin{figure}{section}
\theoremstyle{plain}
\newtheorem{thm}{\protect\theoremname}[section]
\theoremstyle{definition}
\newtheorem{defn}[thm]{\protect\definitionname}
\theoremstyle{plain}
\newtheorem{lem}[thm]{\protect\lemmaname}
\theoremstyle{plain}
\newtheorem{prop}[thm]{\protect\propositionname}
\newcommand{\lyxaddress}[1]{
	\par {\raggedright #1
	\vspace{1.4em}
	\noindent\par}
}

\usepackage[T1]{fontenc}
\usepackage[utf8]{inputenc}
\usepackage[a4paper]{geometry}
\usepackage{color}
\usepackage{dsfont}
\usepackage{amsmath}
\usepackage{amsthm}
\usepackage{amssymb}
\usepackage{setspace}
\usepackage{esint}

\usepackage[colorlinks,pdfpagelabels,pdfstartview = FitH,bookmarksopen
= true,bookmarksnumbered = true,linkcolor = blue,plainpages =
false,hypertexnames = false,citecolor = red,pagebackref=false]{hyperref}

\makeatletter
\numberwithin{equation}{section}
\numberwithin{figure}{section}
\theoremstyle{plain}

\@ifundefined{date}{}{\date{}}
\usepackage{babel}
\usepackage{eucal}
\usepackage{epstopdf}
\usepackage{graphicx}
\theoremstyle{plain}

\newtheoremstyle{boldremark}
    {\dimexpr\topsep/2\relax} 
    {\dimexpr\topsep/2\relax} 
    {}          
    {}          
    {\bfseries} 
    {.}         
    {.5em}      
    {}          

\theoremstyle{boldremark}
\newtheorem{brem} [thm] {Remark} 

\renewcommand{\epsilon}{\varepsilon}

\usepackage{fancyhdr}
\usepackage{blindtext}
\allowdisplaybreaks

\makeatother

\usepackage{babel}
\addto\captionsbritish{\renewcommand{\definitionname}{Definition}}
\addto\captionsbritish{\renewcommand{\lemmaname}{Lemma}}
\addto\captionsbritish{\renewcommand{\theoremname}{Theorem}}
\addto\captionsenglish{\renewcommand{\definitionname}{Definition}}
\addto\captionsenglish{\renewcommand{\lemmaname}{Lemma}}
\addto\captionsenglish{\renewcommand{\theoremname}{Theorem}}
\providecommand{\definitionname}{Definition}
\providecommand{\lemmaname}{Lemma}
\providecommand{\theoremname}{Theorem}

\makeatother

\usepackage{babel}
\addto\captionsbritish{\renewcommand{\definitionname}{Definition}}
\addto\captionsbritish{\renewcommand{\lemmaname}{Lemma}}
\addto\captionsbritish{\renewcommand{\propositionname}{Proposition}}
\addto\captionsbritish{\renewcommand{\theoremname}{Theorem}}
\addto\captionsenglish{\renewcommand{\definitionname}{Definition}}
\addto\captionsenglish{\renewcommand{\lemmaname}{Lemma}}
\addto\captionsenglish{\renewcommand{\propositionname}{Proposition}}
\addto\captionsenglish{\renewcommand{\theoremname}{Theorem}}
\providecommand{\definitionname}{Definition}
\providecommand{\lemmaname}{Lemma}
\providecommand{\propositionname}{Proposition}
\providecommand{\theoremname}{Theorem}

\DeclareMathOperator{\dive}{div}

\begin{document}
\title{\textbf{Gradient bounds for strongly singular or degenerate parabolic systems}}
\author{Pasquale Ambrosio, Fabian Bäuerlein}
\date{April 25, 2024}

\begin{abstract} We consider weak solutions $u:\Omega_{T}\rightarrow\mathbb{R}^{N}$
to parabolic systems of the type
\[
u_{t}-\mathrm{div}\,A(x,t,Du)=f \qquad \mathrm{in}\ \Omega_{T}=\Omega\times(0,T),
\]
where $\Omega$ is a bounded open subset of $\mathbb{R}^{n}$ for
$n\geq2$, $T>0$ and the datum $f$ belongs to a suitable Orlicz
space. The main novelty here is that the partial map $\xi\mapsto A(x,t,\xi)$
satisfies standard $p$-growth and ellipticity conditions for $p>1$
only outside the unit ball $\{\vert\xi\vert<1\}$. For $p>\frac{2n}{n+2}$
we establish that any weak solution
\[
u\in C^{0}((0,T);L^{2}(\Omega,\mathbb{R}^{N}))\cap L^{p}(0,T;W^{1,p}(\Omega,\mathbb{R}^{N}))
\]
admits a locally bounded spatial gradient $Du$. Moreover, assuming
that $u$ is essentially bounded, we recover the same result in the
case $1<p\leq\frac{2n}{n+2}$ and $f=0$. Finally, we also
prove the uniqueness of weak solutions to a Cauchy-Dirichlet problem
associated with the parabolic system above. We emphasize that our
results include both the degenerate case $p\geq2$ and the singular
case $1<p<2$.
\noindent \vspace{0.2cm}
\end{abstract}

\maketitle



\section{Introduction}

In this paper, we are interested in the
regularity of weak solutions $u:\Omega_{T}\rightarrow\mathbb{R}^{N}$
of strongly degenerate or singular parabolic systems of the type 
\begin{equation}
\partial_{t}u^{i}-\dive A^{i}(x,t,Du)=f^{i}\qquad \text{for }i\in\left\{ 1,\ldots,N\right\} \qquad \text{in }\Omega_{T}=\Omega\times(0,T).\label{eq:syst}
\end{equation}
Here and in the following $\Omega$ is a bounded and open subset of $\mathbb{R}^n$ ($n\geq 2$), $T>0$ and $N\geq 1$. The exact assumptions on the vector field $A=(A^1,\dots,A^N)\colon\Omega_T\times\mathbb R^{Nn}\to\mathbb R^{Nn}$ and the vector-valued inhomogeneity $f=(f^{1},\ldots,f^{N})$, $N \geq 1$, will be discussed in detail later. As our main result we will show that weak solutions are locally Lipschitz-continuous in the spatial directions, i.e.~that the spatial gradient $Du$ is locally bounded. 

The chosen terminology can be illustrated by the model system of equations. In fact, in the prototypical system we have in mind, the vector field $A$ is given by
\begin{equation}\label{A-proto}
    A(x,t,\xi) = \frac{( |\xi| -  1)_+^{p-1}}{|\xi|}\xi,
    \qquad(x,t)\in\Omega_T, \xi\in\mathbb R^{Nn},
\end{equation}
for some $p>1$.
First, we note that any time-independent 1-Lipschitz continuous function $u$ is a solution of the homogeneous model system. Accordingly, well established methods from regularity theory, using the second weak spatial derivatives of $u$, cannot be utilized. We will circumvent this difficulty by approximation. 

Before we specify in detail the structure of the considered system of equations, we briefly discuss some results already available in the literature. So far, most progress has been made for the associated elliptic, i.e. time-independent, problem. L. Brasco \cite{brasco2011global} proved the Lipschitz continuity of weak solutions. A. Clop, R. Giova, F. Hatami and A. Passarelli di Napoli \cite{CGHP} generalized this result for systems ($N \geq 2$). The two previous works arise from the study of \textit{strongly degenerate} functionals. These can be regarded as \textit{asymptotically convex} functionals, i.e. functionals having a $p$-Laplacian type structure only at infinity. Such class of functionals has been widely investigated, since the local Lipschitz regularity result
by Chipot and Evans \cite{ChiEva}. In particular, we mention generalizations allowing super- and
sub-quadratic growth \cite{GiaMod,LePaVe,Raymond}, as well as lower order terms \cite{PassVer}. Extensions to various other contexts can be found in the non-exhaustive list \cite{CGGHigherasym}–\cite{Cupini3}, \cite{Eleuteri2016}–\cite{Foss3}, \cite{Scheven}.
The first result concerning gradient continuity was achieved by F. Santambrogio and V. Vespri \cite{santambrogio2010continuity}. The authors restricted themselves to equations and the dimension $n=2$ and proved that a gradient dependent function of the form $g(Du)$ is continuous whenever $g$ is continuous and vanishes on the set of degeneracy, i.e. on the unit ball in case of the prototypical equation. Subsequently, M. Colombo and A. Figalli \cite{ColoFig,COLOMBO201494} extended this result to arbitrary dimensions $n\ge 2$. 
A similar result in the vectorial setting was recently shown by V. Bögelein, F. Duzaar, R. Giova and A. Passarelli di Napoli \cite{BDGP}. L. Mons \cite{mons2023higher} extended this result to systems satisfying more general structural conditions. On the contrary, for the time-dependent problem, little is known so far. The first author \cite{Ambrosio+2023} succeeded in showing a fractional higher differentiability result. Furthermore, A. Gentile and A. Passarelli di Napoli as well as the first author and A. Passarelli di Napoli obtained a higher differentiability result in \cite{GPHigherregularity} and \cite{APwidelydeg} respectively. 

In this paper we prove the boundedness of the spatial gradient under quite general assumptions on the vector field $A$. In particular, this can be seen as a parabolic analogue of Brasco's Lipschitz-continuity result. Our main result reads as follows. For notation and definitions we refer to Section \ref{sec:prelim}.

\begin{thm}\label{thm:theo1} 
Let $n\geq2$, $N\ge1$, $1<p<\infty$ and $f\in L^{\hat{n}+2}\log^{\alpha}L_{loc}(\Omega_{T},\mathbb{R}^{N})$,
where $\hat{n}$ is defined according to $(\ref{eq:nhat})-(\ref{eq:betacondition})$ below and $\alpha>2 \hat{n}+3$.
Moreover, assume that 
\[
u\in C^{0}\left((0,T);L^{2}\left(\Omega,\mathbb{R}^{N}\right)\right)\cap L^{p}\left(0,T;W^{1,p}\left(\Omega,\mathbb{R}^{N}\right)\right)
\]
is a weak solution to $(\ref{eq:syst})$, where $A$ is defined by
$(\ref{eq:strucfun})$ and the structure conditions $(\mathrm{H_{1}})-(\mathrm{H_{4}})$ below
are in force. Then, for any parabolic cylinder $Q_{r}(z_{1})\Subset Q_{R}(z_{0})\Subset\Omega_{T}$ with $r\in(0,1)$ and any $s\in(0,1)$, we have that:\\
\\
$\mathrm{(}a\mathrm{)}$ if $p>\frac{2n}{n+2}$
, the estimate 
\[
\underset{Q_{sr}(z_{1})}{\mathrm{ess} \sup}\,\vert Du\vert \leq \frac{c}{(1-s)^{\frac{\hat{n}+2}{\kappa}}} \cdot \exp\left(c\, \frac{\Vert f\Vert_{L^{\hat{n}+2}\log^{\alpha}L(Q_{r}(z_{1}))}^{\Theta}}{r^{\frac{n-\hat{n}}{\alpha}}}\right)\cdot\left[\fint_{Q_{r}(z_{1})}(\vert Du\vert +1)^{p}\,dz\right]^{\frac{1}{\kappa}}
\]
holds true for some positive constants $c\equiv c(N,n,\hat{n},\alpha,p,C_{1},K)$
and $\Theta\equiv\Theta(\hat{n},\alpha)$, and for 
\[
\kappa:=\begin{cases}
\frac{p (\hat{n}+2)-2 \hat{n}}{2} & \text{if }\frac{2n}{n + 2}<p<2,\\
2 & \text{if } p\geq2;
\end{cases}
\]
$\mathrm{(}b\mathrm{)}$ if $1<p\leq \frac{2n}{n+2}$, $f=0$ and $u\in L^\infty(Q_R(z_0),\mathbb{R}^{N})$, the estimate
\begin{equation*}
    \underset{Q_{sr}(z_{1})}{\mathrm{ess} \sup} \, \vert Du\vert 
    \leq c \,\frac{\left(1+\Vert u\Vert_{L^{\infty}(Q_{R}(z_{0}))}^{2}\right)^{\frac{n+2}{2 p}}}{[(1-s)^{2} \,r]^{\frac{n +2}{p}}}
    \left[\fint_{Q_{r}(z_{1})}(\vert Du\vert +1)^{p} dz\right]^{\frac{1}{p}}
\end{equation*}
holds true for some positive constant $c\equiv c(N,n,\alpha,p,C_{1},K)$.
\end{thm}

In the subcritical case $1<p\leq \frac{2n}{n+2}$ , the weak solutions to $(\ref{eq:syst})$ may not be bounded (see \cite[Sub-chapter 5.4]{DiBene} and note that the $p$-Laplacian satisfies our growth assumptions). Therefore, the extra assumption $u\in L^\infty(Q_R(z_0),\mathbb{R}^{N})$ in Theorem~\ref{thm:theo1} (\textit{b}) is natural. 

Note that some widely degenerate systems can be interpreted, similarly to functionals, as asymptotically regular systems. Therefore, in the special case that $A(x,t,Du)\equiv A(Du)\approx |Du|^{p-2}Du$ for $|Du|\gg 1$ and $p>\frac{2n}{n+2}$ we recover the results obtained in \cite{Kuusi-Mingione, Boegelein}. 

Finally, we explain the choice of the parameter $\hat n$, which is defined by 
\begin{equation}
\hat{n}:=\begin{cases}
\begin{array}{cc}
n & \mathrm{if}\,\, n>2\\
2+\beta & \mathrm{if}\,\, n=2
\end{array}\end{cases}\label{eq:nhat}
\end{equation}
for some $\beta>0$. If $n=2$ and $1<p<2$ we choose $\beta$ in such a way that 
\begin{equation}
    0<\beta < \frac{4(p-1)}{2-p}.\label{eq:betacondition}
\end{equation}
This choice ensures that in the case $n=2$ we have 
$p>\frac{2\Hat{n}}{\Hat{n}+2}$ for any $p>1$. Indeed, this condition will be decisive to carry out the proof in Subsection \ref{subsec:iteration1}.

\subsection{Structural conditions on the vector field \textit{A}}
To specify the structure of the vector field $A:\Omega_{T}\times\mathbb{R}^{Nn}\rightarrow\mathbb{R}^{Nn}$,
we consider a function 
\[
F\colon\Omega_{T}\times[0,\infty)\rightarrow[0,\infty),\qquad        (x,t,s)\mapsto F(x,t,s),
\]
which is convex with respect to the $s$-variable and vanishes for
all $s\in[0,1]$. Furthermore, we shall assume that the partial map
$s\mapsto F(x,t,s)$ is in $C^{1}(\mathbb{R}^{+})\cap C^{2}(\mathbb{R}^{+}\setminus\{1\})$
for almost every $(x,t)\in\Omega_{T}$, while for every $(t,s)\in(0,T)\times[0,\infty)$
the map $x\mapsto F(x,t,s)$ is differentiable almost everywhere.
We additionally suppose that there exist an exponent $p>1$ and some
positive constants $L$, $C_{1}>1$ and $K$ such that, for all $s>1$
and for almost every $x$, $y\in\Omega$ and $t\in(0,T)$, the function
$F$ satisfies the following growth assumptions: 
\[
\begin{cases}
\begin{array}{cc}
\frac{1}{L} (s-1)^{p} \leq F(x,t,s) \leq L s^{p} & (\mathrm{H_{1}})\\
\frac{1}{C_{1}} (s-1)^{p-1} \leq \partial_{s}F(x,t,s) \leq C_{1} (s-1)^{p-1} & (\mathrm{H_{2}})\\
\frac{1}{C_{1}} (s-1)^{p-2} \leq \partial_{ss}F(x,t,s) \leq C_{1} (s-1)^{p-2} & (\mathrm{H_{3}})\\
\left|\partial_{s}F(x,t,s)-\partial_{s}F(y,t,s)\right| \leq K\left|x-y\right| s^{p-1} & (\mathrm{H_{4}})
\end{array} & \end{cases}
\]
Then, the vector field  $A:\Omega_{T}\times\mathbb{R}^{Nn}\rightarrow\mathbb{R}^{Nn}$
is defined by 
\begin{equation}
A(x,t,\xi):=\begin{cases}
\frac{\partial_{s}F(x, t, \left|\xi\right|)}{\left|\xi\right|}\, \xi & \mathrm{if }\,\, \xi\in\mathbb{R}^{Nn}\setminus\{0\},\\[5pt]
0 & \mathrm{if }\,\,\xi=0.
\end{cases}\label{eq:strucfun}
\end{equation}
Note that for the prototypical case 
\[
F(x,t,s)=\frac{a(x,t)}{p} (s-1)_{+}^{p},
\]
where $a:\Omega_{T}\rightarrow\mathbb{R}^{+}$ and $x\mapsto a(x,t)$
is a Lipschitz continuous function that is bounded from below by a
positive constant, one can easily deduce that
the growth conditions ($\mathrm{H_{1}}$)$-$($\mathrm{H_{4}}$) are
fulfilled. If furthermore $a(x,t)\equiv 1$, then we recover the vector field in~\eqref{A-proto}.

\subsection{Strategy of the proof}

Our approach is inspired by \cite{CGHP, CGGHigherasym, Eleuteri2016}. The proof will be achieved by a parabolic Moser iteration technique. However, the implementation is quite subtle due to the degeneracy of the differential operator. 
Since weak solutions may not be twice weakly differentiable with respect to the $x$-variable, we approximate the original system (\ref{eq:syst}). The approximation is chosen in such a way that the regularized vector field $A_\varepsilon$, $\varepsilon>0$, satisfies standard $p$-growth assumptions.
However, proceeding at this point with a standard Moser iteration, the constants would blow up as $\varepsilon\downarrow 0$.
This problem will be overcome by the choice of an appropriate test function in the derivation of the Caccioppoli-type inequality. At this point it is helpful to realize that for $\delta>0$ large enough and $|Du_\varepsilon|>1+\delta$ the $p$-growth conditions of the approximating systems are satisfied independently of $\varepsilon$.
Therefore, we choose a test function vanishing on the set $\{|Du_\varepsilon|\leq 1+\delta\}$. 
We note that the choice of such a test function requires the existence of second weak spatial derivatives, which is the reason why we had to introduce the approximating solutions. 
This test function allows to obtain a Caccioppoli-type inequality and in turn we set up a kind of Moser iteration scheme. 
Note that, similarly to the treatment of the parabolic $p$-Laplacian, in the Moser iteration we have to distinguish between three different regimes of the parameter $p$. The first one is the superquadratic case $p\geq 2$, the second one is the subquadratic supercritical case $\frac{2n}{n+2}< p < 2$ and finally we have the subcritical case $1<p\leq \frac{2n}{n+2}$. In the last case, we need to ensure that also the approximating solutions $u_\varepsilon$ are uniformly bounded with respect to $\varepsilon$. This is achieved by a maximum principle.
In all three previous cases, we prove the boundedness of the spatial gradient $Du_\epsilon$ of the approximating solutions together with quantitative estimates. Here it is worthwhile to note that our quantitative estimates are uniform in $\varepsilon$. Since the approximating solutions converge strongly in $L^p$ to the original solution, the statement of Theorem~\ref{thm:theo1} follows after passing to the limit as $\varepsilon\downarrow 0$. 

\subsection{Plan of the paper}

In Subsection \ref{subsec:notation},  we introduce the notation adopted throughout the paper. In Subsections \ref{subsec:Orlicz-spaces} and \ref{subsec:Steklov}, we recall some basic facts on the used function spaces and the regularization in time. In the following Subsections \ref{subsec:approx} and \ref{subsec:approx2}, the approximating vector field is defined and its growth properties are obtained. Thereby, we distinguish the subquadratic and superquadratic cases.
Subsection \ref{subsec:Algebraic} is devoted to some algebraic inequalities needed for our purposes.

In Section \ref{sec:uniqueness}, we define the approximating problems and prove a strong convergence result in $L^p$. As a by-product, we obtain the uniqueness of the weak solutions to a Cauchy-Dirichlet problem associated with system (\ref{eq:syst}).

In Section \ref{sec:comparison}, we establish a maximum principle for solutions to the approximating problem. This allows us to show that the approximating solutions are essentially bounded, provided that the original solution is itself essentially bounded.

In Section \ref{sec:weak diff} we lay the groundwork for Section \ref{sec:Local}, where we prove local $L^{\infty}$-bounds for the spatial gradient of the approximating solutions. The proof is divided into several subsections. In Subsection \ref{subsec:Caccioppoli} we establish a suitable Caccioppoli-type inequality, from which we derive a reverse Hölder-type inequality in Subsection \ref{subsec:reverse-Hölder}. In these subsections, we need to address separately the three regimes of the parameter $p$ that we referred to above. For this reason, the Moser iteration procedure is split into two steps: in Subsection \ref{subsec:iteration1} we deal with the supercritical case, while in Subsection \ref{subsec:iteration2} we consider the subcritical one.

Finally, in Section \ref{sec:Final} we give the proof of Theorem \ref{thm:theo1}.
\medskip

\textbf{Acknowledgements.} Part of this
paper was written while P. Ambrosio was visiting the University of
Salzburg in November and December 2022. He is thankful to Verena Bögelein
and the host institution for their kind invitation and constant support.
The authors gratefully acknowledge fruitful discussions with Verena
Bögelein and Antonia Passarelli di Napoli, who provided valuable comments
and suggestions during the preparation of the manuscript. This work
has been partially supported by the FWF-Project P36295-N. P. Ambrosio
is a member of the GNAMPA group of INdAM that partially supported
his research through the INdAM$-$GNAMPA 2024 Project ``Fenomeno di Lavrentiev, Bounded Slope Condition e regolarità
per minimi di funzionali integrali con crescite non standard e lagrangiane non uniformemente
convesse'' (CUP: E53C23001670001).

\section{Preliminaries\label{sec:prelim}}

\subsection{Notation and essential tools}\label{subsec:notation}

In this paper we shall denote by $C$ or
$c$ a general positive constant that may vary on different occasions,
even within the same line of estimates. Relevant dependencies on parameters
and special constants will be suitably emphasized using parentheses
or subscripts. The norm we use on the Euclidean spaces $\mathbb{R}^{k}$
will be the standard Euclidean one and it will be denoted by $\left| \cdot \right|$.
In particular, for the vectors $\xi,\eta\in\mathbb{R}^{Nn}$, we write
$\langle\xi,\eta\rangle$ for the usual inner product and $\left|\xi\right|:=\langle\xi,\xi\rangle^{\frac{1}{2}}$
for the corresponding Euclidean norm.

For points in space-time, we will frequently use
abbreviations like $z=(x,t)$ or $z_{0}=(x_{0},t_{0})$, for spatial
variables $x$, $x_{0}\in\mathbb{R}^{n}$ and times $t$, $t_{0}\in\mathbb{R}$.
We also denote by $B_{\varrho}(x_{0})=\left\{ x\in\mathbb{R}^{n}:\left|x-x_{0}\right|<\varrho\right\} $
the $n$-dimensional open ball with radius $\varrho>0$ and center
$x_{0}\in\mathbb{R}^{n}$; when not important, or clear from the context,
we shall omit to denote the center as follows: $B_{\varrho}\equiv B_{\varrho}(x_{0})$.
Unless otherwise stated, different balls in the same context will
have the same center. Moreover, we use the notation 
\[
Q_{\varrho}(z_{0}):=B_{\varrho}(x_{0})\times(t_{0}-\varrho^{2},t_{0}),\qquad z_{0}=(x_{0},t_{0})\in\mathbb{R}^{n}\times\mathbb{R}, \varrho>0,
\]
for the backward parabolic cylinder with vertex $(x_{0},t_{0})$ and
width $\varrho$. We shall sometimes omit the dependence on the vertex
when all cylinders occurring in a proof share the same vertex. For
a general cylinder $Q=B\times(t_{0},t_{1})$, where $B\subset\mathbb{R}^{n}$
and $t_{0}<t_{1}$, we denote by 
\[
\partial_{\mathrm{par}}Q:=(\overline{B}\times\{t_{0}\})\cup(\partial B\times(t_{0},t_{1}))
\]
the \textit{parabolic boundary} of $Q$.

If $E\subseteq\mathbb{R}^{k}$ is a Lebesgue-measurable
set, then we will denote by $\vert E\vert$ its $k$-dimensional Lebesgue
measure. When $0<\vert E\vert<\infty$, the mean
value of a function $v\in L^{1}(E)$ is defined by 
\[
(v)_{E}:=\fint_{E}v(y) dy = \frac{1}{\vert E\vert}\int_{E}v(y) dy.
\]
For every $\upsilon\in L_{loc}^{1}(Q,\mathbb{R}^{k})$,
where $Q\subset\mathbb{R}^{n+1}$ and $k\in\mathbb{N}$, we define
the mollified function $\upsilon_{\varrho}$ as follows:
\begin{equation}
\upsilon_{\varrho}(z):=\int_{\mathbb{R}^{n+1}}\upsilon(\tilde{z}) \eta_{\varrho}(z-\tilde{z}) d\tilde{z},\label{eq:molli}
\end{equation}
where 
\[
\eta_{\varrho}(z):= \frac{1}{\varrho^{n+1}}\,  \eta_{1}\left(\frac{z}{\varrho}\right),\qquad \varrho>0,
\]
with $\eta_{1}\in C_{0}^{\infty}(\mathcal{B}_{1}(0))$\footnote{Here $\mathcal{B}_{1}(0)$ denotes the $(n+1)$-dimensional open unit
ball centered at the origin.} denoting the standard, non-negative, radially symmetric mollifier
in $\mathbb{R}^{n+1}$. Obviously, here the function $\upsilon$ is
meant to be extended by the $k$-dimensional null vector
outside $Q$. 

In this work, we define a weak solution
to (\ref{eq:syst}) as follows: 
\begin{defn}
A function $u\in C^{0}\left((0,T);L^{2}\left(\Omega,\mathbb{R}^{N}\right)\right)\cap L^{p}\left(0,T;W^{1,p}\left(\Omega,\mathbb{R}^{N}\right)\right)$
is a \textit{weak solution} of system (\ref{eq:syst}) if and only
if for any test function $\phi\in C_{0}^{\infty}(\Omega_{T},\mathbb{R}^{N})$
the following integral identity holds: 
\begin{equation}
\int_{\Omega_{T}}\left(u\cdot\partial_{t}\phi-\langle A(x,t,Du),D\phi\rangle\right) dz = -\int_{\Omega_{T}}f\cdot\phi \,dz.\label{eq:weaksol}
\end{equation}
\end{defn}

We conclude this first part of the preliminaries
by recalling the following iteration lemma, which is a standard tool
for ``reabsorbing'' certain terms and can be found, for example,
in \cite[Lemma 6.1]{Giusti}. 
\begin{lem}\label{lem:Giusti}
Let $0\leq\rho_{0}<\rho_{1}<\infty$
and assume that $\Psi:[\rho_{0},\rho_{1}]\rightarrow[0,\infty)$ is
a bounded function satisfying 
\[
\Psi(\rho) \leq \vartheta \Psi(r)+\frac{C}{(r-\rho)^{\sigma}}
\]
for all $\rho_{0}\leq\rho<r\leq\rho_{1}$, for some $\sigma>0$, $\vartheta\in[0,1)$
and a non-negative constant $C$. Then, there exists a constant $\kappa\equiv\kappa(\sigma,\vartheta)>0$
such that for all $\rho_{0}\leq\rho<r\leq\rho_{1}$ we have 
\[
\Psi(\rho) \leq \kappa  \frac{C}{(r-\rho)^{\sigma}} .
\]
\end{lem}

\subsection{Orlicz spaces}\label{subsec:Orlicz-spaces}

Here we recall some basic properties of
Orlicz spaces that will be needed later on (for more details, we refer
to \cite{Ada}). Let $\Psi:[0,\infty)\rightarrow[0,\infty)$ be a
Young function, i.e. $\Psi(0)=0,$ $\Psi$ is increasing and convex.
If $\varSigma$ is an open subset of $\mathbb{R}^{k}$, we define
the Orlicz space $L^{\Psi}(\varSigma)$ generated by the Young function
$\Psi$ as the set of the measurable functions $v:\varSigma\rightarrow\mathbb{R}$
such that 
\[
\int_{\varSigma}\Psi\left(\frac{|v|}{\lambda}\right) dx<\infty
\]
for some $\lambda>0$. When equipped with the Luxemburg norm 
\[
\Vert v\Vert_{L^{\Psi}(\varSigma)}:=\inf\left\{ \lambda>0:\int_{\varSigma}\Psi\left(\frac{|v|}{\lambda}\right) dx\leq1\right\} ,
\]
$L^{\Psi}(\varSigma)$ is a Banach space.

The Zygmund space $L^{q}\log^{\alpha}L(\varSigma)$,
for $1\leq q<\infty$, $\alpha\in\mathbb{R}$ ($\alpha\geq0$ for
$q=1$), is defined as the Orlicz space $L^{\Psi}(\varSigma)$ generated
by the Young function $\Psi(s)\simeq s^{q}\log^{\alpha}(e+s)$ for
every $s\geq s_{0}\geq0$. Therefore, a measurable function $v$ on
$\varSigma$ belongs to $L^{q}\log^{\alpha}L(\varSigma)$ if 
\[
\int_{\varSigma}|v|^{q}\log^{\alpha}(e+|v|) \,dx<\infty.
\]
Moreover, we record that for the function 
\[
\Psi(s)=s^{q}\log^{\alpha}(e+s),\qquad q>1, \alpha\in\mathbb{R},
\]
the following inequality 
\[
\int_{\varSigma}|v|^{q}\log^{\alpha}(e+|v|) \,dx \leq \Vert v\Vert_{L^{\Psi}(\varSigma)}^{\theta}
\]
holds for every $v\in L^{q}\log^{\alpha}L(\varSigma)$ and some $\theta=\theta(q)>0$
(see \cite{CGHP}).

\subsection{Steklov averages}\label{subsec:Steklov}

In this section, we recall the definition
and some elementary properties of Steklov averages. Let us denote
a domain in space time by $Q':=\Omega'\times I$, where $\Omega'\subseteq\Omega$
is a bounded domain and $I:=(t_{1},t_{2})\subseteq(0,T)$. For every
$h\in(0,t_{2}-t_{1})$ and $v\in L^{1}(\Omega'\times I,\mathbb{R}^{k})$,
where $k\in\mathbb{N}$, the \textit{Steklov average} $[v]_{h}(\cdot,t)$
is defined by 
\[
[v]_{h}(x,t):=\begin{cases}
\begin{array}{cc}
\frac{1}{h}\int_{t}^{t+h}v(x,s) ds &   \text{for }   t\in(t_{1},t_{2}-h),\\
0 & \text{for }   t\in[t_{2}-h,t_{2}),
\end{array}\end{cases}
\]
for $x\in\Omega'$. This definition implies, for almost every $(x,t)\in\Omega'\times(t_{1},t_{2}-h)$,
\[
\frac{\partial[v]_{h}}{\partial t}(x,t)= \frac{v(x,t+h)-v(x,t)}{h}.
\]

The proof of the following result is straightforward
from the theory of Lebesgue spaces (see \cite[Lemma 3.2]{DiBene}).\medskip{}

\begin{lem}\label{lem:Stek}
Let $q,r\geq1$ and $v\in L^{r}\left(t_{1},t_{2};L^{q}(\Omega')\right)$.
Then, as $h\rightarrow0$, $[v]_{h}$ converges to $v$ in $L^{r}\left(t_{1},t_{2}-\epsilon;L^{q}(\Omega')\right)$
for every $\epsilon\in(0,t_{2}-t_{1})$. If $v\in C^{0}\left((t_{1},t_{2});L^{q}(\Omega')\right)$,
then as $h\rightarrow0$, $[v]_{h}(\cdot,t)$ converges to $v(\cdot,t)$
in $L^{q}(\Omega')$ for every $t\in(t_{1},t_{2}-\epsilon)$, $\forall \epsilon\in(0,t_{2}-t_{1})$. 
\end{lem}

For further needs, we now recall the well-known
Steklov average formulation of (\ref{eq:syst}) in $Q'=\Omega'\times(t_{1},t_{2})$.
Assume that $u\in C^{0}\left((0,T);L^{2}\left(\Omega,\mathbb{R}^{N}\right)\right)\cap L^{p}\left(0,T;W^{1,p}\left(\Omega,\mathbb{R}^{N}\right)\right)$
is a weak solution to (\ref{eq:syst}) in $\Omega_{T}$ and let $0<h<t_{2}-t_{1}.$
Then, the Steklov average $[u]_{h}(\cdot,t)$ satisfies, for all times
$t\in(t_{1},t_{2}-h)$, 
\[
\int_{\Omega'\times\{t\}}\left(\partial_{t}[u]_{h}\cdot\phi +\langle[A(x,t,Du)]_{h},D\phi\rangle\right) dx = \int_{\Omega'\times\{t\}}[f]_{h}\cdot\phi \,dx
\]
for all $\phi\in C_{0}^{\infty}(\Omega',\mathbb{R}^{N})$. 

\subsection{Approximation of the function \texorpdfstring{$\boldsymbol{F}$}{F} for \texorpdfstring{$\boldsymbol{1<p\leq2}$}{1<p<2}}\label{subsec:approx}

As already mentioned, we assume that
for almost every $(x,t)\in\Omega_{T}$ the map $s\mapsto F(x,t,s)$
is in $C^{1}(\mathbb{R}^{+})\cap C^{2}(\mathbb{R}^{+}\setminus\{1\})$.
Moreover, we notice that in the case $1<p< 2$ both bounds for $\partial_{ss}F(x,t,\cdot)$
in $(\mathrm{H_{3}})$ blow up as $s\rightarrow1^{+}$. This very
singular behavior of $F$ must be avoided, since we need to use the
second derivative $\partial_{ss}F$ to establish a local bound for
the spatial gradient of the weak solutions to (\ref{eq:syst}). Therefore,
for $1<p<2$ and for almost every $(x,t)\in\Omega_{T}$ we approximate
the partial map $s\mapsto F(x,t,s)$ by smoothing it around the singularity
of $\partial_{ss}F(x,t,\cdot)$, in such a way that the resulting
approximation $F_{\varepsilon}(x,t,\cdot)$ coincides with $F(x,t,\cdot)$
outside a small neighborhood of the singularity $s=1$.

Thus, in this section we will assume that $1<p\leq 2$,
unless otherwise stated. For $\varepsilon\in\left(0,\tfrac{1}{2}\right)$
we define the function $v_{\varepsilon}:\mathbb{R}_{0}^{+}\rightarrow\mathbb{R}^{+}$
by 
\begin{equation}
v_{\varepsilon}(s):=\frac{1}{\varepsilon}\int_{\mathbb{R}}\eta_{1}\Big(\frac{w-s}{\varepsilon}\Big)\max \{\varepsilon,w-1\} \,dw,\label{eq:app1}
\end{equation}
where $\eta_{1}\in C_{0}^{\infty}((-1,1))$ denotes the standard,
non-negative, radially symmetric mollifier in $\mathbb{R}$. Keeping
this definition in mind, in what follows we will show that an approximation
of $F(x,t,\cdot)$ is given by 
\[
\tilde{F}_{\varepsilon}(x,t,s):=F(x,t,v_{\varepsilon}(s)+1),\qquad \varepsilon \in \left(0, \tfrac{1}{2}\right).
\]

Firstly, we need to prove that $\tilde{F}_{\varepsilon}$
satisfies non-degenerate growth conditions for $0<\varepsilon<\frac{1}{2}$.
Therefore, we begin our analysis by studying the growth of this function.\\
 By assumption, we know that for almost every $(x,t)\in\Omega_{T}$
the map $s\mapsto\tilde{F}_{\varepsilon}(x,t,s)$ belongs to $C^{2}(\mathbb{R}^{+})$.
For later purposes, we now note that one can easily check that 
\begin{equation}
v_{\varepsilon}(s)=\begin{cases}
\varepsilon & \text{if } 0\leq s\leq1\\
s-1 &    \text{if } s\geq1+2\varepsilon
\end{cases}\label{eq:appro1}
\end{equation}
and 
\begin{equation}
0\leq v'_{\varepsilon}(s)\leq\mathds{1}_{\{ \, \cdot\,> 1\}}(s) ,\qquad 0 \leq v''_{\varepsilon}(s) \leq \tfrac{C}{\varepsilon} \, \mathds{1}_{\{1 < \, \cdot \, < 1+2 \varepsilon\}}(s)\label{eq:app2}
\end{equation}
for all $s\geq0$ and for some universal constant $C>0$. Furthermore,
from the growth assumption ($\mathrm{H_{1}}$) and from definition
(\ref{eq:app1}) we can deduce that 
\[
\frac{1}{c L} (s^{p}-1) \leq \tilde{F}_{\varepsilon}(x,t,s) \leq c L (s^{p}+1)
\]
for all $s\geq1+2 \varepsilon$ and for almost every $(x,t)\in\Omega_{T}$.
As for the derivatives of $\tilde{F}_{\varepsilon}$ with respect
to the $s$-variable, for almost every $(x,t)\in\Omega_{T}$ we have
\begin{equation}
\partial_{s}\tilde{F}_{\varepsilon}(x,t,s) = \partial_{s}F(x,t,v_{\varepsilon}(s)+1)  v'_{\varepsilon}(s)\label{eq:der1}
\end{equation}
and 
\begin{equation}
\partial_{ss}\tilde{F}_{\varepsilon}(x,t,s) = \partial_{ss}F(x,t,v_{\varepsilon}(s)+1) (v'_{\varepsilon}(s))^{2}+\partial_{s}F(x,t,v_{\varepsilon}(s)+1)  v''_{\varepsilon}(s),\label{eq:der2}
\end{equation}
and from assumptions ($\mathrm{H}_{2}$) and ($\mathrm{H}_{3}$)
it follows that $\partial_{s}\tilde{F}_{\varepsilon}$, $\partial_{ss}\tilde{F}_{\varepsilon}\geq0$.
Moreover, combining (\ref{eq:app2}), (\ref{eq:der1}), ($\mathrm{H}_{2}$)
and the fact that $v_{\varepsilon}(s)\leq\max \{2 \varepsilon,s-1\}\leq s$
for every $s>1$ and every $\varepsilon\in\left(0,\frac{1}{2}\right)$,
for almost every $(x,t)\in\Omega_{T}$ we obtain 
\begin{equation*}
\partial_{s}\tilde{F}_{\varepsilon}(x,t,s) \leq C_{1} (v_{\varepsilon}(s))^{p-1} \mathds{1}_{\{ \cdot > 1\}}(s) \leq C_{1} s^{p-1} \mathds{1}_{\{ \cdot > 1\}}(s) \label{eq:DsF_eps}
\end{equation*}
for any $s>0$ and any $\varepsilon\in\left(0,\frac{1}{2}\right)$.
As for the second derivative $\partial_{ss}\tilde{F}_{\varepsilon}$,
combining (\ref{eq:app2}), (\ref{eq:der2}), ($\mathrm{H}_{2}$)
and ($\mathrm{H}_{3}$), for every $s>0$ and for almost every $(x,t)\in\Omega_{T}$
we find 
\begin{equation}
\partial_{ss}\tilde{F}_{\varepsilon}(x,t,s) \leq C_{1} (v_{\varepsilon}(s))^{p-2} \mathds{1}_{\{ \cdot > 1\}}(s)+\frac{C_{1} C}{\varepsilon} (v_{\varepsilon}(s))^{p-1} \mathds{1}_{\{1 < \,  \cdot \,  < 1+2 \varepsilon\}}(s).\label{eq:app3}
\end{equation}
Now, using the fact $v_{\varepsilon}(s)\leq2 \varepsilon$ for $s<1+2 \varepsilon$,
we can estimate the second term on the right-hand side of (\ref{eq:app3})
as follows: 
\begin{align}
\frac{(v_{\varepsilon}(s))^{p-1}}{\varepsilon}  \mathds{1}_{\{1 < \,  \cdot \, < 1+2 \varepsilon\}}(s)  & \leq 2^{p-1}\varepsilon^{p-2} \frac{(1+s^{2})^{\frac{p-2}{2}}}{(1+s^{2})^{\frac{p-2}{2}}}  \mathds{1}_{\{1 < \, \cdot \, < 1+2 \varepsilon\}}(s) \nonumber\\
 & \leq 2^{p-1} 5^{\frac{2-p}{2}} \varepsilon^{p-2} (1+s^{2})^{\frac{p-2}{2}}.
\label{eq:app4}
\end{align}
In order to deal with the first term on the right-hand side of (\ref{eq:app3}),
we distinguish between the cases $1<s<1+2 \varepsilon$ and $s\geq1+2 \varepsilon$.
In the first case, we have $v_{\varepsilon}(s)\geq\varepsilon$ and
therefore we get 
\begin{equation}
(v_{\varepsilon}(s))^{p-2} \mathds{1}_{\{\, \cdot\,> 1\}}(s) \leq \varepsilon^{p-2} \leq 5^{\frac{2-p}{2}} \varepsilon^{p-2} (1+s^{2})^{\frac{p-2}{2}}.\label{eq:app5}
\end{equation}
If, on the other hand, $s\geq1+2 \varepsilon$, then we have $v_{\varepsilon}(s)=s-1$
due to (\ref{eq:appro1}). In this case, using that $(s-1)^{2}\geq\frac{\varepsilon^{2}}{4} (1+s^{2})$
we obtain 
\begin{equation}
(v_{\varepsilon}(s))^{p-2} \mathds{1}_{\{\,\cdot\,> 1\}}(s) = (s-1)^{p-2}\leq 2^{2-p} \varepsilon^{p-2} (1+s^{2})^{\frac{p-2}{2}}\leq 5^{\frac{2-p}{2}} \varepsilon^{p-2} (1+s^{2})^{\frac{p-2}{2}}.\label{eq:app6}
\end{equation}
Joining estimates (\ref{eq:app3})$-$(\ref{eq:app6}), for every
$s>0$ and for almost every $(x,t)\in\Omega_{T}$ we then have 
\begin{equation*}
\partial_{ss}\tilde{F}_{\varepsilon}(x,t,s) \leq c\, \varepsilon^{p-2} (1+s^{2})^{\frac{p-2}{2}},\label{eq:C3estimate}
\end{equation*}
where $c\equiv c(p,C_{1})>0$. This concludes the necessary growth
estimates of $\tilde{F}_{\varepsilon}$.

Now, in order to prove that the function $\tilde{F}_{\varepsilon}$
is indeed a good approximation of $F$, it remains to analyze the
behavior of $\tilde{F}_{\varepsilon}$ as $\varepsilon\searrow0$.
To this end, we notice that (\ref{eq:appro1}) immediately implies
\begin{equation}
\partial_{s}\tilde{F}_{\varepsilon}(x,t,s)\equiv\partial_{s}F(x,t,s)\qquad \text{for }  s\notin(1,1+2 \varepsilon).\label{eq:equaDsFeps}
\end{equation}
Furthermore, for $s\in[1,1+2 \varepsilon]$ we can estimate the difference
of these two derivatives as follows: 
\begin{align}\label{eq:lab1}
\vert\partial_{s}\tilde{F}_{\varepsilon}(x,t,s)-\partial_{s}F(x,t,s)\vert
&\leq \vert\partial_{s}\tilde{F}_{\varepsilon}(x,t,s)\vert+\vert\partial_{s}F(x,t,s)\vert\nonumber\\
&\leq 2^{p} C_{1} \varepsilon^{p-1}
\leq 2^{p} C_{1} \varepsilon^{p-1} s^{p-1}.
\end{align}
Let us explicitly observe that the last term tends to zero as $\varepsilon\searrow0$.
To ensure the convergence result of Lemma \ref{lem:Lpconv} below, we need to accelerate the rate of convergence. For this reason, for any $1<p<2$
we now define the new approximation 
\begin{equation}
F_{\varepsilon}(x,t,s):=\tilde{F}_{\varepsilon^{\frac{1}{p-1}}}(x,t,s),\qquad s\geq 0 .\label{eq:lab2}
\end{equation}
Collecting the above conclusions, we note that $F_\varepsilon$ has the following properties:
\begin{lem}\label{eq:approxin}
    For every $\varepsilon\in\left(0,2^{1-p}\right)$, almost every $(x,t)\in \Omega_T$ and every $s\in \mathbb{R}_{0}^{+}$, we have
    \begin{equation*}
    \begin{cases}
        \frac{1}{c L} (s^{p}-1) \leq F_{\varepsilon}(x,t,s) \leq c L (s^{p}+1), \\
        0\leq \partial_{s} F_{\varepsilon}(x,t,s) \leq C_{1} s^{p-1} \mathds{1}_{\{ \cdot > 1\}}(s),\\
        0\leq \partial_{ss} F_{\varepsilon}(x,t,s) \leq c \,\varepsilon^{\frac{p-2}{p-1}} (1+s^{2})^{\frac{p-2}{2}}, \\
        \vert\partial_{s} F_{\varepsilon}(x,t,s)-\partial_{s}F(x,t,s)\vert
        \leq 
        2^{p} \, C_{1} \,\varepsilon \, s^{p-1}, \\
        \partial_{s} F_{\varepsilon}(x,t,s)\equiv\partial_{s}F(x,t,s)\qquad \mathrm{if \ } s\notin \big(1,1+2 \varepsilon^{\frac{1}{p-1}} \big).
    \end{cases}
\end{equation*}
\end{lem}
However, notice that $2^{1-p}\geq\frac{1}{2}$ whenever $1<p\leq2$, which implies that the previous lemma holds for any $\varepsilon\in(0,1/2)$.
\subsection{Approximation of the vector field \textit{A}}\label{subsec:approx2}
With the approximation (\ref{eq:lab2}) of $F$ and Lemma \ref{eq:approxin} in mind, we can define for every $\varepsilon\in(0,1/2)$ the vector field $A_{\varepsilon}:\Omega_{T}\times\mathbb{R}^{Nn}\rightarrow\mathbb{R}^{Nn}$
by 
\[
A_{\varepsilon}(x,t,\xi):= h_{\varepsilon}(x,t,\vert\xi\vert) \xi ,
\]
where
\begin{equation}
h_{\varepsilon}(x,t,s):=\begin{cases}
\frac{\partial_{s}F_{\varepsilon}(x, t, s)}{s}+\varepsilon (1+s^{2})^{\frac{p-2}{2}} & \text{if } 1<p\leq2,\\
\frac{\partial_{s}F(x, t, s)}{s}+\varepsilon (1+s^{2})^{\frac{p-2}{2}} & \text{if }  p>2.\end{cases}\label{eq:h-fun-bis}
\end{equation}
We thus approximate the structure function $A$ by means of the vector
fields $A_{\varepsilon}$, in order to be allowed to apply some results
from the theory of singular or degenerate parabolic systems to the
weak solutions of problem (\ref{eq:CAUCHYDIR}), introduced in Section \ref{sec:uniqueness}. Therefore,
we now need to check whether $A_{\varepsilon}$ fulfills non-degenerate
growth conditions. This is what we will do hereafter.

From the growth conditions of $\partial_{s}F_{\varepsilon}$
and from the structure of the approximation, for any $p>1$ and for
any $\varepsilon\in(0,\frac{1}{2})$ we immediately obtain 
\begin{equation}
\varepsilon (1+\vert\xi\vert^{2})^{\frac{p-2}{2}} \vert\xi\vert^{2} \leq \langle A_{\varepsilon}(x,t,\xi),\xi\rangle \leq (C_{1}+\varepsilon) (1+\vert\xi\vert^{2})^{\frac{p}{2}}\label{eq:gro1}
\end{equation}
for every $\xi\in\mathbb{R}^{Nn}$ and for almost every $(x,t)\in\Omega_{T}$.
As for the spatial gradient of $A_{\varepsilon}$, by the assumption
($\mathrm{H_{4}}$) we have 
\begin{equation}
\left|D_{x}A_{\varepsilon}(x,t,\xi)\right| \leq 2^{p-1} K (1+\vert\xi\vert)^{p-1}\label{eq:spgraAeps}
\end{equation}
for every $\xi\in\mathbb{R}^{Nn}$, for every $\varepsilon\in(0,\frac{1}{2})$
and for almost every $(x,t)\in\Omega_{T}$.

Now we want to examine the structure of $D_{\xi}A_{\varepsilon}(x,t,\xi)$.
To this end, for any $\xi\in\mathbb{R}^{Nn}\setminus\{0\}$ and for
any $\varepsilon\in(0,\frac{1}{2})$ we define the bilinear form $\mathbf{\mathcal{A}}_{\varepsilon}(x,t,\xi):\mathbb{R}^{Nn^{2}}\times\mathbb{R}^{Nn^{2}}\rightarrow\mathbb{R}$
by 
\[
\mathbf{\mathcal{A}}_{\varepsilon}(x,t,\xi)(\lambda,\zeta):= h_{\varepsilon}(x,t,\left|\xi\right|) \lambda\cdot\zeta+\partial_{s}h_{\varepsilon}(x,t,\left|\xi\right|)\sum_{i,j = 1}^{N}\sum_{k,\ell,m = 1}^{n}\frac{\xi_{k}^{i} \lambda_{k m}^{i} \xi_{\ell}^{j} \zeta_{\ell m}^{j}}{\left|\xi\right|}\qquad \text{for }  \lambda,\zeta\in\mathbb{R}^{Nn^{2}},
\]
and observe that 
\begin{equation}
\sum_{i,j = 1}^{N}\sum_{k,\ell,m = 1}^{n}D_{\xi_{k}^{i}}(A_{\varepsilon})_{\ell}^{j}(x,t,\xi) \lambda_{k m}^{i} \zeta_{\ell m}^{j} = \mathbf{\mathcal{A}}_{\varepsilon}(x,t,\xi)(\lambda,\zeta).\label{eq:bil_form}
\end{equation}

The next lemma provides the relevant ellipticity and boundedness properties
of the bilinear form $\mathbf{\mathcal{A}}_{\varepsilon}(x,t,\xi)$. 

\begin{lem}\label{lem:lemma app1}
Let $1<p<\infty$ and $\varepsilon\in(0,\frac{1}{2})$.
Then, there exists a positive constant $C\equiv C(p,C_{1},\varepsilon)$
such that, for any $\xi\in\mathbb{R}^{Nn}\setminus\{0\}$ and any
$\lambda\in\mathbb{R}^{Nn^{2}}$, we have 
\begin{equation}
\varepsilon\min\{1,p-1\}(1+\vert\xi\vert^{2})^{\frac{p-4}{2}}\vert\xi\vert^{2}\vert\lambda\vert^{2} \leq \mathbf{\mathcal{A}}_{\varepsilon}(x,t,\xi)(\lambda,\lambda)  \leq C (1+\left|\xi\right|^{2})^{\frac{p-2}{2}}\left|\lambda\right|^{2}.\label{eq:lemmabili_a}
\end{equation}
Moreover, for $\delta>4 \varepsilon^{\frac{1}{p-1}}$ and $\left|\xi\right|\geq1+\frac{\delta}{2}$
we get 
\begin{equation}
\frac{1}{c} \left|\xi\right|^{p-2}\left|\lambda\right|^{2} \leq \mathbf{\mathcal{A}}_{\varepsilon}(x,t,\xi)(\lambda,\lambda) \leq c \left|\xi\right|^{p-2}\left|\lambda\right|^{2},\label{eq:lemmabili}
\end{equation}
where $c\equiv c(p,C_{1},\delta)>1$. 
\end{lem}

\begin{proof}[\textbf{{Proof}}]
From (\ref{eq:h-fun-bis}) it
follows that 
\[
\partial_{s}h_{\varepsilon}(x,t,s)=\begin{cases}
\frac{\partial_{ss}F_{\varepsilon}(x, t, s)}{s}-\frac{\partial_{s}F_{\varepsilon}(x, t, s)}{s^{2}}+(p-2) \varepsilon s (1+s^{2})^{\frac{p-4}{2}} & \mathrm{if \ } 1<p\leq2\\
\frac{\partial_{ss}F(x, t, s)}{s}-\frac{\partial_{s}F(x, t, s)}{s^{2}}+(p-2) \varepsilon s (1+s^{2})^{\frac{p-4}{2}} & \mathrm{if \ } p>2.
\end{cases}
\]
In order to prove the assertion, we distinguish between two cases.

When $\partial_{s}h_{\varepsilon}(x,t,\left|\xi\right|)<0$,
from Lemma \ref{eq:approxin} and the growth assumption ($\mathrm{H_{2}}$)
we obtain 
\begin{align*}
\mathbf{\mathcal{A}}_{\varepsilon}(x,t,\xi)(\lambda,\lambda) & \leq h_{\varepsilon}(x,t,\left|\xi\right|)\left|\lambda\right|^{2}\\
 & \leq\begin{cases}
C_{1}\left|\xi\right|^{p-2}\mathds{1}_{\{\left|\xi\right| \geq 1\}}\left|\lambda\right|^{2}+\varepsilon (1+\left|\xi\right|^{2})^{\frac{p-2}{2}}\left|\lambda\right|^{2} & \mathrm{for \ } 1<p\leq 2\\
(C_{1}+\varepsilon) (1+\left|\xi\right|^{2})^{\frac{p-2}{2}}\left|\lambda\right|^{2} & \mathrm{for \ } p>2
\end{cases}\\
 & \leq\begin{cases}
(2^{\frac{2-p}{2}}C_{1}+1) (1+\left|\xi\right|^{2})^{\frac{p-2}{2}}\left|\lambda\right|^{2} & \mathrm{for \ } 1<p\leq 2\\
(C_{1}+1) (1+\left|\xi\right|^{2})^{\frac{p-2}{2}}\left|\lambda\right|^{2} & \mathrm{for \ } p > 2,
\end{cases}
\end{align*}
where we have applied the inequality $\vert\xi\vert^{2} \mathds{1}_{\{\left|\xi\right| \geq 1\}}\geq\frac{1}{2} (1+\vert\xi\vert^{2}) \mathds{1}_{\{\left|\xi\right| \geq 1\}}$
in the case $1<p\leq 2$ and 
\[
\frac{(s-1)_{+}^{p-1}}{s} \leq s^{p-2} \leq (1+s^{2})^{\frac{p-2}{2}}\qquad \mathrm{for} \ s>0 \ \mathrm{and} \ p>2.
\]
This proves the upper bound in (\ref{eq:lemmabili_a}), and the one
in (\ref{eq:lemmabili}) is an immediate consequence. Moreover, using
the Cauchy-Schwarz inequality, we get 
\begin{align*}
\mathbf{\mathcal{A}}_{\varepsilon}(x,t,\xi)(\lambda,\lambda) & \geq\begin{cases}
\partial_{ss}F_{\varepsilon}(x,t,\left|\xi\right|)\left|\lambda\right|^{2}+(p-1) \varepsilon (1+\left|\xi\right|^{2})^{\frac{p-4}{2}}\left|\xi\right|^{2}\left|\lambda\right|^{2} & \mathrm{for \ }1<p\leq 2,\\
\partial_{ss}F(x,t,\left|\xi\right|)\left|\lambda\right|^{2}+(p-1) \varepsilon (1+\left|\xi\right|^{2})^{\frac{p-4}{2}}\left|\xi\right|^{2}\left|\lambda\right|^{2} & \mathrm{for \ }  p>2.\end{cases}
\end{align*}
Since $\partial_{ss}F_{\varepsilon}(x,t,\vert\xi\vert)$, $\partial_{ss}F(x,t,\vert\xi\vert)\geq0$,
we find that 
\[
\mathbf{\mathcal{A}_\varepsilon}(x,t,\xi)(\lambda,\lambda) \geq (p-1) \varepsilon (1+\vert\xi\vert^{2})^{\frac{p-4}{2}}\vert\xi\vert^{2}\vert\lambda\vert^{2}\qquad \text{for every }p>1,
\]
thus proving the lower bound in (\ref{eq:lemmabili_a}). To obtain
the lower bound in (\ref{eq:lemmabili}), we observe that $F_{\varepsilon}(x,t,\left|\xi\right|)=F(x,t,\left|\xi\right|)$
for $\left|\xi\right|\geq1+\frac{\delta}{2}$, so that in view of
($\mathrm{H_{3}}$) we have 
\[
\mathbf{\mathcal{A}}_{\varepsilon}(x,t,\xi)(\lambda,\lambda) \geq \partial_{ss}F(x,t,\vert\xi\vert)\left|\lambda\right|^{2} \geq \frac{1}{C_{1}} (\vert\xi\vert-1)^{p-2}\left|\lambda\right|^{2} \geq \frac{1}{c} \left|\xi\right|^{p-2}\left|\lambda\right|^{2}
\]
for some positive constant $c\equiv c(p,C_{1},\delta)$.

In the case $\partial_{s}h_{\varepsilon}(x,t,\left|\xi\right|)\geq0$,
applying the Cauchy-Schwarz inequality again, from Lemma \ref{eq:approxin}
and the growth condition ($\mathrm{H_{3}}$) we get 
\begin{align*}
\mathbf{\mathcal{A}}_{\varepsilon}(x,t,\xi)(\lambda,\lambda) & \leq\begin{cases}
\partial_{ss}F_{\varepsilon}(x,t,\left|\xi\right|)\left|\lambda\right|^{2}+(p-1) \varepsilon (1+\left|\xi\right|^{2})^{\frac{p-4}{2}}\left|\xi\right|^{2}\left|\lambda\right|^{2} & \mathrm{for \ }1<p\leq2\\
\partial_{ss}F(x,t,\left|\xi\right|)\left|\lambda\right|^{2}+(p-1) \varepsilon (1+\left|\xi\right|^{2})^{\frac{p-4}{2}}\left|\xi\right|^{2}\left|\lambda\right|^{2} & \mathrm{for \ } p>2\end{cases}\\
 & \leq\begin{cases}
\left(p-1+C \varepsilon^{\frac{p-2}{p-1}}\right) (1+\left|\xi\right|^{2})^{\frac{p-2}{2}}\left|\lambda\right|^{2} & \mathrm{for \ } 1<p\leq 2\\
(p-1+C_{1}) (1+\left|\xi\right|^{2})^{\frac{p-2}{2}}\left|\lambda\right|^{2} & \mathrm{for \ } p > 2,
\end{cases}
\end{align*}
where we have used the inequality $(s-1)_{+}^{p-2}\leq(s^{2}+1)^{\frac{p-2}{2}}$,
which holds for every $s\geq0$ and every $p>2$. To obtain the lower
bound in (\ref{eq:lemmabili_a}), we neglect the term $\partial_{s}h_{\varepsilon}(x,t,\left|\xi\right|)$
and use the fact that $\partial_{s}F_{\varepsilon}(x,t,\vert\xi\vert)$,
$\partial_{s}F(x,t,\vert\xi\vert)\geq0$. Thus we have 
\[
    \mathbf{\mathcal{A}}_{\varepsilon}(x,t,\xi)(\lambda,\lambda)
    \geq
    h_{\varepsilon}(x,t,\left|\xi\right|)\left|\lambda\right|^{2}
    \geq
    \varepsilon (1+\left|\xi\right|^{2})^{\frac{p-4}{2}}\left|\xi\right|^{2}\left|\lambda\right|^{2}\qquad \mathrm{for \ every \ }p>1.
\]
Finally, to establish the bounds in (\ref{eq:lemmabili}), one can
argue as above. This time, for $\left|\xi\right|\geq1+\frac{\delta}{2}$
we are allowed to use the growth assumptions ($\mathrm{H_{2}}$) and
($\mathrm{H_{3}}$) also in the case $1<p\leq 2$, since $F_{\varepsilon}(x,t,\left|\xi\right|)=F(x,t,\left|\xi\right|)$.
Furthermore, we can estimate the term $\varepsilon (1+\left|\xi\right|^{2})^{\frac{p-4}{2}}\vert\xi\vert^{2}$
by zero from below and make use of $\varepsilon (1+\left|\xi\right|^{2})^{\frac{p-4}{2}} \vert\xi\vert^{2}\leq\varepsilon (1+\left|\xi\right|^{2})^{\frac{p-2}{2}}\leq2^{\frac{(p-2)_{+}}{2}} \vert\xi\vert^{p-2}$.
After doing this, we get the desired estimates by means of the following
inequalities, which hold for $\left|\xi\right|\geq1+\frac{\delta}{2}$
: 
\begin{equation}
\Big(\frac{\delta}{2+\delta}\Big)^{p-1}\left|\xi\right|^{p-2} \leq \frac{(\left|\xi\right|-1)_{+}^{p-1}}{\left|\xi\right|} \leq \left|\xi\right|^{p-2}\label{eq:approxi1}
\end{equation}
and
\begin{equation}\label{eq:approxi2}
    \min\bigg\{1, \Big(\frac{\delta}{2+\delta}\Big)^{p-2}\bigg\} |\xi|^{p-2}
    \le 
    (|\xi|-1)_{+}^{p-2}
    \le 
    \max\bigg\{1, \Big(\frac{\delta}{2+\delta}\Big)^{p-2}\bigg\} |\xi|^{p-2}.
\end{equation}
\end{proof}

From the previous lemma, it follows that
the bilinear form $(\lambda,\zeta)\mapsto\mathbf{\mathcal{A}}_{\varepsilon}(x,t,\xi)(\lambda,\zeta)$
defines a scalar product on the Euclidean space $\mathbb{R}^{Nn^{2}}$.
As for the modulus of $\mathbf{\mathcal{A}}_{\varepsilon}$, we get
the following result: 

\begin{lem}\label{lem:lemma app2}
Let $1<p<\infty$ and $\varepsilon\in(0,\frac{1}{2})$.
Then, there exists a positive constant $C\equiv C(p,C_{1},\varepsilon)$
such that, for any $\xi\in\mathbb{R}^{Nn}\setminus\{0\}$ and any
$\lambda,\zeta\in\mathbb{R}^{Nn^{2}}$, we have 
\[
\left|\mathbf{\mathcal{A}}_{\varepsilon}(x,t,\xi)(\lambda,\zeta)\right| \leq C (1+\left|\xi\right|^{2})^{\frac{p-2}{2}}\left|\lambda\right|\left|\zeta\right|.
\]
Moreover, for $\delta>4 \varepsilon^{\frac{1}{p-1}}$ and $\left|\xi\right|\geq1+\frac{\delta}{2}$
we get 
\[
\left|\mathbf{\mathcal{A}}_{\varepsilon}(x,t,\xi)(\lambda,\zeta)\right| \leq C_{2} \left|\xi\right|^{p-2}\left|\lambda\right|\left|\zeta\right|,
\]
where $C_{2}\equiv C_{2}(\delta,p,C_{1})>0$. 
\end{lem}

\begin{proof}[\textbf{{Proof}}]
For every $\xi\in\mathbb{R}^{Nn}\setminus\{0\}$
and every $\lambda,\zeta\in\mathbb{R}^{Nn^{2}}$, we obtain 
\[
\left|\mathbf{\mathcal{A}}_{\varepsilon}(x,t,\xi)(\lambda,\zeta)\right| \leq \bigg[\left|h_{\varepsilon}(x,t,\left|\xi\right|)\right|+\left|\partial_{s}h_{\varepsilon}(x,t,\left|\xi\right|)\right|\left|\xi\right|\bigg]\left|\lambda\right|\left|\zeta\right|.
\]
Now, by ($\mathrm{H_{2}}$) and ($\mathrm{H_{3}}$), in the case $p>2$
we get 
\begin{align}
\left|\mathbf{\mathcal{A}}_{\varepsilon}(x,t,\xi)(\lambda,\zeta)\right|  & \leq \left(2  \frac{\partial_{s}F(x,t,\vert\xi\vert)}{\vert\xi\vert}+\partial_{ss}F(x,t,\vert\xi\vert)+\varepsilon (p-1) (1+\vert\xi\vert^{2})^{\frac{p-2}{2}}\right)\left|\lambda\right|\left|\zeta\right|\nonumber\\
 & \leq \left(2 C_{1} \frac{(\vert\xi\vert-1)_{+}^{p-1}}{\vert\xi\vert}+C_{1} (\vert\xi\vert-1)_{+}^{p-2}+\varepsilon (p-1) (1+\vert\xi\vert^{2})^{\frac{p-2}{2}}\right)\left|\lambda\right|\left|\zeta\right|\nonumber\\
 & \leq (3 C_{1}+p-1) (1+\vert\xi\vert^{2})^{\frac{p-2}{2}} \left|\lambda\right|\left|\zeta\right|.
\label{eq:estapprox}
\end{align}
In the case $1<p\leq2$, we use the estimates from Lemma \ref{eq:approxin} to find that 
\begin{align*}
\left|\mathbf{\mathcal{A}}_{\varepsilon}(x,t,\xi)(\lambda,\zeta)\right|  & \leq \left(2  \frac{\partial_{s}F_{\varepsilon}(x,t,\vert\xi\vert)}{\vert\xi\vert}+\partial_{ss}F_{\varepsilon}(x,t,\vert\xi\vert)+\varepsilon (p-1) (1+\vert\xi\vert^{2})^{\frac{p-2}{2}}\right)\left|\lambda\right|\left|\zeta\right|\\
 & \leq \left(2 C_{1} \vert\xi\vert^{p-2} \mathds{1}_{\{ |\xi| > 1\}}(\vert\xi\vert)+\left(c  \varepsilon^{\frac{p-2}{p-1}}+(p-1) \varepsilon\right) (1+\vert\xi\vert^{2})^{\frac{p-2}{2}}\right)\left|\lambda\right|\left|\zeta\right|\\
 & \leq \left(2^{\frac{4-p}{2}} C_{1}+c  \varepsilon^{\frac{p-2}{p-1}}+(p-1) \varepsilon\right) (1+\vert\xi\vert^{2})^{\frac{p-2}{2}} \left|\lambda\right|\left|\zeta\right|.
\end{align*}
We thus obtain the first conclusion of this lemma.

Finally, due to equality (\ref{eq:equaDsFeps}),
if $\left|\xi\right|\geq1+\frac{\delta}{2}$ we only need to use the
assumptions ($\mathrm{H_{2}}$) and ($\mathrm{H_{3}}$) together with
the inequalities (\ref{eq:approxi1}) and (\ref{eq:approxi2}) to
obtain from \eqref{eq:estapprox} 
\begin{align*}
\left|\mathbf{\mathcal{A}}_{\varepsilon}(x,t,\xi)(\lambda,\zeta)\right|  & \leq \left(2 C_{1} \frac{(\vert\xi\vert-1)_{+}^{p-1}}{\vert\xi\vert}+C_{1} (\vert\xi\vert-1)_{+}^{p-2}+\varepsilon (p-1) (1+\vert\xi\vert^{2})^{\frac{p-2}{2}}\right)\left|\lambda\right|\left|\zeta\right|\\
 & \leq \begin{cases}
\left(2 C_{1}+C_{1}\left(\frac{\delta}{2+\delta}\right)^{p-2}+p-1\right)\left|\xi\right|^{p-2}\left|\lambda\right|\left|\zeta\right| & \mathrm{for \ }1<p\leq2\\
\left(3 C_{1}+(p-1) 2^{\frac{p-2}{2}}\right)\left|\xi\right|^{p-2}\left|\lambda\right|\left|\zeta\right| & \mathrm{for \ } p>2. \end{cases}
\end{align*}
These inequalities conclude the proof.
\end{proof}

\subsection{Algebraic inequalities}\label{subsec:Algebraic}

In this section, we gather some relevant
algebraic inequalities that will be needed later on. We start with an elementary assertion, which will be used in the Moser iteration.
\begin{lem}\label{lem:lemit}
    Let $A>1$, $\kappa>1$, $0\leq\alpha<1$, $C,c>0$ and $i\in \mathbb{N}$. Furthermore, let $\{ \beta_j \}_{j\in\mathbb{N}_0}$ satisfy 
    \begin{equation*}
        \beta_j \geq C \kappa^{j}
    \end{equation*}
    for each $j\in \mathbb{N}_0$. Then we have
    \begin{equation}\label{eqit1}
        \prod\limits_{j=0}^{i} A^{\frac{\kappa^{i-(1-\alpha)j}}{\beta_{i+1}}} \leq A^{\frac{1}{C \kappa (1-\kappa^{\alpha-1})}}
    \end{equation}
    and
    \begin{equation}\label{eqit2}
        \prod\limits_{j=0}^{i} A^{j c \,\frac{\kappa^{i+1-j}}{\beta_{i+1} }} \leq A^{\frac{c \kappa}{C (1-\kappa)^2}}.
    \end{equation}
\end{lem}
\begin{proof}[\textbf{{Proof}}]
    The proof of the second inequality is similar to the one of Lemma 2.3. in \cite{bogelein2022boundary}. Thus we only prove the first inequality:
    \begin{align*}
        \prod\limits_{j=0}^{i}  A^{\frac{\kappa^{i-(1-\alpha)j}}{\beta_{i+1}}} 
        & \leq
        A^{\frac{1}{C\kappa} \sum\limits_{j=0}^i \kappa^{(\alpha-1)j} }
        \leq
        A^{\frac{1}{C\kappa} \sum\limits_{j=0}^\infty \kappa^{(\alpha-1)j} }
        = 
        A^{\frac{1}{C \kappa (1-\kappa^{\alpha-1})}}
    \end{align*}
\end{proof}

For $\delta\in(0,1]$ we define the auxiliary function $G_{\delta}:\mathbb{R}^{k}\rightarrow\mathbb{R}^{k}$,
$k\in\mathbb{N}$, as follows 
\[
G_{\delta}(\xi):=\begin{cases}
(\left|\xi\right|-1-\delta)_{+} \frac{\xi}{\left|\xi\right|} &\mathrm{if \ } \xi\in\mathbb{R}^{k}\setminus\{0\},\\
0 & \mathrm{if \ }\xi=0.
\end{cases}
\]

The following two lemmas are concerned with auxiliary estimates for the functions $A_{\varepsilon}$ and $G_{\delta}$ defined above.

\begin{lem}\label{lem:lem3}
Let $1<p<\infty$, $\delta>0$ and $\varepsilon\in(0,\frac{1}{2})$.
Then, there exists a positive constant $C\equiv C(p,C_{1},\delta)$
such that for every $\xi,\tilde{\xi}\in\mathbb{R}^{Nn}$ with $\left|\xi\right|>1+\frac{\delta}{2}$
we have 
\[
\langle A_{\varepsilon}(x,t,\xi)-A_{\varepsilon}(x,t,\tilde{\xi}),\xi-\tilde{\xi}\rangle \geq C\,  \frac{(\left|\xi\right|-1-\frac{\delta}{2})^{p}}{\left|\xi\right|(\left|\xi\right|+\vert\tilde{\xi}\vert)}  \vert\xi-\tilde{\xi}\vert^{2}.
\]
\end{lem}

\begin{proof}[\textbf{{Proof}}]
Here we argue as in \cite[Lemma 2.8]{BDGP}. The inequality
on the left-hand side of (\ref{eq:lemmabili}) implies that 
\[
\mathbf{\mathcal{A}}_{\varepsilon}(x,t,\xi)(\lambda,\lambda) \geq c\,  \frac{(\vert\xi\vert-1-\frac{\delta}{2})_{+}^{p-1}}{\vert\xi\vert} \vert\lambda\vert^{2}
\]
holds for any $\xi\in\mathbb{R}^{Nn}\setminus\{0\}$ and any $\lambda\in\mathbb{R}^{Nn^{2}}$,
where $c\equiv c(p,C_{1},\delta)>0$. Abbreviating $\xi_{s}:=\xi+s (\tilde{\xi}-\xi)$,
for $s\in[0,1]$, we find 
\begin{align}
\langle A_{\varepsilon}(x,t,\xi)-A_{\varepsilon}(x,t,\tilde{\xi}),\xi-\tilde{\xi}\rangle  & \geq c\int_{0}^{1} \frac{(\left|\xi_{s}\right|-1-\frac{\delta}{2})_{+}^{p-1}}{\left|\xi_{s}\right|}  ds  \vert\xi-\tilde{\xi}\vert^{2}\nonumber\\
 & \geq\frac{c}{\left|\xi\right|+\vert\tilde{\xi}\vert}\int_{0}^{1}\left(\left|\xi_{s}\right|-1-\tfrac{\delta}{2}\right)_{+}^{p-1}ds  \vert\xi-\tilde{\xi}\vert^{2}.
\label{eq:est1}
\end{align}
It remains to estimate the integral in the right-hand side of \eqref{eq:est1}.
To this end, we distinguish whether or not $\vert\tilde{\xi}\vert\leq|\xi\vert$.
If $\vert\tilde{\xi}\vert\leq|\xi\vert$, then 
\[
\vert\xi_{s}\vert\geq(1-s) \vert\xi\vert-s \vert\tilde{\xi}\vert \geq (1-2s) \vert\xi\vert>1+\frac{\delta}{2}\qquad \forall \ s\in\bigg[0,\frac{\left|\xi\right|-1-\frac{\delta}{2}}{2\left|\xi\right|}\bigg).
\]
For $s\in\big[0,\frac{|\xi|-1-\delta/2}{4|\xi|}\big]$
this implies a bound from below in the form 
\[
\left(\left|\xi_{s}\right|-1-\frac{\delta}{2}\right)_{+}\geq(1-2s) \vert\xi\vert-1-\frac{\delta}{2}
\geq
\bigg(1-\frac{\left|\xi\right|-1-\frac{\delta}{2}}{2\left|\xi\right|}\bigg) \vert\xi\vert-1-\frac{\delta}{2}\geq\frac{1}{2} \left(\left|\xi\right|-1-\tfrac{\delta}{2}\right).
\]
Thus, we obtain that 
\begin{equation}
\int_{0}^{1}\left(\left|\xi_{s}\right|-1-\tfrac{\delta}{2}\right)_{+}^{p-1}ds \geq\int_{0}^{\frac{\left|\xi\right|-1-\frac{\delta}{2}}{4\left|\xi\right|}}\left(\left|\xi_{s}\right|-1-\tfrac{\delta}{2}\right)_{+}^{p-1}ds \geq \frac{\left(\left|\xi\right|-1-\frac{\delta}{2}\right)^{p}}{2^{p+1}\vert\xi\vert}.\label{eq:est2}
\end{equation}
In the case $\vert\tilde{\xi}\vert>|\xi\vert$, we estimate $\vert\xi_{s}\vert$
from below as follows 
\[
\vert\xi_{s}\vert\geq s \vert\tilde{\xi}\vert-(1-s) \vert\xi\vert \geq (2s-1) \vert\xi\vert>1+\frac{\delta}{2}\qquad \forall \ s\in\bigg(\frac{\left|\xi\right|+1+\frac{\delta}{2}}{2\left|\xi\right|},1\bigg].
\]
Therefore, for $s\in\big[\frac{3\left|\xi\right|+ 1+\delta/2}{4\left|\xi\right|},1\big]$
we get 
\[
\left(\left|\xi_{s}\right|-1-\tfrac{\delta}{2}\right)_{+}\geq(2s-1) \vert\xi\vert-1-\frac{\delta}{2}\geq\bigg(\frac{3\left|\xi\right|+1+\frac{\delta}{2}}{2\left|\xi\right|}-1\bigg)\vert\xi\vert-1-\frac{\delta}{2}\geq\frac{1}{2} \left(\left|\xi\right|-1-\tfrac{\delta}{2}\right),
\]
which yields 
\begin{equation}
\int_{0}^{1}\left(\left|\xi_{s}\right|-1-\tfrac{\delta}{2}\right)_{+}^{p-1}ds \geq\int_{\frac{3\left|\xi\right|+ 1+\frac{\delta}{2}}{4\left|\xi\right|}}^{1}\left(\left|\xi_{s}\right|-1-\tfrac{\delta}{2}\right)_{+}^{p-1}ds \geq \frac{\left(\left|\xi\right|-1-\frac{\delta}{2}\right)^{p}}{2^{p+1}\vert\xi\vert}.\label{eq:est3}
\end{equation}
Combining estimates (\ref{eq:est2}) and (\ref{eq:est3}) with \eqref{eq:est1},
we conclude the proof.
\end{proof}

Using the previous lemma, we can easily achieve the following
result:
\begin{lem}\label{lem:Glemma}
Let $p>1$, $\delta>0$, $0<\varepsilon<\min \{\frac{1}{2},(\frac{\delta}{4})^{p-1}\}$
and $\xi,\tilde{\xi}\in\mathbb{R}^{Nn}$. Then, for every $\nu>0$
and almost every $x\in\Omega$ we have 
\[
\vert G_{\delta}(\xi)-G_{\delta}(\tilde{\xi})\vert^{p} \leq \varepsilon^{\nu}(\max \{\vert\xi\vert,1+\delta\})^{p}+C \varepsilon^{-\nu}\langle A_{\varepsilon}(x,t,\xi)-A_{\varepsilon}(x,t,\tilde{\xi}),\xi-\tilde{\xi}\rangle
\]
for a positive constant $C\equiv C(p,C_{1},\delta)$.
\end{lem}

\begin{proof}[\textbf{{Proof}}]
If $\vert\xi\vert,\vert\tilde{\xi}\vert\leq1+\delta$,
we have $G_{\delta}(\xi)=G_{\delta}(\tilde{\xi})=0$. Therefore, the
claimed inequality immediately follows from Lemma \ref{lem:lem3},
by the positivity of the right-hand side.

Thus, we only need to consider the case where either
$\vert\xi\vert>1+\delta$ or $\vert\tilde{\xi}\vert>1+\delta$. In
order to do this, we first assume that $\vert\xi\vert\geq\vert\tilde{\xi}\vert$.
Note that this implies $\vert\xi\vert>1+\delta$. From \cite[Lemma 2.3]{BDGP}
we know that 
\[
\vert G_{\delta}(\xi)-G_{\delta}(\tilde{\xi})\vert\leq3 \vert\xi-\tilde{\xi}\vert.
\]
For $p\geq2$ this yields 
\[
\vert G_{\delta}(\xi)-G_{\delta}(\tilde{\xi})\vert^{p} \leq 3^{p} \vert\xi-\tilde{\xi}\vert^{p} \leq 3^{p} (\vert\xi\vert+\vert\tilde{\xi}\vert)^{p-2} \vert\xi-\tilde{\xi}\vert^{2} \leq 3^{p} 2^{p-2} \vert\xi\vert^{p-2} \vert\xi-\tilde{\xi}\vert^{2},
\]
while in the case $1<p<2$ we use Young's inequality with exponents
$(\frac{2}{p},\frac{2}{2-p})$ to obtain 
\begin{align*}
\vert G_{\delta}(\xi)-G_{\delta}(\tilde{\xi})\vert^{p}  & = \vert\xi\vert^{\frac{(2-p)p}{2}} \vert\xi\vert^{\frac{(p-2)p}{2}} \varepsilon^{\frac{(2-p)\nu}{2}} \varepsilon^{\frac{(p-2)\nu}{2}} \vert G_{\delta}(\xi)-G_{\delta}(\tilde{\xi})\vert^{p}\\
 & \leq \frac{(2-p) \varepsilon^{\nu}}{4} \vert\xi\vert^{p}+2^{\frac{2(1-p)}{p}}p \varepsilon^{\frac{(p-2)\nu}{p}} \vert\xi\vert^{p-2} \vert G_{\delta}(\xi)-G_{\delta}(\tilde{\xi})\vert^{2}\\
 & \leq \frac{\varepsilon^{\nu}}{4} \vert\xi\vert^{p}+9 p \varepsilon^{-\nu} \vert\xi\vert^{p-2} \vert\xi-\tilde{\xi}\vert^{2},
\end{align*}
where we have used the fact that $0<\varepsilon<\frac{1}{2}$ and
$0<\frac{2-p}{p}<1$ whenever $p\in(1,2)$. Combining both cases,
taking into account the inequalities $\frac{1}{\vert\xi\vert}\leq\frac{2}{\vert\xi\vert+\vert\tilde{\xi}\vert}$
and $\vert\xi\vert\leq\left(2+\frac{2}{\delta}\right)(\vert\xi\vert-1-\frac{\delta}{2})$,
and finally applying Lemma \ref{lem:lem3}, we get 
\begin{align}\vert G_{\delta}(\xi)-G_{\delta}(\tilde{\xi})\vert^{p} & \leq \varepsilon^{\nu} \vert\xi\vert^{p}+\varepsilon^{-\nu} c(\delta,p)  \frac{\left(\vert\xi\vert-1-\frac{\delta}{2}\right)^{p}}{\vert\xi\vert (\vert\xi\vert+\vert\tilde{\xi}\vert)}  \vert\xi-\tilde{\xi}\vert^{2}\nonumber\\
 & \leq \varepsilon^{\nu} \vert\xi\vert^{p}+C(p,C_{1},\delta) \varepsilon^{-\nu} \langle A_{\varepsilon}(x,t,\xi)-A_{\varepsilon}(x,t,\tilde{\xi}),\xi-\tilde{\xi}\rangle.
\label{eq:est4_a}
\end{align}
This proves the claimed inequality in the case $\vert\xi\vert\geq\vert\tilde{\xi}\vert$.
In the remaining case, i.e. $\vert\tilde{\xi}\vert>\vert\xi\vert$,
by interchanging the roles of $\xi$ and $\tilde{\xi}$ and replacing
$\varepsilon$ with $\frac{\varepsilon}{2^{\frac{2+(p-2)_{+}}{\nu}}}$, we obtain from \eqref{eq:est4_a} that
\begin{align*}
\vert G_{\delta}(\xi)-G_{\delta}(\tilde{\xi})\vert^{p} & \leq \frac{\varepsilon^{\nu}}{2^{2+(p-2)_{+}}} \vert\tilde{\xi}\vert^{p}+C \varepsilon^{-\nu} \langle A_{\varepsilon}(x,t,\xi)-A_{\varepsilon}(x,t,\tilde{\xi}),\xi-\tilde{\xi}\rangle\nonumber \\
 & = \frac{\varepsilon^{\nu}}{2^{2+(p-2)_{+}}} (\vert G_{\delta}(\tilde{\xi})\vert+1+\delta)^{p}+C \varepsilon^{-\nu} \langle A_{\varepsilon}(x,t,\xi)-A_{\varepsilon}(x,t,\tilde{\xi}),\xi-\tilde{\xi}\rangle \nonumber \\
 & \leq \frac{\varepsilon^{\nu}}{2} (\vert G_{\delta}(\xi)\vert+1+\delta)^{p}+C \varepsilon^{-\nu} \langle A_{\varepsilon}(x,t,\xi)-A_{\varepsilon}(x,t,\tilde{\xi}),\xi-\tilde{\xi}\rangle \nonumber \\
 & \quad +\frac{\varepsilon^{\nu}}{2} \vert G_{\delta}(\xi)-G_{\delta}(\tilde{\xi})\vert^{p} \nonumber \\
 & \leq \frac{\varepsilon^{\nu}}{2} (\max \{\vert\xi\vert,1+\delta\})^{p}+C \varepsilon^{-\nu} \langle A_{\varepsilon}(x,t,\xi)-A_{\varepsilon}(x,t,\tilde{\xi}),\xi-\tilde{\xi}\rangle \nonumber \\
 & \quad +\frac{1}{2} \vert G_{\delta}(\xi)-G_{\delta}(\tilde{\xi})\vert^{p}.
\end{align*}
Absorbing the last term on the right-hand side
into the left-hand side, we obtain the desired result. 
\end{proof}

\section{A family of regularized parabolic problems \label{sec:uniqueness}}

In this section, we let $u$ be a weak solution of (\ref{eq:syst}) and introduce the approximating
Cauchy-Dirichlet problem (\ref{eq:CAUCHYDIR}), where $\varepsilon\in(0,\frac{1}{2})$
is the approximation parameter. Denoting by $u_{\varepsilon}$ the
unique weak solution of this problem, we will prove the strong
convergence $G_{\delta}(Du_{\varepsilon})\rightarrow G_{\delta}(Du)$ in $L^{p}$ as $\varepsilon\rightarrow0$. As an easy consequence of this convergence, we then establish
the uniqueness of weak solutions to an initial-boundary value problem
associated with (\ref{eq:syst}). 

To set up the approximating problem, for $\varepsilon\in(0,\frac{1}{2})$ we define the truncated vector-valued function $f_{\varepsilon}$
by 
\begin{equation}
f_{\varepsilon}^{i}:=\max\left\{ - \tfrac{1}{\varepsilon} ,\min\left\{ f^{i},\tfrac{1}{\varepsilon}\right\} \right\} ,\qquad i\in\{1,\ldots,N\}.\label{eq:f_eps}
\end{equation}
Moreover, we consider a space-time cylinder $Q':=\Omega'\times I$,
where $\Omega'\subseteq\Omega$ is a bounded domain and $I:=(t_{1},t_{2})\subseteq(0,T)$.
In what follows, it will suffice to assume that $f\in L^{2}(Q',\mathbb{R}^{N})$.\medskip{}

\begin{defn}\label{def:L2-traces} 
Let $\varepsilon\in(0,\frac{1}{2})$
and $u\in C^{0}\left([t_{1},t_{2}];L^{2}\left(\Omega',\mathbb{R}^{N}\right)\right)\cap L^{p}\left(t_{1},t_{2};W^{1,p}\left(\Omega',\mathbb{R}^{N}\right)\right)$.
In this framework, we identify a function 
\[
u_{\varepsilon}\in C^{0}\left([t_{1},t_{2}];L^{2}\left(\Omega',\mathbb{R}^{N}\right)\right)\cap L^{p}\left(t_{1},t_{2};W^{1,p}\left(\Omega',\mathbb{R}^{N}\right)\right)
\]
as a \textit{weak solution of the Cauchy-Dirichlet problem\medskip{}
 } 
\begin{equation}
\begin{cases}
\partial_{t}u_{\varepsilon}-\mathrm{div} [A_{\varepsilon}(x,t,Du_{\varepsilon})]=f_{\varepsilon} & \qquad \text{in } Q',\\
u_{\varepsilon}=u & \qquad \text{on } \partial_{\mathrm{par}}Q',
\tag{\ensuremath{\mathcal{P}_{\varepsilon}}}\end{cases}\label{eq:CAUCHYDIR}
\end{equation}
if and only if $u_{\varepsilon}$ is a weak solution of $(\ref{eq:CAUCHYDIR})_{1}$
and, moreover, 
\[
u_{\varepsilon} \in u+L^{p}\left(t_{1},t_{2};W_{0}^{1,p}\left(\Omega',\mathbb{R}^{N}\right)\right)
\]
and $u_{\varepsilon}(\cdot,t_{1})=u(\cdot,t_{1})$ in the
$L^{2}$-sense, that is 
\begin{equation}
\underset{t \searrow t_{1}}{\lim} \Vert u_{\varepsilon}(\cdot,t)-u(\cdot,t_{1})\Vert_{L^{2}(\Omega')} = 0.\label{eq:L2sense}
\end{equation}
\end{defn}

\begin{brem}\label{rmk2} We know that the regularized
parabolic system $(\ref{eq:CAUCHYDIR})_{1}$ fulfills standard $p $-growth
conditions by virtue of (\ref{eq:bil_form}) and Lemma \ref{lem:lemma app1}.
The existence of a unique weak solution to (\ref{eq:CAUCHYDIR}) can be inferred from the classic existence theory, cf.
\cite[Chapter 2, Theorem 1.2 and Remark 1.2]{Lions}.
%
%
By a difference quotient method, one can show that $Du_\epsilon$ is locally bounded and that $u_{\varepsilon}$ admits second weak spatial derivatives in $L^2_{loc}$, see \cite[Chapter 8]{DiBene}.
For this to hold in the subcritical case $1<p\le \frac{2n}{n +2}$, we additionally have to require that $u_{\varepsilon}$ belongs to $L_{loc}^{r}(Q',\mathbb{R}^{N}$),
where $r\geq2$ satisfies $n(p-2)+rp>0$ (see again \cite[Chapter 8]{DiBene}). Since, in the subcritical case, we always assume that $u_{\varepsilon}\in L_{loc}^{\infty}(Q',\mathbb{R}^{N})$, the latter requirement is trivially fulfilled.
\end{brem} 

\begin{thm}\label{thm:regul} 
Let $p>1$ and $u_{\varepsilon}$
be the weak solution of problem $\mathrm{(\ref{eq:CAUCHYDIR})}$ with
$Q'=Q_{R}(z_{0})\Subset\Omega_{T}$. Moreover, assume that  $u_{\varepsilon}$ satisfies the requirements of Remark~\ref{rmk2} in the case $1<p\le \frac{2n}{n +2}$. Then 
\begin{equation}
Du_{\varepsilon}\in L_{loc}^{2}\left(t_{0}-R^{2},t_{0};W_{loc}^{1,2}\left(B_{R}(x_{0}),\mathbb{R}^{Nn}\right)\right)\cap L_{loc}^{\infty}\left(Q_{R}(z_{0}),\mathbb{R}^{Nn}\right).\label{eq:2ndderivatives}
\end{equation} 
\end{thm}


We now prove the following strong convergence result:

\begin{lem}\label{lem:Lpconv}
Let $p>1$, $\delta>0$, $0<\varepsilon<\min \{\frac{1}{2},(\frac{\delta}{4})^{p-1}\}$
and $f\in L^{2}(\Omega'\times I,\mathbb{R}^{N})$. Moreover, let 
\[u \in C^{0}\left([t_{1},t_{2}];L^{2}\left(\Omega',\mathbb{R}^{N}\right)\right)\cap L^{p}\left(t_{1},t_{2};W^{1,p}\left(\Omega',\mathbb{R}^{N}\right)\right)
\] be a weak solution of $\mathrm{(\ref{eq:syst})}$ and assume
that 
\[
u_{\varepsilon}\in C^{0}\left([t_{1},t_{2}];L^{2}\left(\Omega',\mathbb{R}^{N}\right)\right)\cap L^{p}\left(t_{1},t_{2};W^{1,p}\left(\Omega',\mathbb{R}^{N}\right)\right)
\]
is the unique energy solution of problem $\mathrm{(\ref{eq:CAUCHYDIR})}$.
Then, there exists a constant $\widetilde{C}\equiv\widetilde{C}(p,\hat{n},\delta)\in(0,\min \{\frac{1}{2},(\frac{\delta}{4})^{p-1}\})$
such that for every $\varepsilon\in(0,\widetilde{C})$, the estimate
\begin{align*}
 & \underset{\tau\in I}{\sup}\int_{\Omega'}\vert u_{\varepsilon}-u\vert^{2}(x,\tau) \, dx +\int_{\Omega'\times I}\vert G_{\delta}(Du_{\varepsilon})-G_{\delta}(Du)\vert^{p}\,dz\\
 & \qquad \leq C \varepsilon^{\frac{\hat{n}}{2(\hat{n}+2)}}\int_{\Omega'\times I}(\max \{\vert Du\vert,1+\delta\})^{p}\,dz+C \vert I\vert\int_{(\Omega'\times I)\cap\{ \vert f\vert > \frac{1}{\varepsilon}\} }\vert f\vert^{2}\,dz
\end{align*}
holds for some positive constant $C\equiv C(p,C_{1},\delta)$. In
particular, this estimate implies that 
\[
G_{\delta}(Du_{\varepsilon})\rightarrow G_{\delta}(Du) \quad 
strongly \ in\quad L^{p}(\Omega'\times I,\mathbb{R}^{Nn}),\quad as \quad \varepsilon\rightarrow0^{+}.
\]
\end{lem}

\begin{proof}[\textbf{{Proof}}]
We observe that $(u_{\varepsilon}-u)\in L^{p}(I;W_{0}^{1,p}(\Omega',\mathbb{R}^{N}))$
by the lateral boundary condition. Unfortunately, we cannot test systems
(\ref{eq:syst}) and $(\ref{eq:CAUCHYDIR})_{1}$ with the function
$u_{\varepsilon}-u$, since weak time derivatives might not exist.
Therefore, we resort to the equivalent Steklov averages formulations
of (\ref{eq:syst}) and $(\ref{eq:CAUCHYDIR})_{1}$, thus obtaining
\begin{equation}
    \int_{\Omega'\times\{t\}}
    \big(\partial_{t}[u_{\varepsilon}-u]_{h}\cdot\phi+\langle[A_{\varepsilon}(x,t,Du_{\varepsilon})-A(x,t,Du)]_{h},D\phi\rangle\big)dx=\int_{\Omega'\times\{t\}}[f_{\varepsilon}-f]_{h}\cdot\phi\,dx\label{eq:Stek}
\end{equation}
for every $t\in I=(t_{1},t_{2})$, for every $h\in(0,t_{2}-t_{1})$
and every $\phi\in W_{0}^{1,p}(\Omega',\mathbb{R}^{N})\cap L^{2}(\Omega',\mathbb{R}^{N})$.
Then, for a fixed time slice $\Omega'\times\{t\}$ we can choose the
function $\phi(\cdot,t)$ defined by 
\begin{equation}
\phi^{i}(x,t):=[u_{\varepsilon}^{i}-u^{i}]_{h}(x,t),\qquad i\in\{1,\ldots,N\},\label{eq:tf1}
\end{equation}
as a test function in (\ref{eq:Stek}). For any fixed $\tau\in I$,
the term involving the time derivatives yields\begin{align}\label{eq:est5}
\int_{t_{1}}^{\tau}\int_{\Omega'}
\partial_{t}[u_{\varepsilon}-u]_{h}\cdot\phi\,dx\,dt &= \frac{1}{2}\int_{t_{1}}^{\tau}\int_{\Omega'}\partial_{t}\vert[u_{\varepsilon}-u]_{h}\vert^{2}\,dx\,dt\nonumber\\
&= \frac{1}{2}\int_{\Omega'}\vert[u_{\varepsilon}-u]_{h}(x,\tau)\vert^{2}\,dx-\frac{1}{2}\int_{\Omega'}\vert[u_{\varepsilon}-u]_{h}(x,t_{1})\vert^{2}\,dx
\end{align}  for every $h\in(0,\tau-t_{1})$. Now we pass to the limit as $h\rightarrow0$.
Using Lemma \ref{lem:Stek} and taking into account the growth conditions
of $A_{\varepsilon}$ and the fact that $u_{\varepsilon}=u$ on $\overline{\Omega'}\times\{t_{1}\}$
in the $L^{2}$-sense, from (\ref{eq:Stek}) and \eqref{eq:est5}
we obtain the following inequality 
\begin{align*}
 & \frac{1}{2}\int_{\Omega'}\vert u_{\varepsilon}-u\vert^{2}(x,\tau)\,dx +\int_{\Omega'\times(t_{1},\tau)}\langle A_{\varepsilon}(x,t,Du_{\varepsilon})-A_{\varepsilon}(x,t,Du),Du_{\varepsilon}-Du\rangle\,dx\,dt\\
 & \quad \leq\int_{\Omega'\times(t_{1},\tau)}\vert A(x,t,Du)-A_{\varepsilon}(x,t,Du)\vert\vert Du_{\varepsilon}-Du\vert\,dx\,dt +\int_{\Omega'\times(t_{1},\tau)\cap\{ \vert f\vert > \frac{1}{\varepsilon}\} }\vert f\vert\vert u_{\varepsilon}-u\vert\,dx\,dt
\end{align*}
for every $\tau\in I$. Taking the supremum over $\tau\in I$, we
thus obtain 
\begin{align*}
& \frac{1}{2} \,\underset{\tau\in I}{\sup}\int_{\Omega'}\vert u_{\varepsilon}-u\vert^{2}(x,\tau)\,dx\,+\int_{\Omega'\times I}\langle A_{\varepsilon}(x,t,Du_{\varepsilon})-A_{\varepsilon}(x,t,Du),Du_{\varepsilon}-Du\rangle\,dx\,dt\\
 & \quad \leq\int_{\Omega'\times I}\vert A(x,t,Du)-A_{\varepsilon}(x,t,Du)\vert\vert Du_{\varepsilon}-Du\vert\,dx\,dt +\int_{(\Omega'\times I)\cap\{ \vert f\vert > \frac{1}{\varepsilon}\} }\vert f\vert\vert u_{\varepsilon}-u\vert\,dx\,dt.
\end{align*}
We now apply Young's inequality with exponents $(2,2)$ to control
the last integral as follows
\begin{align*}
\int_{(\Omega'\times I)\cap\{ \vert f\vert > \frac{1}{\varepsilon}\} }\vert f\vert\vert u_{\varepsilon}-u\vert\,dx\,dt &\leq \vert I\vert\int_{(\Omega'\times I)\cap\{ \vert f\vert > \frac{1}{\varepsilon}\} }\vert f\vert^{2}\,dx\,dt+\frac{1}{4 \vert I\vert}\int_{\Omega'\times I}\vert u_{\varepsilon}-u\vert^{2}\,dx\,dt\\
&\leq \vert I\vert\int_{(\Omega'\times I)\cap\{ \vert f\vert > \frac{1}{\varepsilon}\} }\vert f\vert^{2}\,dx\,dt+\frac{1}{4} \,\underset{\tau\in I}{\sup}\int_{\Omega'}\vert u_{\varepsilon}-u\vert^{2}(x,\tau)\,dx.
\end{align*}
Joining the two previous estimates, we get
\begin{align}\label{eq:est12}
& \frac{1}{4} \,\underset{\tau\in I}{\sup}\int_{\Omega'}\vert u_{\varepsilon}-u\vert^{2}(x,\tau)\,dx\,+\int_{\Omega'\times I}\langle A_{\varepsilon}(x,t,Du_{\varepsilon})-A_{\varepsilon}(x,t,Du),Du_{\varepsilon}-Du\rangle\,dx\,dt \nonumber\\
 & \qquad \leq\int_{\Omega'\times I}\vert A(x,t,Du)-A_{\varepsilon}(x,t,Du)\vert\vert Du_{\varepsilon}-Du\vert\,dx\,dt+\vert I\vert\int_{(\Omega'\times I)\cap\{ \vert f\vert > \frac{1}{\varepsilon}\} }\vert f\vert^{2}\,dx\,dt.
\end{align}
In order to estimate the first integral on the right-hand side of
\eqref{eq:est12}, we now distinguish whether or not $p\ge 2$. If $p\ge 2$,
then we have 
\[
A_{\varepsilon}(x,t,Du)=A(x,t,Du)+\varepsilon (1+\vert Du\vert^{2})^{\frac{p-2}{2}}Du,
\]
which implies 
\begin{align*}
& \frac{1}{4} \,\underset{\tau\in I}{\sup}\int_{\Omega'}\vert u_{\varepsilon}-u\vert^{2}(x,\tau)\,dx\,+\int_{\Omega'\times I}\langle A_{\varepsilon}(x,t,Du_{\varepsilon})-A_{\varepsilon}(x,t,Du),Du_{\varepsilon}-Du\rangle\,dx\,dt\\
 & \qquad \leq\varepsilon\int_{\Omega'\times I}(1+\vert Du\vert^{2})^{\frac{p-2}{2}}\vert Du\vert\vert Du_{\varepsilon}-Du\vert\,dx\,dt+\vert I\vert\int_{(\Omega'\times I)\cap\{ \vert f\vert > \frac{1}{\varepsilon}\} }\vert f\vert^{2}\,dx\,dt.
\end{align*}
In what follows, we will denote by $C$ a general positive constant
that only depends on $p$, $C_{1}$ and $\delta$. Using the previous estimate, Lemma \ref{lem:Glemma} with $\nu:=\frac{\hat{n}}{2(\hat{n}+2)}<\frac{1}{2}$ and Young's inequality
with exponents $(p,\frac{p}{p-1})$,
we have
\begin{align} 
& \frac{1}{4} \,\underset{\tau\in I}{\sup}\int_{\Omega'}\vert u_{\varepsilon}-u\vert^{2}(x,\tau)\,dx+\int_{\Omega'\times I}\vert G_{\delta}(Du_{\varepsilon})-G_{\delta}(Du)\vert^{p}\,dx\,dt\nonumber\\
 & \qquad \leq C \varepsilon^{\nu}\int_{\Omega'\times I}(\max \{\vert Du\vert,1+\delta\})^{p}\,dx\,dt+\varepsilon^{\nu}\int_{\Omega'\times I}\vert Du_{\varepsilon}-Du\vert^{p}\,dx\,dt\nonumber\\
 & \qquad \qquad +
 C \vert I\vert\int_{(\Omega'\times I)\cap\{ \vert f\vert > \frac{1}{\varepsilon}\} }\vert f\vert^{2}\,dx\,dt,
\label{eq:est15}
\end{align}
where we have also exploited the facts that $\varepsilon<\varepsilon^{1-\nu}$ and $\varepsilon^{1-\nu}<\varepsilon^{\nu}$, since $0<\varepsilon<\frac{1}{2}$. 

In the case $1<p<2$, we need to use (\ref{eq:lab1}) and (\ref{eq:lab2}),
which imply 
\[
\vert A_{\varepsilon}(x,t,\xi)-A(x,t,\xi)\vert \leq \frac{\vert\partial_{s}F_{\varepsilon}(x,t,\vert\xi\vert)-\partial_{s}F(x,t,\vert\xi\vert)\vert}{\vert\xi\vert}+\varepsilon (1+\vert\xi\vert^{2})^{\frac{p-2}{2}}\vert\xi\vert\leq
\varepsilon(2^{p} C_{1}+1) \vert\xi\vert^{p-1}
\]
for every $\xi\in\mathbb{R}^{Nn}\setminus\{0\}$ and for almost every
$(x,t)\in\Omega'\times I$. Using the above estimate with $\xi=Du$
and arguing as in the case $p>2$, we now obtain from \eqref{eq:est12} that
\begin{align} 
& \frac{1}{4} \,\underset{\tau\in I}{\sup}\int_{\Omega'}\vert u_{\varepsilon}-u\vert^{2}(x,\tau)\,dx+\int_{\Omega'\times I}\vert G_{\delta}(Du_{\varepsilon})-G_{\delta}(Du)\vert^{p}\,dz\nonumber\\
 & \qquad\leq C \varepsilon^{\nu}\int_{\Omega'\times I}(\max \{\vert Du\vert,1+\delta\})^{p}\,dz+\varepsilon^{\nu}\int_{\Omega'\times I}\vert Du_{\varepsilon}-Du\vert^{p}\,dz\nonumber\\
 & \qquad \qquad +
 C \vert I\vert\int_{(\Omega'\times I)\cap\{ \vert f\vert > \frac{1}{\varepsilon}\} }\vert f\vert^{2}\,dz
\label{eq:est16}
\end{align}
for any $0<\varepsilon<\min \{\frac{1}{2},(\frac{\delta}{4})^{p-1}\}$.
Now, from the proof of \cite[Lemma 2.3]{BDGP} we have that 
\[
\vert Du_{\varepsilon}-Du\vert^{p} \leq\left(\frac{\delta+2}{\delta}\right)^{p}\vert G_{\delta}(Du_{\varepsilon})-G_{\delta}(Du)\vert^{p}
\]
if either $\vert Du_{\varepsilon}\vert>1+\delta$ or $\vert Du\vert>1+\delta$.
Using this information to estimate the right-hand side of both \eqref{eq:est15}
and \eqref{eq:est16}, for every $p>1$ we get
\begin{align*}
&\frac{1}{4} \,\underset{\tau\in I}{\sup}\int_{\Omega'}\vert u_{\varepsilon}-u\vert^{2}(x,\tau)\,dx+\int_{\Omega'\times I}\vert G_{\delta}(Du_{\varepsilon})-G_{\delta}(Du)\vert^{p}\,dz\\
&\qquad\leq C \varepsilon^{\nu}\int_{\Omega'\times I}(\max \{\vert Du\vert,1+\delta\})^{p}\,dz+\varepsilon^{\nu} \,2^{p} \,\vert\Omega'\times I\vert (1+\delta)^{p}\\
&\qquad\qquad +
\varepsilon^{\nu} \left(\frac{\delta+2}{\delta}\right)^{p}\int_{\Omega'\times I}\vert G_{\delta}(Du_{\varepsilon})-G_{\delta}(Du)\vert^{p}\,dz+C \vert I\vert\int_{(\Omega'\times I)\cap\{ \vert f\vert > \frac{1}{\varepsilon}\} }\vert f\vert^{2}\,dz\\
&\qquad\leq (2^{p}+C) \varepsilon^{\nu}\int_{\Omega'\times I}(\max \{\vert Du\vert,1+\delta\})^{p}\,dz+\varepsilon^{\nu}\left(\frac{\delta+2}{\delta}\right)^{p}\int_{\Omega'\times I}\vert G_{\delta}(Du_{\varepsilon})-G_{\delta}(Du)\vert^{p}\,dz\\
&\qquad \qquad +
C \vert I\vert\int_{(\Omega'\times I)\cap\{ \vert f\vert > \frac{1}{\varepsilon}\} }\vert f\vert^{2}\,dz.
\end{align*}
At this point, notice that 
\begin{equation}
\varepsilon^{\nu}\left(\frac{\delta+2}{\delta}\right)^{p} \searrow0 \qquad \mathrm{as} \qquad \varepsilon\searrow0,\label{eq:limit1}
\end{equation}
since $\nu:=\frac{\hat{n}}{2(\hat{n}+2)}>0$. Also, recall that $\varepsilon$
is small enough to have $0<\varepsilon<\min \{\frac{1}{2},(\frac{\delta}{4})^{p-1}\}$.
From this and from (\ref{eq:limit1}) it follows that there exists
a constant $\widetilde{C}\equiv\widetilde{C}(p,\hat{n},\delta)\in(0,\min \{\frac{1}{2},(\frac{\delta}{4})^{p-1}\})$
such that, for every $\varepsilon\in(0,\widetilde{C})$, we have 
\[
\varepsilon^{\nu}\left(\frac{\delta+2}{\delta}\right)^{p}\leq \frac{1}{2}\,. 
\]
This allows to absorb the second term on the right-hand side of the last estimate into the left-hand side, so that we finally obtain the desired conclusion. 
\end{proof}

We are now in a position to prove the uniqueness
of weak solutions to the Cauchy-Dirichlet problem
\begin{equation}
\begin{cases}
\partial_{t}u-\dive A(x,t,Du)=f & \mathrm{in}\,\,\, \Omega_{T},\\
u=g & \mathrm{on}\,\,\, \partial_{\mathrm{par}}\Omega_{T},
\end{cases}\label{eq:Cdp}
\end{equation}
where $\Omega$ is a bounded domain in $\mathbb{R}^{n}$, $g\in C^{0}([0,T];L^{2}(\Omega$,$\mathbb{R}^{N}))\cap L^{p}(0,T;W^{1,p}(\Omega,\mathbb{R}^{N}))$
and $f\in L^{2}(\Omega_{T},\mathbb{R}^{N})$. 

\begin{thm}\label{thm:THunique} 
Let $\Omega\subset\mathbb{R}^{n}$
be a bounded domain and let $f\in L^{2}(\Omega_{T},\mathbb{R}^{N})$.
Then, the Cauchy-Dirichlet problem $(\ref{eq:Cdp})$ admits at most
one weak solution. 
\end{thm}

\begin{proof}[\textbf{{Proof}}]
Let $u\in C^{0}([0,T];L^{2}(\Omega,\mathbb{R}^{N}))\cap L^{p}(0,T;W^{1,p}(\Omega,\mathbb{R}^{N}))$
be a weak solution of (\ref{eq:Cdp}) and let $u_{\varepsilon}\in C^{0}([0,T];L^{2}(\Omega,\mathbb{R}^{N}))\cap L^{p}(0,T;W^{1,p}(\Omega,\mathbb{R}^{N}))$
be the unique weak solution to (\ref{eq:CAUCHYDIR}) with
$Q':=\Omega'\times I=\Omega_{T}$. Then, by Lemma \ref{lem:Lpconv},
for every $\delta>0$ there exists a constant $\widetilde{C}\in(0,\min \{\frac{1}{2},(\frac{\delta}{4})^{p-1}\})$
such that for any $\varepsilon\in(0,\widetilde{C})$ we have 
\begin{align}
\sup_{\tau\in(0,T)}\int_{\Omega}\vert u_{\varepsilon}-u\vert^{2}(x,\tau)\,dx  & 
\leq
C \varepsilon^{\frac{\hat{n}}{2(\hat{n}+2)}}\int_{\Omega_{T}}(\max\{|Du|,1+\delta\})^{p} \,dz\nonumber\\
 & \qquad+
 C T\int_{\Omega_{T} \cap \{ \vert f\vert > \frac{1}{\varepsilon}\} }\vert f\vert^{2}\,dz
\label{eq:est19}
\end{align}
for some positive constant $C$ that is independent of $\varepsilon$.
Notice that the right-hand side of \eqref{eq:est19} converges to
zero as $\varepsilon\rightarrow0^{+}$. Moreover, we have 
\begin{equation}
\int_{\Omega_{T}}|u_{\varepsilon}-u|^{2}\,dz \leq T\sup\limits _{\tau\in(0,T)}\int_{\Omega}|u_{\varepsilon}-u|^{2}(x,\tau)\,dx.\label{eq:est20}
\end{equation}
Combining estimates \eqref{eq:est19} and (\ref{eq:est20}), we infer
that $u_{\varepsilon}\rightarrow u$ strongly in $L^{2}(\Omega_{T},\mathbb{R}^{N})$
as $\varepsilon\rightarrow0^{+}$. This implies the uniqueness of
the weak solutions to problem (\ref{eq:Cdp}), by virtue of the uniqueness
of limits in Lebesgue spaces and the uniqueness of the energy solutions
to problem (\ref{eq:CAUCHYDIR}). 
\end{proof}

\section{Maximum principle for the homogeneous system\label{sec:comparison}}

In this section, we want to establish a
maximum principle for the homogeneous system in the case $1<p\leq\frac{2n}{n+2}$. For the proof, we need to assume that $f=0$ and note that the assumed boundedness of $u_\varepsilon$ in Remark \ref{rmk2} is implied by the maximum principle. Therefore, this assumption is not restrictive whenever $f \equiv 0$. By (\ref{eq:f_eps}), we then
have $f_{\varepsilon}=0$ for every $\varepsilon\in(0,\frac{1}{2})$.
Keeping in mind the notation introduced in Section \ref{sec:uniqueness}, we now
set 
\[
k:=\Vert u\Vert_{L^{\infty}(\Omega'\times I)}
\]
and denote by $w$ the $N$-dimensional vector whose components are
all equal to $k$. Notice that $w$ is a trivial solution of system
(\ref{eq:syst}) with $f=0$. Moreover, one can easily check that,
for every $i\in\{1,\ldots,N\}$ and for almost every $t\in I:=(t_{1},t_{2})$,
we have
\[
\left(u_{\varepsilon}^{i}(\cdot,t)-k\right)_{+},\left(u_{\varepsilon}^{i}(\cdot,t)+k\right)_{-}\in W_{0}^{1,p}(\Omega').
\]

Now, for $h\in(0,t_{2}-t_{1})$ we define 
\begin{equation}
\phi^{i}(x,t):=[(u_{\varepsilon}^{i}-k)_{+}]_{h}(x,t),\,\,\,\,\,\,\,\,\,i\in\{1,\ldots,N\}.\label{eq:tf2}
\end{equation}
Using the Steklov averages formulations of $(\ref{eq:CAUCHYDIR})_{1}$
and (\ref{eq:syst}) with $w$ instead of $u$, and arguing as in
the proof of Lemma \ref{lem:Lpconv} with (\ref{eq:tf2}) in place
of (\ref{eq:tf1}), one can easily obtain 
\[
\frac{1}{2}\int_{\Omega'}\vert(u_{\varepsilon}(x,\tau)-w)_{+}\vert^{2}\,dx+\int_{\Omega'\times(t_{1},\tau)}\langle A_{\varepsilon}(x,t,Du_{\varepsilon})-A_{\varepsilon}(x,t,Dw),D[(u_{\varepsilon}-w)_{+}]\rangle\,dx\,dt=0
\]
for every $\tau\in I$. Since $A_{\varepsilon}(x,t,\cdot)$ is a monotone
vector field and $Dw=A_{\varepsilon}(x,t,Dw)=0$, we can omit the latter integral, thus obtaining
\[
\int_{\Omega'}\vert(u_{\varepsilon}(x,\tau)-w)_{+}\vert^{2}\,dx = 0 \qquad \text{for every }\tau\in I.
\]
Therefore, for every $i\in\{1,\ldots,N\}$ and every $\varepsilon\in(0,\frac{1}{2})$,
we have 
\[
u_{\varepsilon}^{i}(x,\tau)\leq k\qquad \text{for almost every } (x,\tau)\in Q':=\Omega'\times I.
\]
Similarly, but using the function $(u_{\varepsilon}^{i}+k)_{-}$ instead
of $(u_{\varepsilon}^{i}-k)_{+}$ in (\ref{eq:tf2}), we get 
\[
u_{\varepsilon}^{i}(x,\tau) \geq-k
\]
for any $i\in\{1,\ldots,N\}$, for any $\varepsilon\in(0,\frac{1}{2})$
and for almost every $(x,\tau)\in Q'$. Thus, we can conclude that
\[
\Vert u_{\varepsilon}\Vert_{L^{\infty}(\Omega'\times I)}
\leq
\sqrt{N} k
=
\sqrt{N}\Vert u\Vert_{L^{\infty}(\Omega'\times I)}\qquad \text{for all } \varepsilon\in\left(0,\tfrac{1}{2}\right).
\]

\section{Weak differentiability\label{sec:weak diff}}

Here we derive
some higher differentiability results that will be useful in the following.
These results involve the spatial gradient of the weak solution to
problem (\ref{eq:CAUCHYDIR}) with $Q'=Q_{R}(z_{0})\Subset\Omega_{T}$.

To begin with, for each
$\gamma\geq0$ and $a>0$ we consider the function 
\begin{equation}\label{eq:bigphi}
\Phi_{\gamma,a}(w):= w^{2} (a+w)^{\gamma-2},\qquad w\geq0. 
\end{equation}
For this function, one can easily check that 
\begin{equation}
\Phi_{\gamma,a}'(w)\leq2 (\gamma+1) w (a+w)^{\gamma-2},\label{eq:ineqPhi}
\end{equation}
from which we can immediately deduce 
\begin{equation}
w \,\Phi'_{\gamma,a}(w)\leq2 (\gamma+1) \Phi_{\gamma,a}(w)\label{PhiDt}
\end{equation}
and 
\[
\Phi'_{\gamma,a}(w)\leq(2a)^{\gamma+1}(\gamma+1)\qquad\text{for every }w\in(0,a).
\]
Using these results, we obtain the following lemma.

\begin{lem}\label{lem:L3}
Let $p>1$,
$\varepsilon\in(0,\frac{1}{2})$, $\gamma\geq0$ and $a>0$. Moreover,
assume that 
\[
u_{\varepsilon}\in C^{0}\left([t_{0}-R^{2},t_{0}];L^{2}\left(B_{R}(x_{0}),\mathbb{R}^{N}\right)\right)\cap L^{p}\left(t_{0}-R^{2},t_{0};W^{1,p}\left(B_{R}(x_{0}),\mathbb{R}^{N}\right)\right)
\]
is the unique energy solution of problem $(\ref{eq:CAUCHYDIR})$ with
$Q'=Q_{R}(z_{0}):=B_{R}(x_{0})\times(t_{0}-R^{2},t_{0})\Subset\Omega_{T}$.
Then 
\[
\Phi_{\gamma,a}((|Du_{\varepsilon}|-a)_{+}) \in L_{loc}^{2}\left(t_{0}-R^{2},t_{0};W_{loc}^{1,2}\left(B_{R}(x_{0})\right)\right).
\]
In the case $1<p\leq\frac{2n}{n +2}$ we additionally assume that $u_{\varepsilon}\in L_{loc}^{\infty}\left(Q_{R}(z_{0}),\mathbb{R}^{N}\right)$.
\end{lem}

\begin{proof}[\textbf{{Proof}}]
Notice that for $m\in\mathbb{R}^{+}$
the function 
\[
\Phi_{\gamma,a,m}(w):=\min \{\Phi_{\gamma,a}(w),m\},\qquad w>0,
\]
is Lipschitz continuous. By Theorem \ref{thm:regul} (see also Remark
\ref{rmk2} for the case $1<p\leq\frac{2 n}{n +2}$), we have that 
\[
Du_{\varepsilon}\in L_{loc}^{2}(t_{0}-R^{2},t_{0};W_{loc}^{1,2}(B_{R}(x_{0}),\mathbb{R}^{Nn})).
\]
Then, by virtue of the chain rule in Sobolev spaces, we obtain 
\[
\Phi_{\gamma,a,m}((|Du_{\varepsilon}|-a)_{+}) \in L_{loc}^{2}(t_{0}-R^{2},t_{0};W_{loc}^{1,2}(B_{R}(x_{0})))
\]
along with 
\[
D\big(\Phi_{\gamma,a,m}((|Du_{\varepsilon}|-a)_{+})\big) = \Phi'_{\gamma,a,m}((|Du_{\varepsilon}|-a)_{+}) \frac{D^{2}u_{\varepsilon}\cdot Du_{\varepsilon}}{|Du_{\varepsilon}|}  \mathbb{\mathds{1}}_{\{a < |Du_{\varepsilon}| < m\}}.
\]
Theorem \ref{thm:regul} also tells us that $Du_{\varepsilon}\in L_{loc}^{\infty}(Q_{R}(z_{0}),\mathbb{R}^{Nn})$.
This implies that for every compact subset $\mathcal{K}$ of $Q_{R}(z_{0})$
there exists a positive number $m=m_{\mathcal{K}}$ such that 
\[
\Phi_{\gamma,a,m}((|Du_{\varepsilon}|-a)_{+}) = \Phi_{\gamma,a}((|Du_{\varepsilon}|-a)_{+}) \qquad \text{in }  \mathcal{K}.
\]
This is sufficient to prove the assertion. Moreover, we have 
\[
D(\Phi_{\gamma,a}((|Du_{\varepsilon}|-a)_{+})) = \Phi'_{\gamma,a}((|Du_{\varepsilon}|-a)_{+}) \frac{D^{2}u_{\varepsilon}\cdot Du_{\varepsilon}}{|Du_{\varepsilon}|}  \mathbb{\mathds{1}}_{\{|Du_{\varepsilon}| > a\}}.
\]
\end{proof}

Now we focus on
the map 
\[
(x,t)\in\Omega_{T} \longmapsto A_{\varepsilon}(x,t,Du_{\varepsilon}(x,t))\in\mathbb{R}^{Nn}.
\]
Using the notation (\ref{eq:molli}), we
obtain the following result:

\begin{lem}\label{lem:L4} 
Let $p>1$,
$\delta>0$ and $0<\varepsilon<\min \{\frac{1}{2},(\frac{\delta}{4})^{p-1}\}$.
Moreover, assume that 
\[
u_{\varepsilon}\in C^{0}\left([t_{0}-R^{2},t_{0}];L^{2}\left(B_{R}(x_{0}),\mathbb{R}^{N}\right)\right)\cap L^{p}\left(t_{0}-R^{2},t_{0};W^{1,p}\left(B_{R}(x_{0}),\mathbb{R}^{N}\right)\right)
\]
is the unique energy solution of problem $(\ref{eq:CAUCHYDIR})$ with
$Q'=Q_{R}(z_{0}):=B_{R}(x_{0})\times(t_{0}-R^{2},t_{0})\Subset\Omega_{T}$.
Then 
\[
A_{\varepsilon}(\cdot,Du_{\varepsilon}) \in L_{loc}^{2}\left(t_{0}-R^{2},t_{0};W_{loc}^{1,2}(B_{R}(x_{0}),\mathbb{R}^{Nn})\right).
\]
In the case $1<p\leq\frac{2n}{n +2}$ we additionally assume that $u_{\varepsilon}\in L_{loc}^{\infty}\left(Q_{R}(z_{0}),\mathbb{R}^{N}\right)$. 
\end{lem}

\begin{proof}[\textbf{{Proof}}]
Integrating by parts, for $0<\varrho\ll1$, for every $\phi\in C_{0}^{\infty}(Q_{R}(z_{0}))$
and every $i\in\{1,\ldots,n\}$ we obtain 
\begin{align}\int_{Q_{R}(z_{0})}A_{\varepsilon}(x,t,(Du_{\varepsilon})_{\varrho}) D_{i}\phi\,dz= &  -\int_{Q_{R}(z_{0})}(D_{x_{i}}A_{\varepsilon})(x,t,(Du_{\varepsilon})_{\varrho})\, \phi\,dz\nonumber\\
 & \qquad -\int_{Q_{R}(z_{0})}(D_{\xi}A_{\varepsilon})(x,t,(Du_{\varepsilon})_{\varrho})\cdot D_{i}(Du_{\varepsilon})_{\varrho}\, \phi\,dz.
\label{eq:est21}
\end{align}
In order to pass to the limit as $\varrho\searrow0$ under the integral signs, we need to estimate the above integrands. From
(\ref{eq:gro1}) it follows that 
\begin{align}
|A_{\varepsilon}(x,t,(Du_{\varepsilon})_{\varrho})| |D_{i}\phi| &  \leq c (1+|(Du_{\varepsilon})_{\varrho}|^{2})^{\frac{p-1}{2}} |D_{i}\phi|\label{eq:est22}
\end{align}
for some positive constant $c\equiv c(p,C_{1},\varepsilon)$. As for
the first integrand on the right-hand side of \eqref{eq:est21},
thanks to (\ref{eq:spgraAeps}) we get 
\begin{align}
|(D_{x_{i}}A_{\varepsilon})(x,t,(Du_{\varepsilon})_{\varrho})| |\phi| &  \leq 2^{p-1} K (1+|(Du_{\varepsilon})_{\varrho}|)^{p-1}|\phi|.\label{eq:est23}
\end{align}
Moreover, by Lemma \ref{lem:lemma app2} we obtain 
\begin{align}
|(D_{\xi}A_{\varepsilon})(x,t,(Du_{\varepsilon})_{\varrho})\cdot D_{i}(Du_{\varepsilon})_{\varrho}| |\phi| &  \leq C (1+|(Du_{\varepsilon})_{\varrho}|^{2})^{\frac{p-2}{2}} |D^{2}(u_{\varepsilon})_{\varrho}| |\phi|,\label{eq:est24}
\end{align}
where $C\equiv C(p,C_{1},\varepsilon)>0$. Now we use a well-known
result on the convergence of mollified functions to deduce that 
\[
(Du_{\varepsilon})_{\varrho}\rightarrow Du_{\varepsilon}\quad\mathrm{in}\quad L_{loc}^{2}(t_{0}-R^{2},t_{0};W_{loc}^{1,2}(B_{R}(x_{0}),\mathbb{R}^{Nn}))\cap L^{p}(Q_{R}(z_{0}),\mathbb{R}^{Nn}), \quad \mathrm{as} \quad \varrho\rightarrow0^{+}.
\]
Combining this with estimates (\ref{eq:est22})$-$(\ref{eq:est24})
and applying the Generalized Lebesgue's
Dominated Convergence Theorem (see \cite[page 92, Theorem 17]{Roy}) to both sides of \eqref{eq:est21},
we find that 
\begin{align}\int_{Q_{R}(z_{0})}A_{\varepsilon}(x,t,Du_{\varepsilon}) D_{i}\phi\,dz= 
& -
\int_{Q_{R}(z_{0})}(D_{x_{i}}A_{\varepsilon})(x,t,Du_{\varepsilon}) \phi\,dz
\nonumber\\& \qquad -
\int_{Q_{R}(z_{0})}(D_{\xi}A_{\varepsilon})(x,t,Du_{\varepsilon})\cdot D_{i}Du_{\varepsilon} \,\phi\,dz
\label{eq:est25}
\end{align}
for every $\phi\in C_{0}^{\infty}(Q_{R}(z_{0}))$ and every $i\in\{1,\ldots,n\}$.
This implies the assertion. In addition, from \eqref{eq:est25} we
obtain that 
\[
D_{i}A_{\varepsilon}(x,t,Du_{\varepsilon})=(D_{x_{i}}A_{\varepsilon})(x,t,Du_{\varepsilon})+(D_{\xi}A_{\varepsilon})(x,t,Du_{\varepsilon})\cdot D_{i}Du_{\varepsilon} \qquad \text{for every }i\in\{1,\ldots,n\}.
\]
\end{proof}

For further needs,
we now introduce the auxiliary function $H_{\lambda}:\mathbb{R}^{Nn}\rightarrow\mathbb{R}^{+}$
defined by 
\[
H_{\lambda}(\xi):=\max \{1+\lambda,|\xi|\},
\]
where $\lambda>0$ is a parameter. For this function we record the
following result, whose proof is omitted, since it is similar to that
of Lemma \ref{lem:L3}.

\begin{lem}
Let $p>1$, $\varepsilon\in(0,\frac{1}{2})$
and $\lambda>0$. Moreover, assume that 
\[
u_{\varepsilon}\in C^{0}\left([t_{0}-R^{2},t_{0}];L^{2}\left(B_{R}(x_{0}),\mathbb{R}^{N}\right)\right)\cap L^{p}\left(t_{0}-R^{2},t_{0};W^{1,p}\left(B_{R}(x_{0}),\mathbb{R}^{N}\right)\right)
\]
is the unique energy solution of problem $(\ref{eq:CAUCHYDIR})$ with
$Q'=Q_{R}(z_{0}):=B_{R}(x_{0})\times(t_{0}-R^{2},t_{0})\Subset\Omega_{T}$.
Then 
\[
H_{\lambda}(Du_{\varepsilon}) \in L_{loc}^{2}\left(t_{0}-R^{2},t_{0};W_{loc}^{1,2}\left(B_{R}(x_{0})\right)\right).
\]
In the case $1<p\leq\frac{2n}{n +2}$ we additionally assume that $u_{\varepsilon}\in L_{loc}^{\infty}\left(Q_{R}(z_{0}),\mathbb{R}^{N}\right)$.
\end{lem}

\section{Local boundedness of \texorpdfstring{$Du_{\varepsilon}$}{Du}}\label{sec:Local}

As before, by $u$ we denote a weak solution of (\ref{eq:syst}) and we let $u_\epsilon$ be the unique weak solution of
the regularized problem $(\ref{eq:CAUCHYDIR})$ with $Q'=Q_{R}(z_{0}):=B_{R}(x_{0})\times(t_{0}-R^{2},t_{0})\Subset\Omega_{T}$.
Our aim in this section is to establish local $L^{\infty}$ estimates for $Du_{\varepsilon}$ with constants independent of $\varepsilon$. We will
achieve this result by using the Moser iteration technique, which
is based on Caccioppoli-type inequalities. We will obtain this kind
of inequalities by first differentiating the system of differential
equations, and then testing the resulting equation with a suitable
power of the weak solution itself. The groundwork for doing this has
been laid in Section \ref{sec:weak diff}.

To move forward, we now fix $\delta>0$ and $\gamma\geq0$.
Moreover, to shorten our notation we set 
\[
a:=1+\frac{\delta}{2}\qquad ,\qquad b:=1+\delta ,
\]
and we drop the subscript $\varepsilon$ for the weak solution $u_{\varepsilon}$
and the subscripts $\gamma,a$ for the function $\Phi_{\gamma,a}$
defined in (\ref{eq:bigphi}). Therefore, from now on we will simply
write $u$ and $\Phi$ in place of $u_{\varepsilon}$ and $\Phi_{\gamma,a}$
respectively, unless otherwise stated. To simplify our notation even
more, we additionally introduce the function $P:Q_{R}(z_{0})\rightarrow\mathbb{R}_{0}^{+}$
defined by 
\[
P:=(\vert Du\vert-a)_{+} ,
\]
and its ``mollified'' version 
\begin{equation}
P_{\varrho}:=(\vert Du_{\varrho}\vert-a)_{+} ,\qquad \varrho>0,\label{eq:P_rho}
\end{equation}
with an intentional abuse of the notation (\ref{eq:molli}) on the
left-hand side of (\ref{eq:P_rho}). By the elementary properties
of Sobolev functions, the weak spatial gradient of $P$ is given by
\begin{equation}
DP = \frac{D^{2}u\cdot Du}{|Du|}  \mathds{1}_{\{\vert Du\vert > a\}} .\label{eq:gradP}
\end{equation}

\subsection{Step 1: Caccioppoli--type inequalities}\label{subsec:Caccioppoli}

In order to prove the local boundedness
of $Du$, we now test the weak formulation of system $(\ref{eq:CAUCHYDIR})_{1}$
with the function $D_{i} \tilde{\phi}_{\varrho}$, where
\[
\tilde{\phi}_{\varrho}\equiv\tilde{\phi}_{i,\varrho}:=D_{i} u_{\varrho}\cdot\Phi(P_{\varrho}) \psi,\qquad  i\in\{1,\ldots,n\},
\]
and $\psi\in W_{0}^{1,\infty}(Q_{R}(z_{0}))$ is a non-negative cut-off
function that will be specified later. We thus obtain 
\begin{equation}
\int_{Q_{R}(z_{0})}\sum_{i=1}^{n}\left(\partial_{t}(D_{i} u_{\varrho})\cdot\tilde{\phi}_{i,\varrho}+\langle D_{i}[A_{\varepsilon}(x,t,Du)]_{\varrho},D\tilde{\phi}_{i,\varrho}\rangle+(f_{\varepsilon})_{\varrho}\cdot D_{i} \tilde{\phi}_{i,\varrho}\right)dz=0.\label{eq:wefo3}
\end{equation}
At this stage, we fix $z_{1}=(x_{1},t_{1})\in Q_{R}(z_{0})$ and $r\in(0,1)$
such that $Q_{r}(z_{1})\Subset Q_{R}(z_{0})$. Moreover, we choose
$\psi(x,t)=\chi(t) \omega^{2}(t) \eta^{2}(x)$ with $\chi,\omega\in W^{1,\infty}(\mathbb{R},[0,1])$,
$\chi(t_{1})=0$, $\partial_{t}\chi\leq0$, $\omega(t_{1}-r^{2})=0$,
$\partial_{t}\omega\geq0$ and $\eta\in C_{0}^{\infty}(B_{r}(x_{1}),[0,1])$.
With such a choice of $\psi$ and integrating by parts, for the term
involving the time derivative we get 
\begin{align}
J_{0,\varrho} & :=\int_{Q_{R}(z_{0})}\sum_{i=1}^{n}\partial_{t}(D_{i} u_{\varrho})\cdot D_{i} u_{\varrho} \,\Phi(P_{\varrho}) \psi\,dx\,dt = \frac{1}{2}\int_{Q_{R}(z_{0})}\partial_{t}\vert Du_{\varrho}\vert^{2} \Phi(P_{\varrho}) \psi\,dx\,dt\nonumber\\
 & = \frac{1}{2}\int_{Q_{R}(z_{0})}\partial_{t}\left[\int_{0}^{\vert Du_{\varrho}\vert^{2}}\Phi((\sqrt{w}-a)_{+})\,dw\right]\psi\,dx\,dt\nonumber\\
 & =-\frac{1}{2}\int_{Q_{R}(z_{0})}\int_{0}^{\vert Du_{\varrho}\vert^{2}}\Phi((\sqrt{w}-a)_{+})\,dw\,  \partial_{t}\chi(t)\omega^{2}(t)\eta^{2}(x)\,dx\,dt\nonumber\\
 & \qquad -\int_{Q_{R}(z_{0})}\int_{0}^{\vert Du_{\varrho}\vert^{2}}\Phi((\sqrt{w}-a)_{+})\,dw\,  \chi(t)\omega(t)\partial_{t}\omega(t)\eta^{2}(x)\,dx\,dt.
\label{eq:est26}
\end{align}
Now, for $\gamma>0$ we estimate the inner integral from above and
below as follows: 
\begin{align*} 
& \int_{0}^{\vert Du_{\varrho}\vert^{2}}\Phi((\sqrt{w}-a)_{+})\,dw=\int_{a^{2}}^{\vert Du_{\varrho}\vert^{2}}(\sqrt{w}-a)_{+}^{2} w^{\frac{\gamma-2}{2}}dw=2\int_{a}^{\vert Du_{\varrho}\vert}(w-a)_{+}^{2} w^{\gamma-1}dw\nonumber\\
 & \qquad \leq 2\int_{a}^{\vert Du_{\varrho}\vert}w^{\gamma+1}dw\cdot\mathds{1}_{\{|Du_{\varrho}|>a\}}=\frac{2}{\gamma+2} |Du_{\varrho}|^{\gamma+2} \mathds{1}_{\{|Du_{\varrho}|>a\}}\leq\frac{2}{\gamma+2} (H_{\frac{\delta}{2}}(Du_{\varrho}))^{\gamma+2}
\end{align*}
and 
\begin{align*} 
& \int_{0}^{\vert Du_{\varrho}\vert^{2}}\Phi((\sqrt{w}-a)_{+})\,dw=2\int_{a}^{\vert Du_{\varrho}\vert}(w-a)_{+}^{2} w^{\gamma-1}dw\geq2\int_{b}^{\vert Du_{\varrho}\vert}\frac{(w-a)^{2}}{w^{2}} w^{\gamma+1}\,dw\cdot\mathds{1}_{\{|Du_{\varrho}|>b\}}\nonumber\\
 & \qquad \geq \frac{\delta^{2}}{2 b^{2}(\gamma+2)} (|Du_{\varrho}|^{\gamma+2}-b^{\gamma+2}) \mathds{1}_{\{|Du_{\varrho}|>b\}}= \frac{\delta^{2}}{2 b^{2}(\gamma+2)} (|Du_{\varrho}|^{\gamma+2}-b^{\gamma+2})_{+} .
\end{align*}
Recalling that $\partial_{t}\chi\leq0$, $\partial_{t}\omega\geq0$
and $\chi,\omega\geq0$, and plugging the two previous inequalities into \eqref{eq:est26}, we obtain 
\begin{align*}
J_{0,\varrho} &  \geq-\frac{\delta^{2}}{4 b^{2}(\gamma+2)}\int_{Q_{r}(z_{1})}(|Du_{\varrho}|^{\gamma+2}-b^{\gamma+2})_{+} \partial_{t}\chi(t)\omega^{2}(t)\eta^{2}(x)\,dx\,dt\\
 & \qquad -\frac{2}{\gamma+2}\int_{Q_{r}(z_{1})}(H_{\frac{\delta}{2}}(Du_{\varrho}))^{\gamma+2} \chi(t)\omega(t)\partial_{t}\omega(t)\eta^{2}(x)\,dx\,dt.
\end{align*}
Inserting this into (\ref{eq:wefo3}) and letting $\varrho\searrow0$
yields the following inequality 
\begin{equation}
J_{0}+J_{1}+J_{2}+J_{3}+J_{4}+J_{5}+J_{6}+J_{7}+J_{8}+J_{9}\leq0,\label{eq:Jinequality}
\end{equation}
where the terms $J_{0}$ -- $J_{9}$ are defined by
\begin{align*}
J_{0} & :=- \frac{\delta^{2}}{4 b^{2}(\gamma+2)}\int_{Q_{r}(z_{1})} \big(|Du|^{\gamma+2}-b^{\gamma+2} \big)_{+}\ \partial_{t}\chi \,\omega^{2} \eta^{2}\,dz\\
 & \quad \,\, -\frac{2}{\gamma+2}\int_{Q_{r}(z_{1})} \big(H_{\frac{\delta}{2}}(Du) \big)^{\gamma+2} \chi \omega \,\partial_{t}\omega \,\eta^{2}\,dz,\\
J_{1} & :=\int_{Q_{r}(z_{1})}\mathcal{A}_{\varepsilon}(x,t,Du)(D^{2}u,D^{2}u) \Phi(P) \psi\,dz,\\
J_{2} & :=\int_{Q_{r}(z_{1})}\sum_{i=1}^{n}\sum_{j=1}^{N}\sum_{k=1}^{n} \big\langle D_{\xi_{k}^{j}}A_{\varepsilon}(x,t,Du)D_{i}D_{k}u^{j},D_{i}u DP \big\rangle \Phi'(P) \psi\,dz,\\
J_{3} & :=2\int_{Q_{r}(z_{1})}\mathcal{A}_{\varepsilon}(x,t,Du)(D^{2}u,Du\otimes D\eta) \Phi(P) \chi \omega^{2} \eta\,dz,\\
J_{4} & :=\int_{Q_{r}(z_{1})}\sum_{i=1}^{n}f_{\varepsilon}\cdot D_{i}^{2}u  \,\Phi(P) \psi\,dz,\\
J_{5} & :=\int_{Q_{r}(z_{1})}\sum_{i=1}^{n}f_{\varepsilon}\cdot D_{i}u  \,\Phi'(P) D_{i}P  \,\psi\,dz,\\
J_{6} & :=2\int_{Q_{r}(z_{1})}\sum_{i=1}^{n}f_{\varepsilon}\cdot D_{i}u  \,\Phi(P) \chi \omega^{2} \eta D_{i}\eta\,dz,\\
J_{7} & :=\int_{Q_{r}(z_{1})}\sum_{i=1}^{n} \big\langle\big[D_{x_{i}}A_{\varepsilon}\big](x,t,Du),D_{i}Du \big\rangle \Phi(P) \psi\,dz,\\
J_{8} & :=\int_{Q_{r}(z_{1})}\sum_{i=1}^{n}\big\langle\big[D_{x_{i}}A_{\varepsilon}\big](x,t,Du),D_{i}u DP \big\rangle \Phi'(P) \psi\,dz,\\
J_{9} & :=2\int_{Q_{r}(z_{1})}\sum_{i=1}^{n} \big\langle\big[D_{x_{i}}A_{\varepsilon}\big](x,t,Du),D\eta D_{i}u \big\rangle \Phi(P) \chi \omega^{2} \eta\,dz.
\end{align*}
For convenience of notation, we now abbreviate $Q_r=Q_r(z_1)$ and $B_r=B_r(x_1)$. 
In what follows, we will
estimate each of the last nine terms above. We first prove that $J_{2}$
is non-negative, thus we can drop it in the following. We limit ourselves
to dealing with the case $p>2$, since for $1<p\leq 2$ the same result
follows in a similar way. For $p>2$ we have 
\begin{align*}
D_{\xi_{k}^{j}}[(A_{\varepsilon})_{m}^{\ell}(x,t,\xi)] & =D_{\xi_{k}^{j}}\bigg[\frac{\partial_{s}F(x,t,|\xi|)}{|\xi|} \xi_{m}^{\ell}+\varepsilon \big(|\xi|^{2}+1\big)^{\frac{p-2}{2}}\xi_{m}^{\ell}\bigg]\\
 & =\bigg(\frac{\partial_{s}F(x,t,|\xi|)}{|\xi|}+\varepsilon \big(|\xi|^{2}+1\big)^{\frac{p-2}{2}}\bigg) \delta_{mk} \delta_{\ell j}\\
 & \quad\;+\bigg(\frac{|\xi| \partial_{ss}F(x,t,|\xi|)-\partial_{s}F(x,t,|\xi|)}{|\xi|^{3}}+(p-2) \varepsilon \big(|\xi|^{2}+1\big)^{\frac{p-4}{2}}\bigg) \xi_{k}^{j} \xi_{m}^{\ell} ,
\end{align*}
where $\delta_{mk}$ and $\delta_{\ell j}$ denote the Kronecker
delta. Using the above equality and the expression for $D|Du|$, we
then obtain 
\begin{align}\label{eq:pleasedonotremoveme}
\langle &D_{\xi_{k}^{j}}A_{\varepsilon}(x,t,Du)D_{i}D_{k}u^{j},D_{i}u DP\rangle\nonumber\\
& =
\sum_{m=1}^{n}\sum_{\ell=1}^{N}D_{\xi_{k}^{j}}[(A_{\varepsilon})_{m}^{\ell}(x,t,Du)] D_{i}D_{k}u^{j} D_{i}u^{\ell} D_{m}|Du|\nonumber\\
 & =\sum_{m=1}^{n}\sum_{\ell=1}^{N}\bigg(\frac{\partial_{s}F(x,t,|Du|)}{|Du|}+\varepsilon \big(|Du|^{2}+1\big)^{\frac{p-2}{2}}\bigg)D_{i}D_{m}u^{\ell} D_{i}u^{\ell} D_{m}|Du|\nonumber\\
 & \qquad +\sum_{m=1}^{n}\sum_{\ell=1}^{N}\bigg(\frac{|Du|\partial_{ss}F(x,t,|Du|)-\partial_{s}F(x,t,|Du|)}{|Du|^{3}}+(p-2) \varepsilon \big(|Du|^{2}+1\big)^{\frac{p-4}{2}}\bigg)\cdot\nonumber\\
 & \qquad\qquad \cdot D_{k}u^{j} D_{m}u^{\ell} D_{i}D_{k}u^{j} D_{i}u^{\ell} D_{m}|Du|\nonumber\\
 & =\big(\partial_{s}F(x,t,|Du|)+\varepsilon \big(|Du|^{2}+1\big)^{\frac{p-2}{2}}|Du|\big) |D|Du||^{2}\nonumber\\
 & \qquad +\sum_{m=1}^{n}\sum_{\ell=1}^{N}\bigg(\frac{|Du| \partial_{ss}F(x,t,|Du|)-\partial_{s}F(x,t,|Du|)}{|Du|^{2}}+(p-2) \varepsilon \big(|Du|^{2}+1\big)^{\frac{p-4}{2}}|Du|\bigg)\cdot\nonumber\\
 & \qquad\qquad \cdot D_{m}u^{\ell} D_{i}u^{\ell} D_{m}|Du| D_{i}|Du|
\end{align}
almost everywhere in the set $Q_{r}\cap\{|Du|>a\}$. Now, the
Cauchy-Schwarz inequality implies that 
\[
\sum_{m=1}^{n}\sum_{\ell=1}^{N}D_{m}u^{\ell} D_{i}u^{\ell} D_{m}|Du| D_{i}|Du| \leq |Du|^{2} |D|Du||^{2}.
\]
Combining this with the fact that $\partial_{s}F(x,t,|Du|),\partial_{ss}F(x,t,|Du|)\geq0$, from \eqref{eq:pleasedonotremoveme} we infer that 
\begin{align*}
 & \langle D_{\xi_{k}^{j}}A_{\varepsilon}(x,t,Du)D_{i}D_{k}u^{j},D_{i}u DP\rangle\\
 & \qquad \geq\partial_{s}F(x,t,|Du|) |D|Du||^{2}- \frac{\partial_{s}F(x,t,|Du|)}{|Du|^{2}} \sum_{m=1}^{n}\sum_{\ell=1}^{N}D_{m}u^{\ell} D_{i}u^{\ell} D_{m}|Du| D_{i}|Du|\geq0
\end{align*}
almost everywhere in $Q_{r}\cap\{|Du|>a\}$. Furthermore, we
obviously have 
\[
\langle D_{\xi_{k}^{j}}A_{\varepsilon}(x,t,Du)D_{i}D_{k}u^{j},D_{i}u DP\rangle=0 \qquad \mathrm{when} \quad \vert Du\vert\leq a.
\]
Thus, taking into account that $\psi\geq0$ and $\Phi'(w)=(a+w)^{\gamma-3} w(\gamma w+2a)\geq0$
for every $w\geq0$, we arrive at the desired conclusion, that is
$J_{2}\geq0$. 

We now deal with the terms
$J_{1}$ and $J_{3}$. We let $0<\varepsilon<\min \{\frac{1}{2},(\frac{\delta}{4})^{p-1}\}$.
By Lemma \ref{lem:lemma app1} we have 
\begin{equation}
J_{1} \geq \frac{1}{c_{2}}\int_{Q_{r}}\vert Du\vert^{p-2}\vert D^{2}u\vert^{2} \Phi(P) \psi\,dz.\label{eq:est31}
\end{equation}
Using the Cauchy-Schwarz inequality together
with Young's inequality and again Lemma \ref{lem:lemma app1}, we get 
\begin{align}
	\vert J_{3}\vert 
	&\leq
	2\int_{Q_{r}}
	\vert\mathcal{A}_{\varepsilon}(x,t,Du)(D^{2}u,Du\otimes D\eta)\vert 
	\Phi(P) \chi \omega^{2} \eta\,dz\nonumber\\
 	&\leq
	2\int_{Q_{r}}
	\sqrt{\mathcal{A}_{\varepsilon}(x,t,Du)(D^{2}u,D^{2}u)}
	\sqrt{\mathcal{A}_{\varepsilon}(x,t,Du)(Du\otimes D\eta,Du\otimes D\eta)} 
	\Phi(P) \chi \omega^{2} \eta\,dz\nonumber\\
 	& \leq
	\frac{1}{2} J_{1}+
	2\int_{Q_{r}}
	\mathcal{A}_{\varepsilon}(x,t,Du)(Du\otimes D\eta,Du\otimes D\eta) 
	\Phi(P) \chi \omega^{2}\,dz \nonumber\\
 	&\le 
 	\frac{1}{2} J_{1} +
	2 c_{2}\int_{Q_{r}}\vert Du\vert^{p} \Phi(P) \chi \omega^{2} \vert D\eta\vert^{2}\,dz,
\label{eq:est31bis}
\end{align}
where $c_{2}\equiv c_{2}(p,C_{1},\delta)>1$. 

Now we estimate the terms
containing $f_{\varepsilon}$. To this end, we will use that $\vert f_{\varepsilon}\vert\leq\vert f\vert$.
Let us first consider $J_{6}$. Applying Young's inequality with exponents
$(2,2)$ and using the definitions of $P$, $\Phi$ and $H_{\delta}$,
we find that 
\begin{align}
\vert J_{6}\vert & \leq 2 \int_{Q_{r}}|f| |Du| \Phi(P) \chi \omega^{2} \eta |D\eta|\,dz\nonumber\\
 & \leq\int_{Q_{r}}|f|^{2} |Du| \Phi(P) \psi\,dz
 +
 \int_{Q_{r}}|Du| \Phi(P) \chi \omega^{2} |D\eta|^{2}\,dz\nonumber\\
 & \leq
 \int_{Q_{r}}|f|^{2} (H_{\delta}(Du))^{\gamma+1} \psi\,dz
 +
 \int_{Q_{r}}|Du|^{p} \Phi(P) \chi \omega^{2} |D\eta|^{2}\,dz.
\label{eq:est32}
\end{align}
We now deal with the terms
$J_{4}$ and $J_{5}$ at the same time. Using equality (\ref{eq:gradP}),
the fact that $\vert Du\vert\leq2P$ on the set $Q_{r}\cap\{\vert Du\vert\geq2a\}$
and inequality (\ref{PhiDt}) with $w=P$, we get
\begin{align}
\vert J_{4}\vert+\vert J_{5}\vert & \leq n\int_{Q_{r}}|f| |D^{2}u| \Phi(P) \psi\,dz
+
\int_{Q_{r}\cap\{\vert Du\vert \geq 2a\}}|f| |Du| \Phi'(P) \vert D^{2}u\vert \psi\,dz\nonumber\\
 & \qquad +
 \int_{Q_{r}\cap\{a < \vert Du\vert < 2a\}}|f| |Du| \Phi'(P)\,\vert D^{2}u\vert \psi\,dz\nonumber\\
 & \leq 5 n (\gamma+1)\int_{Q_{r}}|f| \Phi(P) \vert D^{2}u\vert \psi\,dz +
 \int_{Q_{r}\cap\{a < \vert Du\vert < 2a\}}|f| |Du| \Phi'(P) \vert D^{2}u\vert \psi\,dz.
\label{eq:est33}
\end{align}
At this point, we apply Young's inequality with $\kappa>0$ and exponents
$(2,2)$ to control the first term on the right-hand side of \eqref{eq:est33}.
In the case $p>2$, we obtain 
\begin{align} 
& 5 n (\gamma+1)\int_{Q_{r}}|f| \Phi(P) \vert D^{2}u\vert \psi\,dz\nonumber\\
 & \qquad \leq \kappa\int_{Q_{r}}\Phi(P) \vert D^{2}u\vert^{2} \psi\,dz+\frac{25 n^{2}(\gamma+1)^{2}}{4\kappa}\int_{Q_{r}}|f|^{2} \Phi(P) \psi\,dz\nonumber\\
 & \qquad \leq \kappa\int_{Q_{r}}\Phi(P) \vert D^{2}u\vert^{2} \psi\,dz+\frac{25 n^{2}(\gamma+1)^{2}}{4\kappa}\int_{Q_{r}}|f|^{2} (H_{\frac{\delta}{2}}(Du))^{\gamma} \psi\,dz\nonumber\\
 & \qquad \leq \kappa\int_{Q_{r}}\Phi(P) \vert D^{2}u\vert^{2} \vert Du\vert^{p-2} \psi\,dz+\frac{25 n^{2}(\gamma+1)^{2}}{4\kappa}\int_{Q_{r}}|f|^{2} (H_{\delta}(Du))^{\gamma+1} \psi\,dz.
\label{eq:est34}
\end{align}
If, on the other hand, $1<p\leq 2$, then we have 
\begin{align} 
& 5n(\gamma+1)\int_{Q_{r}}|f| \Phi(P) \vert D^{2}u\vert \psi\,dz\nonumber\\
 & \qquad \leq
 \kappa\int_{Q_{r}}\Phi(P) \vert D^{2}u\vert^{2} \vert Du\vert^{p-2} \psi\,dz +
 \frac{25 n^{2}(\gamma+1)^{2}}{4\kappa}\int_{Q_{r}}|f|^{2} \Phi(P) \vert Du\vert^{2-p} \psi\,dz\nonumber\\
 & \qquad \leq
 \kappa\int_{Q_{r}}\Phi(P) \vert D^{2}u\vert^{2} \vert Du\vert^{p-2} \psi\,dz +
 \frac{25 n^{2}(\gamma+1)^{2}}{4\kappa}\int_{Q_{r}}|f|^{2} (H_{\delta}(Du))^{\gamma+1} \psi\,dz,
\label{eq:est35}
\end{align}
where, in the last line, we have used that $\Phi(P)=0$ in $Q_{r}\cap\{\vert Du\vert\leq a\}$,
while $\Phi(P) \vert Du\vert^{2-p}\leq\vert Du\vert^{\gamma+1}\leq(H_{\delta}(Du))^{\gamma+1}$
in $Q_{r}\cap\{\vert Du\vert>a\}$. Now we estimate the second
integral on the right-hand side of \eqref{eq:est33}. Using Young's
inequality with exponents $(2,2)$ as well as inequalities (\ref{eq:ineqPhi})
and (\ref{PhiDt}) with $w=P$, for every $\zeta>0$ we obtain 
\begin{align*} 
& \int_{Q_{r}\cap\{a<\vert Du\vert<2a\}}|f| |Du| \Phi'(P) \vert D^{2}u\vert \psi\,dz\nonumber\\
 & \qquad =
 \int_{Q_{r}\cap\{a<\vert Du\vert<2a\}}\left[|f|^{2} |Du|^{4-p}  \frac{\gamma+1}{P+\zeta}\right]^{\frac{1}{2}}\left[|Du|^{p-2} \vert D^{2}u\vert^{2}  \frac{P+\zeta}{\gamma+1}\right]^{\frac{1}{2}}\Phi'(P) \psi\,dz\nonumber\\
 & \qquad \leq
 \frac{\kappa}{2 (\gamma+1)}\int_{Q_{r}\cap\{a<\vert Du\vert<2a\}}|Du|^{p-2} |D^{2}u|^{2} (P+\zeta) \Phi'(P) \psi\,dz\nonumber\\
 & \qquad\qquad +
 \frac{\gamma+1}{2 \kappa}\int_{Q_{r}\cap\{a<\vert Du\vert<2a\}}|f|^{2} |Du|^{4-p}  \frac{\Phi'(P)}{P+\zeta} \psi\,dz\nonumber\\
 & \qquad \leq
 \kappa\int_{Q_{r}\cap\{a<\vert Du\vert<2a\}}|Du|^{p-2} |D^{2}u|^{2} \Phi(P) \psi\,dz\nonumber\\
 & \qquad\qquad +
 \kappa\zeta\int_{Q_{r}\cap\{a<\vert Du\vert<2a\}}|Du|^{p-2} |D^{2}u|^{2} \vert Du\vert^{\gamma} \psi\,dz\nonumber\\
 & \qquad\qquad +
 \frac{(\gamma+1)^{2}}{\kappa}\int_{Q_{r}\cap\{a<\vert Du\vert<2a\}}|f|^{2} |Du|^{4-p} \vert Du\vert^{\gamma-2} \psi\,dz\nonumber\\
 & \qquad\leq \kappa\int_{Q_{r}}|Du|^{p-2} |D^{2}u|^{2} \Phi(P) \psi\,dz+(2a)^{\gamma} \kappa \zeta\int_{Q_{r}\cap\{a<\vert Du\vert<2a\}}|Du|^{p-2} |D^{2}u|^{2}\,dz\nonumber\\
 & \qquad\qquad +
 \frac{(\gamma+1)^{2}}{\kappa}\int_{Q_{r}\cap\{a<\vert Du\vert<2a\}}|f|^{2} (H_{\delta}(Du))^{\gamma+1} \psi\,dz,
\end{align*}
where the last inequality is due to the fact that $\vert Du\vert^{4-p} \vert Du\vert^{\gamma-2}=\vert Du\vert^{\gamma+1} \vert Du\vert^{1-p}\leq\vert Du\vert^{\gamma+1}\le(H_{\delta}(Du))^{\gamma+1}$
for every $p>1$ on the set $Q_{r}\cap\{a<\vert Du\vert<2a\}$.
We now recall that, by virtue of Theorem \ref{thm:regul}, we have
\[
|D^{2}u|^{2} |Du|^{p-2} \in L_{loc}^{1}(Q_{R}(z_{0})).
\]
Therefore, letting $\zeta\searrow0$ in the preceding inequality, we get 
\begin{align} 
& \int_{Q_{r}\cap\{a<\vert Du\vert<2a\}}|f| |Du| \Phi'(P) \vert D^{2}u\vert \psi\,dz\nonumber\\
 & \qquad \leq
 \kappa\int_{Q_{r}}|Du|^{p-2} |D^{2}u|^{2} \Phi(P) \psi\,dz+\frac{(\gamma+1)^{2}}{\kappa}\int_{Q_{r}\cap\{a<\vert Du\vert<2a\}}|f|^{2} (H_{\delta}(Du))^{\gamma+1} \psi\,dz.
\label{eq:est37}
\end{align}
At this point, joining estimates \eqref{eq:est33}$-$\eqref{eq:est37}, we obtain 
\begin{align}\vert J_{4}\vert+\vert J_{5}\vert & \leq \kappa (n+1)\int_{Q_{r}}\Phi(P) \vert D^{2}u\vert^{2} \vert Du\vert^{p-2} \psi\,dz\nonumber\\
 & \qquad +
 \frac{8 n^{2} (\gamma+1)^{2}}{\kappa}\int_{Q_{r}}|f|^{2} (H_{\delta}(Du))^{\gamma+1} \psi\,dz.
\label{eq:est38}
\end{align}
$\hspace*{1em}$ We now turn our attention
to $J_{7}$ and $J_{9}$. Using (\ref{eq:spgraAeps}) and
Young's inequality, we find 
\begin{align}
|J_{7}| & \leq 2^{2p-2} K n\int_{Q_{r}}|Du|^{p-1} |D^{2}u| \Phi(P) \psi\,dz\nonumber\\
 & \leq \kappa\int_{Q_{r}}|Du|^{p-2} |D^{2}u|^{2} \Phi(P) \psi\,dz+\frac{16^{p} K^{2} n^{2}}{2 \kappa}\int_{Q_{r}}|Du|^{p} \Phi(P) \psi\,dz.
\label{eq:est39}
\end{align}
Similarly, we have 
\begin{align}
|J_{9}| & \leq 2^{2p-1} K n\int_{Q_{r}}|Du|^{p} |D\eta| \Phi(P) \chi \omega^{2} \eta\,dz\nonumber\\
 & \leq 2^{2p-2} K^{2} n\int_{Q_{r}}|Du|^{p} \Phi(P) \psi \,dz+2^{2p-2} n\int_{Q_{r}}|Du|^{p} \vert D\eta\vert^{2} \Phi(P) \chi \omega^{2}\,dz.
\label{eq:est40}
\end{align}
$\hspace*{1em}$ Finally we estimate $J_{8}$.
Using estimate (\ref{eq:spgraAeps}) and equality (\ref{eq:gradP}),
we obtain 
\begin{align}
	|J_{8}| 
	& \leq
	2^{2p-2}K n\int_{Q_{r}}|Du|^{p} |DP| \Phi'(P) \psi\,dz 	
	\leq
	2^{2p-2}Kn\int_{Q_{r}}|Du|^{p} |D^{2}u| \Phi'(P) \psi\,dz\nonumber\\
 & =
	2^{2p-2}Kn\int_{Q_{r}\cap\{a<\vert Du\vert<2a\}}|Du|^{p} |D^{2}u| \Phi'(P) \psi\,dz\nonumber\\
 & \qquad +
 2^{2p-2}Kn\int_{Q_{r}\cap\{\vert Du\vert \geq 2a\}}|Du|^{p} |D^{2}u| \Phi'(P) \psi\,dz.
\label{eq:est41}
\end{align}
Now we deal with the second integral on the right-hand side of \eqref{eq:est41}.
Using the fact that $\vert Du\vert\leq2P$ in $Q_{r}\cap\{\vert Du\vert\geq2a\}$,
the inequality (\ref{PhiDt}) with $w=P$ as well as Young's inequality,
we get 
\begin{align} 
	& 2^{2p-2}Kn\int_{Q_{r}\cap\{\vert Du\vert \geq 2a\}}|Du|^{p} |D^{2}u| \Phi'(P) \psi\,dz\nonumber\\
 	& \qquad \leq
	2^{2p-1}Kn\int_{Q_{r}\cap\{\vert Du\vert \geq 2a\}}|Du|^{p-1} |D^{2}u| P \Phi'(P) \psi\,dz\nonumber\\
 & \qquad \leq
 2^{2p}(\gamma+1)Kn\int_{Q_{r}\cap\{\vert Du\vert \geq 2a\}}|Du|^{p-1} |D^{2}u| \Phi(P) \psi\,dz\nonumber\\
 & \qquad \leq
 \kappa\int_{Q_{r}}|Du|^{p-2} |D^{2}u|^{2} \Phi(P) \psi\,dz+\frac{16^{p}(\gamma+1)^{2} K^{2} n^{2}}{4 \kappa} \int_{Q_{r}}|Du|^{p} \Phi(P) \psi\,dz\nonumber\\
 & \qquad \leq
 \kappa\int_{Q_{r}}|Du|^{p-2} |D^{2}u|^{2} \Phi(P) \psi\,dz+\frac{16^{p}(\gamma+1)^{2} K^{2} n^{2}}{4 \kappa} \int_{Q_{r}}|Du|^{\gamma+p} \psi\,dz.
\label{eq:est42}
\end{align}
We now estimate the first term on the right-hand side of \eqref{eq:est41}
by resorting to the same procedure leading to \eqref{eq:est37}.
Thus, for every $\sigma>0$ we have 
\begin{align*} 
& 2^{2p-2}Kn\int_{Q_{r}\cap\{a<\vert Du\vert<2a\}}|Du|^{p} |D^{2}u| \Phi'(P) \psi\,dz\nonumber\\
 & \qquad =
 \int_{Q_{r}\cap\{a<\vert Du\vert<2a\}}\left[2^{4p-4} K^{2} n^{2} \vert Du\vert^{p+2} \frac{\gamma+1}{P+\sigma}\right]^{\frac{1}{2}}\left[\vert Du\vert^{p-2} \vert D^{2}u\vert^{2} \frac{P+\sigma}{\gamma+1}\right]^{\frac{1}{2}}\Phi'(P) \psi\,dz\nonumber\\
 & \qquad \leq
 \frac{\kappa}{2 (\gamma+1)} \int_{Q_{r}\cap\{a<\vert Du\vert<2a\}}\vert Du\vert^{p-2} \vert D^{2}u\vert^{2} (P+\sigma) \Phi'(P) \psi\,dz\nonumber\\
 & \qquad\qquad +
 \frac{16^{p}(\gamma+1) K^{2} n^{2}}{32 \kappa} \int_{Q_{r}\cap\{a<\vert Du\vert<2a\}}\vert Du\vert^{p+2}  \frac{\Phi'(P)}{P+\sigma}  \psi\,dz\nonumber\\
 & \qquad \leq
 \kappa\int_{Q_{r}}\vert Du\vert^{p-2} \vert D^{2}u\vert^{2} \Phi(P) \psi\,dz +
 (2a)^{\gamma} \kappa \sigma\int_{Q_{r}}\vert Du\vert^{p-2} \vert D^{2}u\vert^{2} dz\nonumber\\
 & \qquad\qquad +
 \frac{16^{p}(\gamma+1)^{2} K^{2} n^{2}}{4 \kappa} \int_{Q_{r}}\vert Du\vert^{\gamma+p} \psi\,dz.
\end{align*}
Letting $\sigma\searrow0$ in the last inequality, we then obtain 
\begin{align}\label{eq:est44}
& 2^{2p-2}Kn\int_{Q_{r}\cap\{a<\vert Du\vert<2a\}}|Du|^{p} |D^{2}u| \Phi'(P) \psi\,dz\nonumber\\
 & \qquad \leq
 \kappa\int_{Q_{r}}\vert Du\vert^{p-2} \vert D^{2}u\vert^{2} \Phi(P) \psi\,dz +
 \frac{16^{p}(\gamma+1)^{2} K^{2} n^{2}}{4 \kappa} \int_{Q_{r}}\vert Du\vert^{\gamma+p} \psi\,dz.
\end{align}
At this point, combining estimates \eqref{eq:est41}$-$\eqref{eq:est44}, we find 
\begin{equation}
\vert J_{8}\vert
\leq
2\kappa\int_{Q_{r}}\vert Du\vert^{p-2} \vert D^{2}u\vert^{2} \Phi(P) \psi\,dz +
\frac{16^{p}(\gamma+1)^{2} K^{2} n^{2}}{2 \kappa} \int_{Q_{r}}\vert Du\vert^{\gamma+p} \psi\,dz.\label{eq:est45}
\end{equation}
$\hspace*{1em}$ Now, observe that from
inequality (\ref{eq:Jinequality}) it follows that 
\begin{equation}
J_{0}+J_{1} \leq \vert J_{3}\vert+\vert J_{4}\vert+\vert J_{5}\vert+\vert J_{6}\vert+\vert J_{7}\vert+\vert J_{8}\vert+\vert J_{9}\vert,\label{eq:Jineq2}
\end{equation}
since $J_{2}\geq0$. Plugging estimates (\ref{eq:est31})$-$\eqref{eq:est32},
\eqref{eq:est38}$-$\eqref{eq:est40} and (\ref{eq:est45}) into
(\ref{eq:Jineq2}) and choosing $\kappa=\frac{1}{4 c_{2} (n+4)}$,
we arrive at 
\begin{align}\label{eq:est46}
 	- & \frac{\delta^{2}}{4 b^{2}(\gamma+2)}\int_{Q_{r}}(|Du|^{\gamma+2}-b^{\gamma+2})_{+}\ \partial_{t}\chi \,\omega^{2}\eta^{2}\,dz +
	\frac{1}{4 c_{2}}\int_{Q_{r}}|Du|^{p-2}\vert D^{2}u\vert^{2} \Phi(P) \psi\,dz\nonumber \\
 	& \leq
	c(\gamma+1)^{2} \bigg[
	\int_{Q_{r}} \big[|Du|^{p}\Phi(P)\chi\omega^{2}\vert D\eta\vert^{2} +
	|Du|^{\gamma+p}\psi \big]\,dz +
	\int_{Q_{r}}|f|^{2} (H_{\delta}(Du))^{\gamma+1} \psi\,dz\bigg] \nonumber \\
 	& \quad +
	\frac{2}{\gamma+2}\int_{Q_{r}}(H_{\frac{\delta}{2}}(Du))^{\gamma+2} \chi \omega \,\partial_{t}\omega \,\eta^{2}\,dz,
\end{align}
where $c\equiv c(n,p,\delta,C_{1},K)>1$.

At this stage, we perform
a particular choice of the function $\chi$. For a fixed time $\tau\in(t_{1}-r^{2},t_{1})$
and $\vartheta\in(0,t_{1}-\tau)$, we define the Lipschitz continuous
function $\chi$ by 
\[
\chi(t):=\begin{cases}
1 & \text{if }  t\leq\tau, \\
1-\frac{t-\tau}{\vartheta} & \text{if }  \tau<t\leq\tau+\vartheta,\\
0 & \text{if }  t>\tau+\vartheta.
\end{cases}
\]
Therefore, we have 
\[
-\int_{Q_{r}(z_{1})}(|Du|^{\gamma+2}-b^{\gamma+2})_{+}\ \partial_{t}\chi \,\omega^{2}\eta^{2}\,dz = \frac{1}{\vartheta}\int_{\tau}^{\tau+\vartheta}\int_{B_{r}(x_{1})}(|Du|^{\gamma+2}-b^{\gamma+2})_{+}\ \omega^{2}\eta^{2}\,dx\,dt,
\]
which converges for almost every $\tau\in(t_{1}-r^{2},t_{1})$ to
\[
\int_{B_{r}(x_{1})\times\{\tau\}}(|Du|^{\gamma+2}-b^{\gamma+2})_+  \omega^{2}\eta^{2}\,dx \qquad \mathrm{as} \qquad \vartheta\searrow0.
\]
Hence, letting $\vartheta\searrow0$ and taking the supremum over
$\tau\in(t_{1}-r^{2},t_{1})$, estimate \eqref{eq:est46} turns into
\begin{align*}
	& \frac{\delta^{2}}{4b^{2}(\gamma+2)}
	\sup\limits _{\tau\in(t_{1}-r^{2},t_{1})}
	\int_{B_{r}\times\{\tau\}}(|Du|^{\gamma+2}-b^{\gamma+2})_{+}  	
	\omega^{2}\eta^{2}\,dx +
 	\frac{1}{4c_{2}}\int_{Q_{r}}|Du|^{p-2}\vert D^{2}u\vert^{2}\Phi(P)
	\omega^{2}\eta^{2}\,dz\\
 	& \qquad \leq
 	c(\gamma+1)^{2}\int_{Q_{r}} 
	\big[|Du|^{p} \Phi(P) \omega^{2} |D\eta|^{2} +
	|Du|^{\gamma+p}\omega^{2}\eta^{2} \big]\,dz\\
 	& \qquad\qquad +
 	\int_{Q_{r}}(H_{\frac{\delta}{2}}(Du))^{\gamma+2}
	\omega\,\partial_{t}\omega\,\eta^{2}\,dz +
	c(\gamma+1)^{2}\int_{Q_{r}}|f|^{2} (H_{\delta}(Du))^{\gamma+1}
	\omega^{2}\eta^{2}\,dz \\
 	& \qquad \leq
	c(\gamma+1)^{2}\int_{Q_{r}}(H_{\delta}(Du))^{\gamma+p}(\omega^{2}
	\vert D\eta\vert^{2}+\omega^{2}\eta^{2})\,dz\\
 	& \qquad\qquad +
	\int_{Q_{r}}(H_{\frac{\delta}{2}}(Du))^{\gamma+2}\omega\,\partial_{t}\omega
	\,\eta^{2}\,dz +
	c(\gamma+1)^{2}\int_{Q_{r}}|f|^{2}(H_{\delta}(Du))^{\gamma+1}
	\omega^{2}\eta^{2}\,dz.
\end{align*}
In summary, we have so far obtained the following result.

\begin{prop}
Let $p>1$, $\gamma>0$, $\delta>0$,
$b=1+\delta$ and $0<\varepsilon<\min \{\frac{1}{2},(\frac{\delta}{4})^{p-1}\}$.
Moreover, assume that 
\[
u_{\varepsilon}\in C^{0}\left([t_{0}-R^{2},t_{0}];L^{2}\left(B_{R}(x_{0}),\mathbb{R}^{N}\right)\right)\cap L^{p}\left(t_{0}-R^{2},t_{0};W^{1,p}\left(B_{R}(x_{0}),\mathbb{R}^{N}\right)\right)
\]
is the unique energy solution of problem
$(\ref{eq:CAUCHYDIR})$ with $Q'=Q_{R}(z_{0})\Subset\Omega_{T}$ and
$u$ a weak solution of $\mathrm{(\ref{eq:syst})}$, satisfying the
additional assumption of Remark \ref{rmk2}
if $1<p\le\frac{2n}{n + 2}$. Then, for any parabolic
cylinder $Q_{r}(z_{1})\Subset Q_{R}(z_{0})$
with $r\in(0,1)$ and any cut-off functions $\eta\in C_{0}^{\infty}(B_{r}(x_{1}),[0,1])$
and $\omega\in W^{1,\infty}(\mathbb{R},[0,1])$ satisfying $\omega(t_{1}-r^{2})=0$
and $\partial_{t}\omega\geq0$, the estimate
\begin{align}\label{eq:est47}
	&\frac{\delta^{2}}{4b^{2}(\gamma+2)}
	\sup\limits_{\tau\in(t_{1}-r^{2},t_{1})}
	\int_{B_{r}(x_{1})\times\{\tau\}}(|Du_{\varepsilon}|^{\gamma+2}-b^{\gamma+2})_{+} 
	\omega^{2}\eta^{2}\,dx\nonumber\\
	&\qquad\qquad+
	\frac{1}{4c_{2}}\int_{Q_{r}(z_{1})}
	|Du_{\varepsilon}|^{p-2} \vert D^{2}u_{\varepsilon}\vert^{2} \Phi(P) 
	\omega^{2}\eta^{2}\,dz\nonumber\\
	&\qquad \leq
	c(\gamma+1)^{2}\int_{Q_{r}(z_{1})}(H_{\delta}(Du_{\varepsilon}))^{\gamma+p} (\omega^{2} \vert D\eta\vert^{2}+\omega^{2}\eta^{2})\,dz \nonumber\\
	&\qquad\qquad +
	\int_{Q_{r}(z_{1})}(H_{\frac{\delta}{2}}(Du_{\varepsilon}))^{\gamma+2} 
	\omega \partial_{t}\omega \eta^{2}\,dz +
	c (\gamma+1)^{2}\int_{Q_{r}(z_{1})}|f|^{2} (H_{\delta}(Du_{\varepsilon}))^{\gamma+1} \omega^{2}\eta^{2}\,dz
\end{align}
holds true for some constants $c_{2}\equiv c_{2}(p,C_{1},\delta)>1$
and $c\equiv c(n,p,\delta,C_{1},K)>1$.\smallskip{}
\end{prop}

In what follows, we will write again
$u$ and $\Phi$ in place of $u_{\varepsilon}$ and $\Phi_{\gamma,a}$
respectively, unless otherwise specified. First, we estimate the second integral on the left-hand side of \eqref{eq:est47} from below.
To this end, we note that 
\[
D\left[(H_{\delta}(Du))^{\frac{\gamma+p}{2}}\right] = \frac{\gamma+p}{2} (H_{\delta}(Du))^{\frac{\gamma+p-2}{2}} \sum_{i=1}^{n}\sum_{j=1}^{N} \frac{D_{i}u^{j} D(D_{i}u^{j})}{\vert Du\vert}  \mathds{1}_{\{\vert Du\vert > 1+\delta\}} .
\]
From this identity, we infer 
\begin{align}\left|D\left[(H_{\delta}(Du))^{\frac{\gamma+p}{2}}\right]\right|^{2}  & \leq n \frac{(\gamma+p)^{2}}{4} (H_{\delta}(Du))^{\gamma+p-2} \vert D^{2}u\vert^{2} \mathds{1}_{\{\vert Du\vert > 1+\delta\}}\nonumber\\
 & \leq n \frac{(\gamma+p)^{2}}{4} \vert Du\vert^{\gamma+p-4} \vert D^{2}u\vert^{2} \vert Du\vert^{2} \mathds{1}_{\{\vert Du\vert > 1+\delta\}}\nonumber\\
 & \leq n \frac{(\gamma+p)^{2} (1+\delta)^{2}}{\delta^{2}} \vert Du\vert^{\gamma+p-4} P^{2} \vert D^{2}u\vert^{2}\nonumber\\
 & \leq n \frac{(\gamma+p)^{2} (1+\delta)^{2}}{\delta^{2}} \vert Du\vert^{p-2} \Phi(P) \vert D^{2}u\vert^{2},
\label{eq:est52}
\end{align}
where, in the second to last line, we have used the fact that 
\[
\vert Du\vert \mathds{1}_{\{\vert Du\vert > 1+\delta\}}=\left(P+1+\tfrac{\delta}{2}\right) \mathds{1}_{\{\vert Du\vert > 1+\delta\}}\leq
\frac{2(1+\delta)}{\delta} P \mathds{1}_{\{\vert Du\vert > 1+\delta\}}\leq
\frac{2(1+\delta)}{\delta}P.
\]
Using estimate \eqref{eq:est52} in combination with Young's inequality,
we then obtain 
\begin{align*}
\left|D\left[\eta (H_{\delta}(Du))^{\frac{\gamma+p}{2}}\right]\right|^{2}  & \leq 2 (H_{\delta}(Du))^{\gamma+p} \vert D\eta\vert^{2}+2 \left|D\left[(H_{\delta}(Du))^{\frac{\gamma+p}{2}}\right]\right|^{2}\eta^{2}\\
 & \leq 2 (H_{\delta}(Du))^{\gamma+p} \vert D\eta\vert^{2}+2 n \frac{(\gamma+p)^{2} b^{2}}{\delta^{2}} \vert Du\vert^{p-2} \Phi(P) \vert D^{2}u\vert^{2}\eta^{2},
\end{align*}
from which we deduce 
\begin{align} & \frac{\delta^{2}}{16 n (\gamma+p)^{2} \,b^{2} \,c_{2}} \left|D\left[\eta (H_{\delta}(Du))^{\frac{\gamma+p}{2}}\right]\right|^{2}\omega^{2}\nonumber\\
 & \qquad \leq \frac{\delta^{2}}{8 n (\gamma+p)^{2} \,b^{2} \,c_{2}} (H_{\delta}(Du))^{\gamma+p} \vert D\eta\vert^{2} \omega^{2}+ \frac{1}{8 c_{2}} \vert Du\vert^{p-2} \Phi(P) \vert D^{2}u\vert^{2} \omega^{2} \eta^{2}.
\label{eq:est53}
\end{align}
Now we consider the case $1<p\leq \frac{2n}{n + 2}$. This
implies 
\begin{equation}
\gamma+ \frac{4}{n+2} \leq \gamma+2-p < \gamma+1 < \gamma+p.\label{eq:gammaineq}
\end{equation}
Our goal now is to reduce the exponent of $H_{\frac{\delta}{2}}(Du)$ in
the second integral on the right-hand side of \eqref{eq:est47}. To
this end, we observe that 
\begin{align} 
& \int_{Q_{r}(z_{1})}(H_{\frac{\delta}{2}}(Du))^{\gamma+2} \omega \partial_{t}\omega \eta^{2}\,dz\nonumber\\
 & \qquad \leq 2 a^{2}\int_{Q_{r}(z_{1})}(H_{\frac{\delta}{2}}(Du))^{\gamma} \omega \partial_{t}\omega \eta^{2}\,dz+2\int_{Q_{r}(z_{1})}(H_{\frac{\delta}{2}}(Du))^{\gamma} P^{2} \omega \partial_{t}\omega \eta^{2}\,dz\nonumber\\
 & \qquad \leq 2 a^{2}\int_{Q_{r}(z_{1})}(H_{\delta}(Du))^{\gamma} \omega \partial_{t}\omega \eta^{2}\,dz+2\int_{Q_{r}(z_{1})}\vert Du\vert^{\gamma} P^{2} \omega \partial_{t}\omega \eta^{2}\,dz.
\label{eq:est48}
\end{align}
In addition, for $\vert Du\vert>a$ we have 
\begin{align*}
 \sum_{i=1}^{n}\sum_{j=1}^{N}D_{i} [u^{j} \vert Du\vert^{\gamma-2} D_{i}u^{j} P^{2} \eta^{2}]
 &=
 \vert Du\vert^{\gamma} P^{2} \eta^{2}+\vert Du\vert^{\gamma-2} P^{2} \eta^{2} \sum_{i=1}^{n}\sum_{j=1}^{N}u^{j} D_{i}^{2}u^{j}\\
 & \quad+
 (\gamma-2) \vert Du\vert^{\gamma-3} P^{2} \eta^{2} \sum_{i,k = 1}^{n}\sum_{j,\ell = 1}^{N}u^{j} D_{i}u^{j} \frac{D_{k}u^{\ell} D_{i}D_{k}u^{\ell}}{\vert Du\vert}\\
 & \quad+
 2\vert Du\vert^{\gamma-2}P \eta^{2} \sum_{i,k = 1}^{n}\sum_{j,\ell = 1}^{N}u^{j} D_{i}u^{j} \frac{D_{k}u^{\ell} D_{i}D_{k}u^{\ell}}{\vert Du\vert}\\
 & \quad+
 2\vert Du\vert^{\gamma-2}P^{2} \eta \sum_{i=1}^{n}\sum_{j=1}^{N}u^{j} D_{i}u^{j} D_{i}\eta .
\end{align*}
Using the above identity to estimate the last integral of \eqref{eq:est48}
and taking into account that $\eta\in C_{0}^{\infty}(B_{r}(x_{1}))$ and $\omega$ is independent of the $x$-variable,
we obtain 
\begin{align} 
 \int_{Q_{r}(z_{1})}\vert Du\vert^{\gamma} P^{2} \omega \partial_{t}\omega \eta^{2} \, dz
&\leq \Vert u\Vert_{L^{\infty}(Q_{r}(z_{1}))} \int_{Q_{r}(z_{1})}\vert Du\vert^{\gamma-2} \vert D^{2}u\vert P^{2} \omega \partial_{t}\omega \eta^{2}\,dz\nonumber\\
 & \quad + 
 \vert\gamma-2\vert \,\Vert u\Vert_{L^{\infty}(Q_{r}(z_{1}))} \int_{Q_{r}(z_{1})}\vert Du\vert^{\gamma-2} \vert D^{2}u\vert P^{2} \omega \partial_{t}\omega \eta^{2}\,dz\nonumber\\
 &  \quad + 
 2\,\Vert u\Vert_{L^{\infty}(Q_{r}(z_{1}))} \int_{Q_{r}(z_{1})}\vert Du\vert^{\gamma-1} \vert D^{2}u\vert P \omega \partial_{t}\omega \eta^{2}\,dz\nonumber\\
  &\quad + 
  2\,\Vert u\Vert_{L^{\infty}(Q_{r}(z_{1}))} \int_{Q_{r}(z_{1})}\vert Du\vert^{\gamma-1} P^{2} \omega \partial_{t}\omega \eta \vert D\eta\vert\,dz\nonumber\\
 & \leq
 (\gamma+5)\Vert u\Vert_{L^{\infty}(Q_{r}(z_{1}))} \int_{Q_{r}(z_{1})}\vert Du\vert^{\gamma-1} \vert D^{2}u\vert P \omega \partial_{t}\omega \eta^{2}\,dz\nonumber\\
 &  \quad + 
 2\,\Vert u\Vert_{L^{\infty}(Q_{r}(z_{1}))} \int_{Q_{r}(z_{1})}\vert Du\vert^{\gamma+1} \omega \partial_{t}\omega \eta \vert D\eta\vert\,dz.
\label{eq:est49}
\end{align}
Joining estimates \eqref{eq:est48} and \eqref{eq:est49} and applying
Young's inequality, we then have 
\begin{align*}
 & \int_{Q_{r}(z_{1})}(H_{\frac{\delta}{2}}(Du))^{\gamma+2} \omega \partial_{t}\omega \eta^{2}\,dz\\
 & \qquad \leq
 2a^{2}\int_{Q_{r}(z_{1})}(H_{\delta}(Du))^{\gamma} \omega \partial_{t}\omega \eta^{2}\,dz+4\, \Vert u\Vert_{L^{\infty}(Q_{r}(z_{1}))} \int_{Q_{r}(z_{1})}\vert Du\vert^{\gamma+1} \omega \partial_{t}\omega \eta \vert D\eta\vert\,dz\\
 & \qquad\qquad +
 10(\gamma+1)\Vert u\Vert_{L^{\infty}(Q_{r}(z_{1}))} \int_{Q_{r}(z_{1})}\vert Du\vert^{\gamma-1} \vert D^{2}u\vert P \omega \partial_{t}\omega \eta^{2}\,dz\\
 & \qquad =
 2a^{2}\int_{Q_{r}(z_{1})}(H_{\delta}(Du))^{\gamma} \omega \partial_{t}\omega \eta^{2}\,dz +
 4\,\Vert u\Vert_{L^{\infty}(Q_{r}(z_{1}))} \int_{Q_{r}(z_{1})}\vert Du\vert^{\gamma+1} \omega \partial_{t}\omega \eta \vert D\eta\vert\,dz\\
 & \qquad\qquad +
 10(\gamma+1)\Vert u\Vert_{L^{\infty}(Q_{r}(z_{1}))} \int_{Q_{r}(z_{1})}(\vert Du\vert^{\frac{\gamma+2-p}{2}} \partial_{t}\omega \eta) (\vert Du\vert^{\frac{p-2}{2}} \vert D^{2}u\vert \vert Du\vert^{\frac{\gamma-2}{2}}P \omega \eta)\,dz\\
 & \qquad \leq
 2a^{2}\int_{Q_{r}(z_{1})}(H_{\delta}(Du))^{\gamma} \omega \partial_{t}\omega \eta^{2}\,dz +
 4\,\Vert u\Vert_{L^{\infty}(Q_{r}(z_{1}))} \int_{Q_{r}(z_{1})}\vert Du\vert^{\gamma+1} \omega \partial_{t}\omega \eta \vert D\eta\vert\,dz\\
 & \qquad\qquad +
 c(\gamma+1)^{2}\Vert u\Vert_{L^{\infty}(Q_{r}(z_{1}))}^{2} \int_{Q_{r}(z_{1})}\vert Du\vert^{\gamma+2-p} (\partial_{t}\omega)^{2} \eta^{2}\,dz\\
 & \qquad\qquad +
 \frac{1}{8c_{2}}\int_{Q_{r}(z_{1})}\vert Du\vert^{p-2} \vert D^{2}u\vert^{2} \Phi(P) \omega^{2} \eta^{2}\,dz.
\end{align*}
Now we notice that $a^{2}\leq b (H_{\delta}(Du))^{p}$, $\vert Du\vert^{\gamma+1}\leq(H_{\delta}(Du))^{\gamma+p}$ and $\vert Du\vert^{\gamma+2-p}\le(H_{\delta}(Du))^{\gamma+p}$
(by virtue of (\ref{eq:gammaineq})).
Using this information and Young's inequality, from the previous estimate we deduce
\begin{align}\label{eq:est50}
\int_{Q_{r}(z_{1})} & (H_{\frac{\delta}{2}}(Du))^{\gamma+2}\omega\partial_{t}\omega\eta^{2}\,dz\nonumber\\
 &\leq
 2b\int_{Q_{r}(z_{1})}(H_{\delta}(Du))^{\gamma+p}\omega\partial_{t}\omega\eta^{2}\,dz +
 4\,\Vert u\Vert_{L^{\infty}(Q_{r}(z_{1}))}\int_{Q_{r}(z_{1})}(H_{\delta}(Du))^{\gamma+p}\omega\partial_{t}\omega\eta\vert D\eta\vert\,dz
  \nonumber\\
 & \quad+
 c(\gamma+1)^{2} \Vert u\Vert_{L^{\infty}(Q_{r}(z_{1}))}^{2} \int_{Q_{r}(z_{1})}(H_{\delta}(Du))^{\gamma+p} (\partial_{t}\omega)^{2} \eta^{2}\,dz\nonumber\\
 & \quad+
 \frac{1}{8c_{2}}\int_{Q_{r}(z_{1})}\vert Du\vert^{p-2} \vert D^{2}u\vert^{2} \Phi(P) \omega^{2} \eta^{2}\,dz
 \nonumber\\
 & \leq
 2b\int_{Q_{r}(z_{1})}(H_{\delta}(Du))^{\gamma+p}[\omega \partial_{t}\omega \eta^{2} +
 \omega^{2}\vert D\eta\vert^{2}]\,dz 
 \nonumber\\
 & \quad+
 c(\gamma+1)^{2}\Vert u\Vert_{L^{\infty}(Q_{r}(z_{1}))}^{2} \int_{Q_{r}(z_{1})}(H_{\delta}(Du))^{\gamma+p} (\partial_{t}\omega)^{2} \eta^{2}\,dz\nonumber\\
 & \quad+
 \frac{1}{8c_{2}}\int_{Q_{r}(z_{1})}\vert Du\vert^{p-2} \vert D^{2}u\vert^{2} \Phi(P) \omega^{2} \eta^{2}\,dz.
\end{align}

At this point, in order
to deal jointly with the cases where $p$ is greater or less than
the critical exponent $\frac{2n}{n + 2}$, we introduce
a new exponent $\mathfrak{p}$, defined as follows: 
\begin{equation}
\mathfrak{p}:=\begin{cases}
p & \text{if }  1<p\leq \frac{2n}{n + 2},\\
2 & \text{if }  \frac{2n}{n + 2}<p<2,\\
p & \text{if }  p\geq2.
\end{cases}\label{eq:newexponent}
\end{equation}
Combining \eqref{eq:est47} and \eqref{eq:est50}, we obtain in the
case $1<p\leq \frac{2n}{n + 2}$ the following
inequality 
\begin{align} 
	& \frac{\delta^{2}}{4b^{2}(\gamma+2)}
	\sup\limits _{\tau\in(t_{1}-r^{2},t_{1})}
	\int_{B_{r}(x_{1})\times\{\tau\}}(|Du|^{\gamma+2}-b^{\gamma+2})_{+}  
	\omega^{2}\eta^{2}\,dx\nonumber\\
 	& +
	\frac{1}{8c_{2}}\int_{Q_{r}(z_{1})}|Du|^{p-2} \vert D^{2}u\vert^{2} \Phi(P) \omega^{2}\eta^{2}\,dz\nonumber\\
 	& \qquad \leq
	c(\gamma+1)^{2}\int_{Q_{r}(z_{1})}(H_{\delta}(Du))^{\gamma+\mathfrak{p}} [\omega^{2}\vert D\eta\vert^{2} +\omega \partial_{t}\omega \eta^{2} +(\Vert u\Vert_{L^{\infty}(Q_{r}(z_{1}))}^{2} (\partial_{t}\omega)^{2}+\omega^{2}) \eta^{2}]\,dz\nonumber\\
 	& \qquad\qquad +
	c(\gamma+1)^{2}\int_{Q_{r}(z_{1})}|f|^{2} (H_{\delta}(Du))^{\gamma+1} \omega^{2}\eta^{2}\,dz.
\label{eq:est51}
\end{align}
For convenience of notation, we now use the indicator function
\[
\mathds{1}_{(1,\frac{2n}{n+2}]}(p):=\begin{cases}
\begin{array}{cc}
1 & \mathrm{if}\,\,1<p\leq\frac{2n}{n+2},\\
0 & \mathrm{otherwise},\,\,\,\,\,\,\,\,\,\,\,\,
\end{array}\end{cases}
\]
to indicate that the terms multiplied by it need to be taken into
account only in the subcritical case $1<p\leq\frac{2n}{n+2}$. Integrating \eqref{eq:est53} over $Q_{r}(z_{1})$ and combining the resulting inequality with \eqref{eq:est51} in the case $1<p\leq \frac{2n}{n+2}$, after some algebraic manipulation, for every $p>1$ we get the following estimate:\medskip{}

\begin{prop}[Caccioppoli-type inequality]\label{prop:Caccioppoli}
Let $p>1$, $\gamma>0$,
$\delta>0$, $b=1+\delta$ and $0<\varepsilon<\min \{\frac{1}{2},(\frac{\delta}{4})^{p-1}\}$.
Moreover, let $\mathfrak{p}$ be defined as in $(\ref{eq:newexponent})$
and assume that 
\[
u_{\varepsilon}\in C^{0}\left([t_{0}-R^{2},t_{0}];L^{2}\left(B_{R}(x_{0}),\mathbb{R}^{N}\right)\right)\cap L^{p}\left(t_{0}-R^{2},t_{0};W^{1,p}\left(B_{R}(x_{0}),\mathbb{R}^{N}\right)\right)
\]
is the unique energy solution of problem
$(\ref{eq:CAUCHYDIR})$ with $Q'=Q_{R}(z_{0})\Subset\Omega_{T}$ and
$u$ a weak solution of $\mathrm{(\ref{eq:syst})}$, satisfying the
additional assumption of Remark~\ref{rmk2}
if $1<p\le\frac{2n}{n + 2}$. Then, for any parabolic
cylinder $Q_{r}(z_{1})\Subset Q_{R}(z_{0})$
with $r\in(0,1)$ and any cut-off functions $\eta\in C_{0}^{\infty}(B_{r}(x_{1}),[0,1])$
and $\omega\in W^{1,\infty}(\mathbb{R},[0,1])$ satisfying $\omega(t_{1}-r^{2})=0$
and $\partial_{t}\omega\geq0$, the estimate
\begin{align*} 
\frac{\gamma+p}{2} & \sup\limits_{\tau\in(t_{1}-r^{2},t_{1})}\int_{B_{r}(x_{1})\times\{\tau\}}(|Du_{\varepsilon}|^{\gamma+2}-b^{\gamma+2})_{+}  \omega^{2}\eta^{2} \,dx +
\int_{Q_{r}(z_{1})}\left|D\Big[\eta(H_{\delta}(Du_{\varepsilon}))^{\frac{\gamma+p}{2}}\Big]\right|^{2}\omega^{2}\, dz\nonumber\\
 & \leq
k (\gamma+p)^{4}\int_{Q_{r}(z_{1})}(H_{\delta}(Du_{\varepsilon}))^{\gamma+\mathfrak{p}} \big[\omega^{2}\vert D\eta\vert^{2} +\omega \partial_{t}\omega \eta^{2} +\omega^{2}\eta^{2} \big]\,dz
\\ & \qquad +
k (\gamma+p)^{4}\int_{Q_{r}(z_{1})}|f|^{2} (H_{\delta}(Du_{\varepsilon}))^{\gamma+1} \omega^{2}\eta^{2}\,dz \\
& \qquad + \mathds{1}_{(1,\frac{2n}{n+2}]}(p) \,k (\gamma+p)^{4} \,\Vert u_{\varepsilon}\Vert_{L^{\infty}(Q_{r}(z_{1}))}^{2}\int_{Q_{r}(z_{1})}(H_{\delta}(Du_{\varepsilon}))^{\gamma+\mathfrak{p}}   (\partial_{t}\omega)^{2} \eta^{2}\,dz
\end{align*}
holds true for some constant $k>1$ depending on $n$, $p$,
$\delta$, $C_{1}$ and $K$. 
\end{prop}

\subsection{Step 2: reverse Hölder-type inequalities }\label{subsec:reverse-Hölder}

Another step towards
the local boundedness of $\vert Du\vert$ consists in
using the estimate in Proposition~\ref{prop:Caccioppoli} to derive two different reverse Hölder-type inequalities, which will
in turn be employed to start the Moser iteration procedure. With this
aim in mind, in what follows we keep the assumptions, notations and
cylinders used for the proof of Proposition \ref{prop:Caccioppoli}.

Let us start with the case 
$p>\frac{2n}{n + 2}$. Applying Hölder's inequality
with exponents $\left(\frac{\hat{n}}{\hat{n}-2},\frac{\hat{n}}{2}\right)$
and the Sobolev embedding theorem on the time slices $B^{t}:=B_{r}(x_{1})\times\{t\}$
for almost every $t\in(t_{1}-r^{2},t_{1})$, we obtain 
\begin{align} 
& \fint_{B^{t}}\omega^{2+\frac{4}{\hat{n}}} \eta^{2+\frac{4}{\hat{n}}} (H_{\delta}(Du))^{\gamma+p+\frac{2 (\gamma+2)}{\hat{n}}}\,dx\nonumber\\
 & \qquad \leq
 (\omega(t))^{2+\frac{4}{\hat{n}}} \bigg[\fint_{B^{t}}\big[\eta (H_{\delta}(Du))^{\frac{\gamma+p}{2}}\big]^{\frac{2 \hat{n}}{\hat{n}-2}}\,dx\bigg]^{\frac{\hat{n}-2}{\hat{n}}}\bigg[\fint_{B^{t}}\eta^{2} (H_{\delta}(Du))^{\gamma+2}\,dx\bigg]^{\frac{2}{\hat{n}}}\nonumber\\
 & \qquad \leq C_{S} \,r^{2} \bigg[\fint_{B^{t}}\big\vert D\big[\eta (H_{\delta}(Du))^{\frac{\gamma+p}{2}}\big]\big\vert^{2} \omega^{2}\,dx\bigg]\bigg[\fint_{B^{t}}\eta^{2} (H_{\delta}(Du))^{\gamma+2} \omega^{2}\,dx\bigg]^{\frac{2}{\hat{n}}},
\label{eq:est54}
\end{align}
where $C_{S}\equiv C_{S}(N,n,p)\geq1$. Now, using the definitions
of $H_{\delta}$ and $b$, the properties of $\omega$ and $\eta$,
and applying Proposition~\ref{prop:Caccioppoli}, we can estimate
the second mean value on the right-hand side of \eqref{eq:est54}
as follows: 
\begin{align*} 
\fint_{B^{t}} \eta^{2} [H_{\delta}(Du(x,t))]^{\gamma+2}\omega^{2}\,dx 
&\leq
\frac{c(n)}{r^{n}}
\sup\limits_{\tau\in(t_{1}-r^{2},t_{1})}\int_{B_r(x_1)\times\{\tau\}}(H_{\delta}(Du))^{\gamma+2}\omega^{2}\eta^{2}\,dx\nonumber\\
 &=
 \frac{c(n)}{r^{n}}\sup\limits_{\tau\in(t_{1}-r^{2},t_{1})}\int_{B_r(x_1)\times\{\tau\}}[b^{\gamma+2}+(\vert Du\vert^{\gamma+2}-b^{\gamma+2})_{+}]\omega^{2}\eta^{2}\,dx\nonumber\\
 & \leq
 \frac{\tilde{c}(\gamma+p)^{3}}{r^{n}}\bigg[
 \mathbf{I}_{1} +
 \int_{Q_{r}(z_{1})}|f|^{2} (H_{\delta}(Du))^{\gamma+1} \omega^{2}\eta^{2}\,dz\bigg],
\end{align*}
where $\tilde{c}\equiv\tilde{c}(n,p,\delta,C_{1},K)>1$ and
\[
    \mathbf{I}_{1}
    :=
    b^{\gamma+2}\vert\mathrm{spt} \eta\vert+
 \int_{Q_{r}(z_{1})}(H_{\delta}(Du))^{\gamma+\mathfrak{p}}
 \big[\omega^{2}(\vert D\eta\vert^{2} + \eta^{2}) + \omega\partial_{t}\omega\eta^{2}\big]\,dz.
\]
Inserting the preceding inequality into \eqref{eq:est54}, integrating with respect
to time over $(t_{1}-r^{2},t_{1})$ and using Proposition~\ref{prop:Caccioppoli} again, we obtain 
\begin{align} 
\mathbf{I}_{2}
&:=
\fint_{Q_{r}(z_{1})} \omega^{2+\frac{4}{\hat{n}}} \eta^{2+\frac{4}{\hat{n}}} (H_{\delta}(Du))^{\gamma+p+\frac{2 (\gamma+2)}{\hat{n}}}\,dz\nonumber\\
 & \leq
 \frac{C_{2}(\gamma+p)^{4+\frac{6}{\hat{n}}}}{r^{n(1+\frac{2}{\hat{n}})}} \Bigg[
 \mathbf{I}_{1}^{1+\frac{2}{\hat{n}}} +
 \bigg(\int_{Q_{r}(z_{1})}|f|^{2} (H_{\delta}(Du))^{\gamma+1} \omega^{2}\eta^{2} dz\bigg)^{1+\frac{2}{\hat{n}}}\Bigg],
\label{eq:est56}
\end{align}
where $C_{2}\equiv C_{2}(N,n,\hat{n},p,\delta,C_{1},K)>1.$ For a
fixed $\mu>0$ that will be chosen later, we now split the cylinder
$Q_{r}(z_{1})$ as follows 
\[
Q_{r}(z_{1})=\left\{ z\in Q_{r}(z_{1}):\vert f\vert^{2}
\leq
\exp\left( \mu^{-1} \right)-e\right\}
 \cup 
\left\{ z\in Q_{r}(z_{1}):\vert f\vert^{2}>\exp\left( \mu^{-1} \right)-e\right\} 
=:
E_{\mu}^{1} \cup E_{\mu}^{2} ,
\]
so that 
\begin{align}
\int_{Q_{r}(z_{1})}|f|^{2}(H_{\delta}(Du))^{\gamma+1}\omega^{2}\eta^{2}\,dz & =\int_{E_{\mu}^{1}}|f|^{2}(H_{\delta}(Du))^{\gamma+1}\omega^{2}\eta^{2}\,dz+
\int_{E_{\mu}^{2}}|f|^{2}(H_{\delta}(Du))^{\gamma+1}\omega^{2}\eta^{2}\,dz\nonumber\\
 & =: 
 J_{10}+J_{11} .
\label{eq:est57}
\end{align}
From the definition of $E_{\mu}^{1}$, it immediately follows that
\begin{equation}
J_{10} \leq \exp\left(\frac{1}{\mu}\right)\int_{Q_{r}(z_{1})}(H_{\delta}(Du))^{\gamma+1} \omega^{2}\eta^{2} dz \leq \exp\left(\frac{1}{\mu}\right)\int_{Q_{r}(z_{1})}(H_{\delta}(Du))^{\gamma+\mathfrak{p}} \omega^{2}\eta^{2}\,dz .\label{eq:J10}
\end{equation}
In order to deal with the term $J_{11}$, we first observe that 
\begin{equation}
\log^{\frac{2\alpha}{\hat{n}+2}}(e+\vert f\vert^{2})
<
\log^{\frac{2\alpha}{\hat{n}+2}}\big((e+\vert f\vert)^{2}\big)
=
2^{\frac{2\alpha}{\hat{n}+2}} \log^{\frac{2\alpha}{\hat{n}+2}}(e+\vert f\vert)\label{eq:logineq}
\end{equation}
and 
\begin{equation}
(\gamma+1)\left(1+\frac{2}{\hat{n}}\right)
<
\gamma+p+\frac{2 (\gamma+2)}{\hat{n}} .\label{eq:holdineq}
\end{equation}
Using (\ref{eq:logineq}) and (\ref{eq:holdineq}), Hölder's inequality,
the definition of $E_{\mu}^{2}$ and the assumption $|f|\in L^{\hat{n}+2}\log^{\alpha}L_{loc}(\Omega_{T})$,
we obtain
\begin{align}
J_{11} & =\int_{E_{\mu}^{2}}|f|^{2} \log^{\frac{2\alpha}{\hat{n}+2}}(e+\vert f\vert^{2}) \log^{- \frac{2\alpha}{\hat{n}+2}}(e+\vert f\vert^{2}) (H_{\delta}(Du))^{\gamma+1} \omega^{2}\eta^{2}\,dz\nonumber\\
 & \leq
 2^{\frac{2\alpha}{\hat{n}+2}}\int_{E_{\mu}^{2}}|f|^{2} \log^{\frac{2\alpha}{\hat{n}+2}}(e+\vert f\vert) \log^{- \frac{2\alpha}{\hat{n}+2}}(e+\vert f\vert^{2}) (H_{\delta}(Du))^{\gamma+1} \omega^{2}\eta^{2}\,dz\nonumber\\
 & \leq
 2^{\frac{2\alpha}{\hat{n}+2}} \Bigg[\int_{E_{\mu}^{2}}|f|^{\hat{n}+2} \log^{\alpha}(e+\vert f\vert)\,dz\Bigg]^{\frac{2}{\hat{n}+2}}\cdot\nonumber\\
 & \qquad \cdot
 \Bigg[\int_{E_{\mu}^{2}}\log^{- \frac{2\alpha}{\hat{n}}}(e+\vert f\vert^{2}) (H_{\delta}(Du))^{(\gamma+1)\left(1+\frac{2}{\hat{n}}\right)} \omega^{2+\frac{4}{\hat{n}}} \eta^{2+\frac{4}{\hat{n}}} \,dz\Bigg]^{\frac{\hat{n}}{\hat{n}+2}}\nonumber\\
 & \leq
 2^{\frac{2\alpha}{\hat{n}+2}} \Vert f\Vert_{L^{\hat{n}+2}\log^{\alpha}L(Q_{r}(z_{1}))}^{\theta} \mu^{\frac{2\alpha}{\hat{n}+2}}\Bigg[\int_{E_{\mu}^{2}}(H_{\delta}(Du))^{\gamma+p+\frac{2 (\gamma+2)}{\hat{n}}} \omega^{2+\frac{4}{\hat{n}}} \eta^{2+\frac{4}{\hat{n}}}\,dz\Bigg]^{\frac{\hat{n}}{\hat{n}+2}}\nonumber\\
 & \leq
 2^{\frac{2\alpha}{\hat{n}+2}} \mu^{\frac{2\alpha}{\hat{n}+2}} \Vert f\Vert_{L^{\hat{n}+2}\log^{\alpha}L(Q_{r}(z_{1}))}^{\theta}\, 
 \Bigg[\int_{Q_{r}(z_{1})}(H_{\delta}(Du))^{\gamma+p+\frac{2 (\gamma+2)}{\hat{n}}} \omega^{2+\frac{4}{\hat{n}}} \eta^{2+\frac{4}{\hat{n}}}\,dz\Bigg]^{\frac{\hat{n}}{\hat{n}+2}},
\label{eq:J11}
\end{align}
where $\theta\equiv\theta(\hat{n})>0$. 
Joining $\eqref{eq:est56}-(\ref{eq:J10})$ and
\eqref{eq:J11}, after some algebraic manipulation we arrive at 
\begin{equation}
\mathbf{I}_{2} \leq \frac{C_{3} (\gamma+p)^{4+\frac{6}{\hat{n}}}}{r^{\frac{2 n}{\hat{n}}}} \left[r^{2} \mu^{\frac{2\alpha}{\hat{n}}} \Vert f\Vert_{L^{\hat{n}+2}\log^{\alpha}L(Q_{r}(z_{1}))}^{\frac{\theta(\hat{n}+2)}{\hat{n}}}  \,\mathbf{I}_{2}+\frac{\exp\big(\frac{1}{\mu}\left(1+\frac{2}{\hat{n}}\right)\big)}{r^{n}}  \mathbf{I}_{1}^{1+\frac{2}{\hat{n}}}\right],\label{eq:est58}
\end{equation}
where $C_{3}\equiv C_{3}(N,n,\hat{n},\alpha,p,\delta,C_{1},K)>1.$

At this point, we consider estimate \eqref{eq:est56}
in the case $\Vert f\Vert_{L^{\hat{n}+2}\log^{\alpha}L(Q_{r}(z_{1}))}=0$
and the above inequality if $\Vert f\Vert_{L^{\hat{n}+2}\log^{\alpha}L(Q_{r}(z_{1}))}>0$.
In the latter case, choosing $\mu$ such that 
\[
\frac{C_{3} (\gamma+p)^{4+\frac{6}{\hat{n}}}}{r^{2 (\frac{n}{\hat{n}}-1)}}  \mu^{\frac{2\alpha}{\hat{n}}} \Vert f\Vert_{L^{\hat{n}+2}\log^{\alpha}L(Q_{r}(z_{1}))}^{\frac{\theta(\hat{n}+2)}{\hat{n}}} = \frac{1}{2} ,
\]
from (\ref{eq:est58}) we obtain 
\begin{equation}
\mathbf{I}_{2}
\leq
C_{4}(\gamma+p)^{4+\frac{6}{\hat{n}}}
\exp\left(C_{5}(\gamma+p)^{\frac{2\hat{n}+3}{\alpha}}\right) \mathbf{I}_{1}^{1+\frac{2}{\hat{n}}},\label{eq:est59}
\end{equation}
where the constants on the right-hand side are of the type 
\begin{equation}
    C_{4}
    =
    \frac{C}{r^{n(1+\frac{2}{\hat{n}})}}
    \qquad\mathrm{and}\qquad 
    C_{5}=C\Vert f\Vert_{L^{\hat{n}+2}\log^{\alpha}L(Q_{r}(z_{1}))}^{\frac{\theta (\hat{n}+2)}{2\alpha}} r^{\frac{\hat{n}-n}{\alpha}},\label{eq:auxcons}
\end{equation}
for some constant $C\equiv C(N,n,\hat{n},\alpha,p,\delta,C_{1},K)>1$.
Moreover, from \eqref{eq:est56} it immediately follows that inequality
\eqref{eq:est59} holds true also when $\Vert f\Vert_{L^{\hat{n}+2}\log^{\alpha}L(Q_{r}(z_{1}))}=0$.

Let us now consider the
case $1<p\leq \frac{2n}{n + 2}$. Arguing exactly as
above, this time we arrive at 
\begin{align*}
    \mathbf{I}_{2}
    &\leq
    C_{4}(\gamma+p)^{4+\frac{6}{\hat{n}}}
    \exp\left(C_{5}(\gamma+p)^{\frac{2 \hat{n} +3}{\alpha}}\right) \bigg[b^{\gamma+2}\vert\mathrm{spt}\eta\vert\\
    &\qquad +\int_{Q_{r}(z_{1})}(H_{\delta}(Du))^{\gamma+\mathfrak{p}}
    \big[\omega^{2}(\vert D\eta\vert^{2} + \eta^2) + \omega\partial_{t}\omega\eta^{2} + \Vert u\Vert_{L^{\infty}(Q_{r}(z_{1}))}^{2}(\partial_{t}\omega)^{2} \eta^2\big]\,dz\bigg]^{1+\frac{2}{\hat{n}}},
\end{align*}
where the constants $C_{4}$ and $C_{5}$ are of the same type as
in (\ref{eq:auxcons}).

In summary, we have obtained the following
result:

\begin{prop}[Reverse Hölder-type inequality]
\label{prop:Reverse}
Under the assumptions of Proposition~\ref{prop:Caccioppoli}, for any parabolic cylinder $Q_{r}(z_{1})\Subset Q_{R}(z_{0})$
with $r\in(0,1)$ and any cut-off functions $\eta\in C_{0}^{\infty}(B_{r}(x_{1}),[0,1])$
and $\omega\in W^{1,\infty}(\mathbb{R},[0,1])$ satisfying $\omega(t_{1}-r^{2})=0$
as well as $\partial_{t}\omega\geq0$, the estimate
\begin{align} 
& \fint_{Q_{r}(z_{1})}  
(\omega\eta)^{2+\frac{4}{\hat{n}}}(H_{\delta}(Du_{\varepsilon}))^{\gamma+p+\frac{2(\gamma+2)}{\hat{n}}}\,dz
\leq
C(\gamma+p)^{4+\frac{6}{\hat{n}}}
\exp\left[\frac{C \Vert f\Vert_{L^{\hat{n}+2}\log^{\alpha}L(Q_{r}(z_{1}))}^{\Theta} (\gamma+p)^{\frac{2 \hat{n}+3}{\alpha}}}{r^{\frac{n-\hat{n}}{\alpha}}}\right] 
\nonumber \\ & \qquad \cdot
\bigg[r^{2}\fint_{Q_{r}(z_{1})}(H_{\delta}(Du_{\varepsilon}))^{\gamma+\mathfrak{p}}
\big[\omega^{2}(\vert D\eta\vert^{2} + \eta^2) + \omega\partial_{t}\omega\eta^{2} \big] \, dz
+
\frac{b^{\gamma+2}}{r^{n}}\vert\mathrm{spt}\eta\vert 
\nonumber \\ & \qquad\qquad + 
\mathds{1}_{(1,\frac{2n}{n+2}]}(p) \,\Vert u_{\varepsilon} \Vert_{L^{\infty}(Q_{r}(z_{1}))}^{2} \,r^{2}\fint_{Q_{r}(z_{1})} (H_{\delta}(Du_{\varepsilon}))^{\gamma+\mathfrak{p}}  (\partial_{t}\omega)^{2}\eta^{2} \, dz\bigg]^{1+\frac{2}{\hat{n}}}
\label{eq:Hold2}
\end{align} holds true for some constants $C\equiv C(N,n,\hat{n},\alpha,p,\delta,C_{1},K)>1$ and $\Theta\equiv\Theta(\hat{n},\alpha)>0$.
\end{prop}

\subsection{Step 3: the iteration for \texorpdfstring{$\boldsymbol{p>\frac{2n}{n+2}}$}{p>2n/(n+2)} \label{subsec:iteration1}}

Thanks to Proposition \ref{prop:Reverse}, we can now start the Moser iteration procedure. Once again, we shall keep both the assumptions and the notations used in Propositions \ref{prop:Caccioppoli} and \ref{prop:Reverse}. At this step, we assume that $p>\frac{2n}{n+2}$. By virtue of (\ref{eq:nhat}) and (\ref{eq:betacondition}), this implies that $p>\frac{2 \hat{n}}{\hat{n}+2}$. We define by induction a sequence $\{\gamma_{k}\}_{k\in\mathbb{N}_{0}}$ by
letting $\gamma_{0}=0$ and 
\[
\gamma_{k}:=\begin{cases}
\left(1+\frac{2}{\hat{n}}\right)\gamma_{k-1}+\frac{4}{\hat{n}}-2+p & \text{if }\frac{2 \hat{n}}{\hat{n}+2}<p<2\\
\left(1+\frac{2}{\hat{n}}\right)\gamma_{k-1}+\frac{4}{\hat{n}} & \text{if }  p\geq2 
\end{cases}
\qquad ,\,k\in\mathbb{N}.
\]
Notice that $\gamma_{k}>0$ for every $k\in\mathbb{N}$, since $p>\frac{2 \hat{n}}{\hat{n}+2}$.
Furthermore, this sequence diverges to $+\infty$ as $k\rightarrow\infty$
and by induction we have 
\[
\gamma_{k}=\begin{cases}
\left[2-\frac{\hat{n} (2-p)}{2}\right]\left[\left(1+\frac{2}{\hat{n}}\right)^{k}-1\right] & \text{if }\frac{2 \hat{n}}{\hat{n}+2}<p<2\\[9pt]
2\left[\left(1+\frac{2}{\hat{n}}\right)^{k}-1\right] & \text{if }  p\geq2
\end{cases}
\quad ,\,k\in\mathbb{N}_{0}.
\]
In addition, one can easily check that 
\begin{equation}
\gamma_{k}+p+\frac{2 (\gamma_{k}+2)}{\hat{n}} = \gamma_{k+1}+\mathfrak{p} \qquad \forall k\in\mathbb{N}_{0}, \forall  p>\frac{2 \hat{n}}{\hat{n}+2} .\label{eq:seq1}
\end{equation}
Now, for $k\in\mathbb{N}_{0}$ and $s\in(0,1)$ we set 
\[
r_{k}:=s r+\frac{(1-s) r}{2^{k}},\qquad \mathcal{B}_{k}:=B_{r_{k}}(x_{1})\qquad\mathrm{and}\qquad\mathcal{Q}_{k}:=Q_{r_{k}}(z_{1})=\mathcal{B}_{k}\times(t_{1}-r_{k}^{2},t_{1})
\]
and choose cut-off functions $\eta_{k}\in C_{0}^{\infty}(\mathcal{B}_{k},[0,1])$
such that $\eta_{k}\equiv1$ on $\mathcal{B}_{k+1}$ and $\vert D\eta_{k}\vert\leq\frac{2^{k+2}}{(1-s) r}$
and $\omega_{k}\in W^{1,\infty}((t_{1}-r_{k}^{2},t_{1}),[0,1])$ such
that $\omega_{k}(t_{1}-r_{k}^{2})=0$, $\omega_{k}\equiv1$ on $(t_{1}-r_{k+1}^{2},t_{1})$
and $0\leq\partial_{t} \omega_{k}\leq\frac{2^{2k+2}}{(1-s)^{2} r^{2}}$.
Note that this is possible since 
\[
r_{k}^{2}-r_{k+1}^{2}=
\frac{3 (1-s)^{2} r^{2}}{2^{2k+2}}+\frac{(1-s) s r^{2}}{2^{k}}
>
\frac{(1-s)^{2} r^{2}}{2^{2k+2}}.
\]
Using equality (\ref{eq:seq1}) and replacing $\gamma$, $\eta$ and
$\omega$ with $\gamma_{k}$, $\eta_{k}$ and $\omega_{k}$ respectively,
we obtain from \eqref{eq:Hold2} the following
\textit{recursive reverse Hölder-type inequality} 
\begin{align*}
 \fint_{Q_{r}(z_{1})} &(\omega_{k}\eta_{k})^{2+\frac{4}{\hat{n}}} (H_{\delta}(Du_{\varepsilon}))^{\gamma_{k+1}+\mathfrak{p}}\,dz\\
 &\leq
 C (\gamma_{k}+p)^{4+\frac{6}{\hat{n}}}
 \exp\left[\frac{C\Vert f\Vert_{L^{\hat{n}+2}\log^{\alpha}L(Q_{r}(z_{1}))}^{\Theta}(\gamma_{k}+p)^{\frac{2\hat{n}+3}{\alpha}}}{r^{\frac{n-\hat{n}}{\alpha}}}\right] \\
 & \quad\cdot
 \bigg[\frac{b^{\gamma_{k}+2}}{r^{n}} \vert\mathrm{spt}\eta_{k}\vert +
 r^{2}\fint_{Q_{r}(z_{1})}(H_{\delta}(Du_{\varepsilon}))^{\gamma_{k}+\mathfrak{p}}
 \big[\omega_{k}^{2}(\vert D\eta_{k}\vert^{2} + \eta_k^2) + \omega_{k}\partial_{t}\omega_{k}\eta_{k}^{2} \big]\,dz\bigg]^{1+\frac{2}{\hat{n}}},
\end{align*}
which holds for any $k\in\mathbb{N}_{0}$. Exploiting the properties
of $\eta_{k}$ and $\omega_{k}$ as well as the definitions of $r_{k}$,
$\mathcal{Q}_{k}$, $b$, $H_{\delta}$ and $\mathfrak{p}$, we obtain for every
$k\in\mathbb{N}_{0}$ that
\begin{align}\frac{b^{\gamma_{k}+2}}{r^{n}} \vert\mathrm{spt \ } \eta_{k}\vert  & \leq \frac{1}{r^{n} r_{k}^{2}}\int_{\mathcal{Q}_{k}}b^{\gamma_{k}+2}\,dz
\leq
\frac{4^{k}}{(1-s)^{2} r^{n+2}}\int_{\mathcal{Q}_{k}}(1+\delta)^{\gamma_{k}+p-p+2}\,dz\nonumber\\
 & \leq \frac{4^{k} r_{k}^{n+2} c(n,p,\delta)}{(1-s)^{2} r^{n+2}} \fint_{\mathcal{Q}_{k}}(H_{\delta}(Du_{\varepsilon}))^{\gamma_{k}+\mathfrak{p}}\,dz
\label{eq:ite0}
\end{align}
and 
\begin{align*}
 r^{2}\fint_{Q_{r}(z_{1})}(H_{\delta}(Du_{\varepsilon}))^{\gamma_{k}+\mathfrak{p}}
 \big[\omega_{k}^{2}(\vert D\eta_{k}\vert^{2} + \eta_k^2) + \omega_{k}\partial_{t}\omega_{k}\eta_{k}^{2} \big]\,dz
 \leq
 \frac{4^{k+4}r_{k}^{n+2}}{(1-s)^{2}r^{n+2}}\fint_{\mathcal{Q}_{k}}(H_{\delta}(Du_{\varepsilon}))^{\gamma_{k}+\mathfrak{p}}\,dz ,
\end{align*}
where, in the last line, we have used the fact that
\begin{equation}
    \omega_{k}^{2}(\vert D\eta_{k}\vert^{2} + \eta_k^2) 
    \leq
    \frac{2^{2k+4}}{(1-s)^{2}r^{2}} + 1
    \le
    \frac{4^{k+3}}{(1-s)^{2}r^{2}} .\label{eq:etagr2}
\end{equation}
In addition, one can easily deduce that 
\begin{equation}
\fint_{\mathcal{Q}_{k+1}}(H_{\delta}(Du_{\varepsilon}))^{\gamma_{k+1}+\mathfrak{p}}\,dz
\leq
\frac{r^{n+2}}{r_{k+1}^{n+2}} \fint_{Q_{r}(z_{1})}(\omega_{k}\eta_{k})^{2+\frac{4}{\hat{n}}} (H_{\delta}(Du_{\varepsilon}))^{\gamma_{k+1}+\mathfrak{p}}\,dz\label{eq:lefthand}
\end{equation}
for any $k\in\mathbb{N}_{0}$. Joining the last four inequalities
and using the fact that $r_{k}=2r_{k+1}-sr\leq2r_{k+1}$
and $r_{k}\leq r$ for every $k\in\mathbb{N}_{0}$, we find 
\begin{align}
    \fint_{\mathcal{Q}_{k+1}}(H_{\delta}(Du_{\varepsilon}))^{\gamma_{k+1}+\mathfrak{p}}\,dz
    \leq \, & 
    C(\gamma_{k}+p)^{4+\frac{6}{\hat{n}}}
    \exp\left[\frac{C\Vert f\Vert_{L^{\hat{n}+2}\log^{\alpha}L(Q_{r}(z_{1}))}^{\Theta} (\gamma_{k}+p)^{\frac{2 \hat{n}+3}{\alpha}}}{r^{\frac{n-\hat{n}}{\alpha}}}\right] \nonumber\\
 & \cdot\bigg[\frac{4^{k}}{(1-s)^{2}} \fint_{\mathcal{Q}_{k}}(H_{\delta}(Du_{\varepsilon}))^{\gamma_{k}+\mathfrak{p}}\,dz\bigg]^{1+\frac{2}{\hat{n}}}
\label{eq:ite1}
\end{align}
for any $k\in\mathbb{N}_{0}$, where $C\equiv C(N,n,\hat{n},\alpha,p,\delta,C_{1},K)>1$ and $\Theta\equiv\Theta(\hat{n},\alpha)>0$.
To shorten our notation, we now set 
\begin{equation}
M_{k}:=\fint_{\mathcal{Q}_{k}}(H_{\delta}(Du_{\varepsilon}))^{\gamma_{k}+\mathfrak{p}}\,dz,\qquad k\in\mathbb{N}_{0} ,\label{eq:mean_value}
\end{equation}
\begin{equation}
c_{s}
:=
\frac{C}{(1-s)^{2+\frac{4}{\hat{n}}}}
\qquad\mathrm{and}\qquad 
c_{r}
:=
C\Vert f\Vert_{L^{\hat{n}+2}\log^{\alpha}L(Q_{r}(z_{1}))}^{\Theta} r^{\frac{\hat{n}-n}{\alpha}},\label{eq:cscr}
\end{equation}
so that the last inequality turns into 
\[
M_{k+1}
\leq
c_{s}(\gamma_{k}+p)^{4+\frac{6}{\hat{n}}} 
\,4^{k(1+\frac{2}{\hat{n}})}
\exp\left(c_{r}(\gamma_{k}+p)^{\frac{2\hat{n}+3}{\alpha}}\right) 
M_{k}^{1+\frac{2}{\hat{n}}}
\]
for $k\in\mathbb{N}_{0}$. Iterating the above estimate, we obtain 
\[
M_{k}
\le
\prod_{j=0}^{k-1}
\left[c_{s}(\gamma_{j}+p)^{4+\frac{6}{\hat{n}}}\,4^{j(1+\frac{2}{\hat{n}})}
\exp\left(c_{r}(\gamma_{j}+p)^{\frac{2\hat{n}+3}{\alpha}}\right)\right]^{(1+\frac{2}{\hat{n}})^{k-1-j}}M_{0}^{(1+\frac{2}{\hat{n}})^{k}}
\]
for any $k\in\mathbb{N}$. This inequality can be rewritten as follows:
\begin{equation}
M_{k}^{\frac{1}{\gamma_{k}+\mathfrak{p}}}
\le
\prod_{j=0}^{k-1}\left[c_{s}(\gamma_{j}+p)^{4+\frac{6}{\hat{n}}}\,4^{j(1+\frac{2}{\hat{n}})}
\exp\left(c_{r}(\gamma_{j}+p)^{\frac{2\hat{n}+3}{\alpha}}\right)\right]^{\frac{(1+\frac{2}{\hat{n}})^{k-1-j}}{\gamma_{k}+\mathfrak{p}}}M_{0}^{\frac{(1+\frac{2}{\hat{n}})^{k}}{\gamma_{k}+\mathfrak{p}}}.
\label{eq:recursive_ineq}
\end{equation}
Next, we observe that 
\begin{align}
\lim_{k\rightarrow\infty}\frac{\left(1+\frac{2}{\hat{n}}\right)^{k}}{\gamma_{k}+\mathfrak{p}} & = \lim_{k\rightarrow\infty}
\begin{cases}
\frac{\left(1+\frac{2}{\hat{n}}\right)^{k}}{2+\left[2-\frac{\hat{n} (2-p)}{2}\right]\left[\left(1+\frac{2}{\hat{n}}\right)^{k}-1\right]} & \mathrm{if \ }\frac{2 \hat{n}}{\hat{n}+2}<p<2\\[10pt]
\frac{\left(1+\frac{2}{\hat{n}}\right)^{k}}{p+2\left[\left(1+\frac{2}{\hat{n}}\right)^{k}-1\right]} & \mathrm{if \ }p\geq2
\end{cases} \nonumber\\
&=\begin{cases}
\frac{2}{4-\hat{n} (2-p)} & \text{if }\frac{2 \hat{n}}{\hat{n}+2}<p<2\\[5pt]
\frac12 & \text{if } p\geq2,
\end{cases}\label{eq:ite3}
\end{align}
so that 
\begin{equation}
\lim_{k\rightarrow\infty} M_{0}^{\frac{(1+\frac{2}{\hat{n}})^{k}}{\gamma_{k}+\mathfrak{p}}}=\left[\fint_{Q_{r}(z_{1})}(H_{\delta}(Du_{\varepsilon}))^{\mathfrak{p}}\,dz\right]^{\frac{1}{\varphi}},\label{eq:ite4}
\end{equation}
where 
\begin{equation}
\varphi:=\begin{cases}
2-\frac{\hat{n} (2-p)}{2} & \mathrm{if\ }\frac{2 \hat{n}}{\hat{n}+2}<p<2\\
2 & \mathrm{if \ } p\geq2.
\end{cases}\label{eq:ite5}
\end{equation}
One can easily check that $\mathfrak{p}-\varphi\geq0$ and $\varphi\leq p$
whenever $p>\frac{2 \hat{n}}{\hat{n}+2}$. Using this information
and the fact that $c_{s}>1$, we can apply inequality (\ref{eqit1}) to obtain the following estimate:
\begin{equation}
\prod_{j=0}^{k-1}c_{s}^{\frac{\left(1+\frac{2}{\hat{n}}\right)^{k-1-j}}{\gamma_{k}+\mathfrak{p}}} 
\leq c_{s}^{\frac{\hat{n}}{2 \varphi}}.
\label{eq:product1}
\end{equation}
Similarly, by inequality (\ref{eqit2}) we get
\begin{equation}
\prod_{j=0}^{k-1}4^{j \frac{\left(1+\frac{2}{\hat{n}}\right)^{k-j}}{\gamma_{k}+\mathfrak{p}}} \leq 2^{\frac{\hat{n} (\hat{n}+2)}{2 \varphi}}.
\label{eq:product2}
\end{equation}
Now we use that $1<\gamma_{j}+p\leq p (\frac{\Hat{n}+2}{\Hat{n}})^{j}$ for any $j\in\mathbb{N}_{0}$ to
estimate the following product by means of (\ref{eqit1}) and (\ref{eqit2}):
\begin{equation}
\prod_{j=0}^{k-1}\left[(\gamma_{j}+p)^{4+\frac{6}{\hat{n}}}\right]^{\frac{\left(1+\frac{2}{\hat{n}}\right)^{k-1-j}}{\gamma_{k}+\mathfrak{p}}}
\leq
\prod_{j=0}^{k-1} \bigg[ p\bigg( \frac{\Hat{n}+2}{\Hat{n}}\bigg)^{j} \bigg] ^{4 \frac{\left(1+\frac{2}{\hat{n}}\right)^{k-j}}{\gamma_{k}+\mathfrak{p}}}
\leq
p^{\frac{2 (\hat{n}+2)}{\varphi}}\left(\frac{\hat{n}+2}{\hat{n}}\right)^{\frac{\hat{n} (\hat{n}+2)}{\varphi}}.
\label{eq:product2_bis}
\end{equation}
It remains to estimate the product involving the exponential terms with inequality (\ref{eqit1}).
Recalling that $\frac{2\hat{n}+3}{\alpha}<1$, we obtain 
\begin{align}\label{eq:product4}
&\prod_{j=0}^{k-1}\left[\exp\left(c_{r} (\gamma_{j} + p)^{\frac{2 \hat{n} + 3}{\alpha}}\right)\right]^{\frac{\left(1 + \frac{2}{\hat{n}}\right)^{k - 1 - j}}{\gamma_{k} + \mathfrak{p}}}
\leq
\prod_{j=0}^{k-1}\left[\exp\left(c_{r}\, p^{\frac{2 \hat{n}+3}{\alpha}} \Big(\frac{\hat{n} + 2}{\Hat{n}}\Big)^{j\,\frac{2 \hat{n}+3}{\alpha}}\right) \right]^{\frac{\left(1+\frac{2}{\hat{n}}\right)^{k-1-j}}{\gamma_{k}+\mathfrak{p}}} 
\nonumber\\ & \qquad \leq
\exp\left[\frac{c_{r}}{\gamma_{k}+\mathfrak{p}}  \,p^{\frac{2 \hat{n}+3}{\alpha}} \left(\frac{\hat{n}+2}{\hat{n}}\right)^{k-1} \frac{(\hat{n}+2)^{1-\frac{2 \hat{n}+3}{\alpha}}}{(\hat{n}+2)^{1-\frac{2 \hat{n}+3}{\alpha}}-(\hat{n})^{1-\frac{2 \hat{n}+3}{\alpha}}}\right].
\end{align}
Combining estimates \eqref{eq:product1}$-$\eqref{eq:product4} and
recalling the definitions of $c_{s}$ and $c_{r}$ in (\ref{eq:cscr}), we obtain 
from (\ref{eq:recursive_ineq}) that 
\begin{equation}\label{eq:ite6}
	M_{k}^{\frac{1}{\gamma_{k}+\mathfrak{p}}}
	\leq
	\frac{c}{(1-s)^{\frac{\hat{n}+2}{\varphi}}} \cdot \exp\left[c\,  \frac{\Vert f\Vert_{L^{\hat{n}+2}\log^{\alpha}L(Q_{r}(z_{1}))}^{\Theta}}{r^{\frac{n-\hat{n}}{\alpha}}} \,\cdot\,\frac{\left(1+\frac{2}{\hat{n}}\right)^{k}}{\gamma_{k}+\mathfrak{p}}\right]
M_{0}^{\frac{\left(1+\frac{2}{\hat{n}}\right)^{k}}{\gamma_{k}+\mathfrak{p}}}
\end{equation}
for any $k\in\mathbb{N}$, with $c\equiv c(N,n,\hat{n},\alpha,p,\delta,C_{1},K)>0$.
Since the constant $c$ is independent of $k\in\mathbb{N}$, recalling
(\ref{eq:mean_value}), (\ref{eq:ite3}), (\ref{eq:ite4}) and (\ref{eq:ite5})
and passing to the limit as $k\rightarrow\infty$ in both sides of
(\ref{eq:ite6}), we arrive at 
\begin{equation}
\underset{Q_{sr}(z_{1})}{\mathrm{ess}\, \sup}  \,H_{\delta}(Du_{\varepsilon})
\leq
\frac{c}{(1-s)^{\frac{\hat{n}+2}{\varphi}}}\cdot
\exp\left(\frac{c\,\Vert f\Vert_{L^{\hat{n}+2}\log^{\alpha}L(Q_{r}(z_{1}))}^{\Theta}}{r^{\frac{n-\hat{n}}{\alpha}}}\right)\left[\fint_{Q_{r}(z_{1})}(H_{\delta}(Du_{\varepsilon}))^{\mathfrak{p}}\,dz\right]^{\frac{1}{\varphi}},\label{eq:ite7}
\end{equation}
due to the fact that $r_{k}\searrow s r$ and $\gamma_{k}\nearrow\infty$
as $k\rightarrow\infty$. At this stage we need to separate the cases
$\frac{2 \hat{n}}{\hat{n}+2}<p<2$ and $p\geq2$.\\
 $\hspace*{1em}$Let us first assume that $\frac{2 \hat{n}}{\hat{n}+2}<p<2$.
In this case we have 
\[
\mathfrak{p}=2>p ,\qquad\varphi=2- \frac{\hat{n} (2-p)}{2}\qquad \mathrm{and}\qquad \frac{\varphi}{2-p} >1.
\]
From Theorem \ref{thm:regul} we know that $H_{\delta}(Du_{\varepsilon})\in L_{loc}^{\infty}(Q_{R}(z_{0}))$.
Therefore, from (\ref{eq:ite7}) we deduce
\begin{align*}
    &\underset{Q_{sr}(z_{1})}{\mathrm{ess}\,\sup} \,H_{\delta}(Du_{\varepsilon}) \\
    &\quad \leq
    \frac{c}{(1-s)^{\frac{\hat{n}+2}{\varphi}}}
    \exp\left(\frac{c\,\Vert f\Vert_{L^{\hat{n}+2}\log^{\alpha}L(Q_{r}(z_{1}))}^{\Theta}}{r^{\frac{n-\hat{n}}{\alpha}}}\right)
    \underset{Q_{r}(z_{1})}{\mathrm{ess}\,\sup}\,(H_{\delta}(Du_{\varepsilon}))^{\frac{2-p}{\varphi}}\left[\fint_{Q_{r}(z_{1})}(H_{\delta}(Du_{\varepsilon}))^{p}\,dz\right]^{\frac{1}{\varphi}} \\
    &\quad\le 
    \frac{1}{2}\, \underset{Q_{r}(z_{1})}{\mathrm{ess}\,\sup}\,H_{\delta}(Du_{\varepsilon}) \\ 
    &\qquad +
    \frac{c}{(1-s)^{\frac{2\hat{n}+4}{p(\hat{n}+2)-2\hat{n}}}}
    \exp\left(\frac{c\,\Vert f\Vert_{L^{\hat{n}+2}\log^{\alpha}L(Q_{r}(z_{1}))}^{\Theta}}{r^{\frac{n-\hat{n}}{\alpha}}}\right)\left[\fint_{Q_{r}(z_{1})}(H_{\delta}(Du_{\varepsilon}))^{p}\,dz\right]^{\frac{2}{p(\hat{n}+2)-2\hat{n}}},
\end{align*}
where, in the second to last line, we have applied Young's inequality with exponents $(\frac{\varphi}{2-p},\frac{\varphi}{\varphi-2+p})$.
Note that the preceding inequality holds for any $s\in(0,1)$. Hence, we can absorb the essential
supremum on the right-hand side using Lemma \ref{lem:Giusti} with
$\rho_{0}=sr$ and $\rho_{1}=r$. This yields 
\begin{align*}
\underset{Q_{sr}(z_{1})}{\mathrm{ess}\,\sup}\,\,H_{\delta}(Du_{\varepsilon})
\leq
\frac{c}{(1-s)^{\frac{2\hat{n}+4}{p(\hat{n}+2)-2\hat{n}}}}
\exp\left(\frac{c\Vert f\Vert_{L^{\hat{n}+2}\log^{\alpha}L(Q_{r}(z_{1}))}^{\Theta}}{r^{\frac{n-\hat{n}}{\alpha}}}\right)
\left[\fint_{Q_{r}(z_{1})}(H_{\delta}(Du_{\varepsilon}))^{p}\,dz\right]^{\frac{2}{p(\hat{n}+2)-2\hat{n}}}
\end{align*}
for a positive constant $c$ depending on $N,n,\hat{n},\alpha,p,\delta,C_{1}$
and $K$, but not on $\varepsilon$.

Finally, when $p\geq2$, inequality (\ref{eq:ite7})
reads as 
\[
\underset{Q_{sr}(z_{1})}{\mathrm{ess}\,\sup}\,H_{\delta}(Du_{\varepsilon})
\leq
\frac{c}{(1-s)^{\frac{\hat{n}+2}{2}}}\exp\left(\frac{c\,\Vert f\Vert_{L^{\hat{n}+2}\log^{\alpha}L(Q_{r}(z_{1}))}^{\Theta}}{r^{\frac{n-\hat{n}}{\alpha}}}\right)\left[\fint_{Q_{r}(z_{1})}(H_{\delta}(Du_{\varepsilon}))^{p}\,dz\right]^{\frac{1}{2}}.
\]
We have thus proved the following result, which ensures the desired
local $L^{\infty}$-bound of $Du_{\varepsilon}$ for all $p>\frac{2 n}{n+2}$.

\begin{thm}\label{thm:ite_theo1}
Let $p>\frac{2 n}{n+2}$,
$\delta>0$ and $0<\varepsilon<\min \{\frac{1}{2},(\frac{\delta}{4})^{p-1}\}$,
where $\hat{n}$ is defined according to $(\ref{eq:nhat})-(\ref{eq:betacondition})$. Moreover,
assume that 
\[
u_{\varepsilon}\in C^{0}\left([t_{0}-R^{2},t_{0}];L^{2}\left(B_{R}(x_{0}),\mathbb{R}^{N}\right)\right)\cap L^{p}\left(t_{0}-R^{2},t_{0};W^{1,p}\left(B_{R}(x_{0}),\mathbb{R}^{N}\right)\right)
\]
is the unique energy solution of problem
$(\ref{eq:CAUCHYDIR})$ with $Q'=Q_{R}(z_{0})\Subset\Omega_{T}$ and
$u$ a weak solution of $\mathrm{(\ref{eq:syst})}$. Then, for any
parabolic cylinder $Q_{r}(z_{1})\Subset Q_{R}(z_{0})$
with $r\in(0,1)$ and any $s\in(0,1)$, we have 
\begin{equation}
\underset{Q_{sr}(z_{1})}{\mathrm{ess}\,\sup} \,H_{\delta}(Du_{\varepsilon}) \leq \frac{c}{(1-s)^{\frac{\hat{n}+2}{\kappa}}} \cdot \exp\left(c\,\frac{\Vert f\Vert_{L^{\hat{n}+2}\log^{\alpha}L(Q_{r}(z_{1}))}^{\Theta}}{r^{\frac{n-\hat{n}}{\alpha}}}\right)\left[\fint_{Q_{r}(z_{1})}(H_{\delta}(Du_{\varepsilon}))^{p}\,dz\right]^{\frac{1}{\kappa}}\label{eq:ite_theo1_est}
\end{equation}
for some positive constants $c\equiv c(N,n,\hat{n},\alpha,p,\delta,C_{1},K)$
and $\Theta\equiv\Theta(\hat{n},\alpha)$, and for 
\begin{equation}
\kappa:=\begin{cases}
\frac{p (\hat{n}+2)-2 \hat{n}}{2} & \text{if }\frac{2 n}{n+2}<p<2\\
2 & \text{if } p\geq2.\end{cases}\label{eq:kappa}
\end{equation}
\end{thm}

\subsection{Step 4: the iteration for \texorpdfstring{$\boldsymbol{1<p\protect\leq\frac{2n}{n+2}}$}{1<p<=2n/(n+2)} }\label{subsec:iteration2}

We now start the iteration procedure for
$1<p\leq \frac{2n}{n+2}$, keeping both the assumptions
and the notations used in Propositions \ref{prop:Caccioppoli}
and \ref{prop:Reverse}. Note that, by the choice of $\beta$ in (\ref{eq:betacondition}), the condition $1<p\leq \frac{2n}{n+2}$ implies that $n\geq3$, and therefore $\Hat{n}=n$. We will however restrict ourselves to the case $f=0$. We again define by induction a sequence
$\{\gamma_{k}\}_{k\in\mathbb{N}_{0}}$ by letting $\gamma_{0}=0$
and 
\[
\gamma_{k}:=\left(1+\frac{2}{n}\right)\gamma_{k-1}+ \frac{2 p}{n} ,\qquad k\in\mathbb{N}.
\]
This sequence diverges to $+\infty$ as $k\rightarrow\infty$ and,
by induction, we have 
\[
\gamma_{k}=p\left[\left(1+\frac{2}{n}\right)^{k}-1\right],\qquad k\in\mathbb{N}_{0}.
\]
Moreover, one can easily check that 
\begin{equation}
\gamma_{k}+p+\frac{2 (\gamma_{k}+2)}{n} > \gamma_{k+1}+p \qquad
\forall  k\in\mathbb{N}_{0}, \forall  p\in\left(1, \frac{2n}{n+2}\right].\label{eq:seq1-1}
\end{equation}
Now, for $k\in\mathbb{N}_{0}$ and $s\in(0,1)$ we define $r_{k}$,
$\mathcal{B}_{k}$, $\mathcal{Q}_{k}$ and the cut-off functions $\eta_{k}$
and $\omega_{k}$ as in Section \ref{subsec:iteration1}. Using inequality
(\ref{eq:seq1-1}) and replacing $\gamma$, $\eta$ and $\omega$
with $\gamma_{k}$, $\eta_{k}$ and $\omega_{k}$ respectively, we obtain from
\eqref{eq:Hold2} the following
\textit{recursive reverse Hölder-type inequality} 
\begin{align*}
&\fint_{Q_{r}(z_{1})} 
(\omega_{k}\eta_{k})^{2+\frac{4}{n}}(H_{\delta}(Du_{\varepsilon}))^{\gamma_{k+1}+p}\,dz \nonumber\\
&\ \ \leq
C (\gamma_{k}+p)^{4+\frac{6}{n}}
\bigg[\frac{b^{\gamma_{k}+2}}{r^{n}}\vert\mathrm{spt}\eta_{k}\vert
\nonumber\\
&\ \ \quad +
r^{2}\fint_{Q_{r}(z_{1})}(H_{\delta}(Du_{\varepsilon}))^{\gamma_{k}+p}
\big[(\vert D\eta_{k}\vert^{2}+\eta_{k}^{2})\omega_{k}^{2} + \omega_{k}\partial_{t}\omega_{k}\eta_{k}^{2} +
\Vert u_{\varepsilon}\Vert_{L^{\infty}(Q_{r}(z_{1}))}^{2}(\partial_{t}\omega_{k})^{2}\eta_{k}^{2}\big]\,dz\bigg]^{1+\frac{2}{n}},
\end{align*}
which holds for any $k\in\mathbb{N}_{0}$. Exploiting the properties
of $\eta_{k}$ and $\omega_{k}$ as well as the definition of $\mathcal{Q}_{k}$,
for every $k\in\mathbb{N}_{0}$ we get 
\begin{align*}
& r^{2}\fint_{Q_{r}(z_{1})}(H_{\delta}(Du_{\varepsilon}))^{\gamma_{k}+p}
\big[(\vert D\eta_{k}\vert^{2}+\eta_{k}^{2})\omega_{k}^{2} + \omega_{k}\partial_{t}\omega_{k}\eta_{k}^{2} +
\Vert u_{\varepsilon}\Vert_{L^{\infty}(Q_{r}(z_{1}))}^{2}(\partial_{t}\omega_{k})^{2}\eta_{k}^{2}\big]\,dz \nonumber\\
 & \qquad \leq
 \frac{4^{2k+3}r_{k}^{n+2}}{(1-s)^{2}r^{n+2}}
 \left[1+ \frac{\Vert u_{\varepsilon}\Vert_{L^{\infty}(Q_{r}(z_{1}))}^{2}}{(1-s)^{2}r^{2}}\right]\fint_{\mathcal{Q}_{k}}(H_{\delta}(Du_{\varepsilon}))^{\gamma_{k}+p}\,dz,
\end{align*}
where, in the last line, we have used (\ref{eq:etagr2}). Combining the two previous inequalities with 
\eqref{eq:ite0} and (\ref{eq:lefthand}), and using the fact that
$r_{k}\leq 2 r_{k+1}$ and $r_{k}\leq r$ for every $k\in\mathbb{N}_{0}$,
we find 
\begin{align*}
\fint_{\mathcal{Q}_{k+1}}(H_{\delta}(Du_{\varepsilon}))^{\gamma_{k+1}+p}\,dz  & 
\leq 
C (\gamma_{k}+p)^{4+\frac{6}{n}}\\
& \qquad \cdot
\left[\frac{16^{k}}{(1-s)^{2}}\left(1+\frac{\Vert u_{\varepsilon}\Vert_{L^{\infty}(Q_{r}(z_{1}))}^{2}}{(1-s)^{2}r^{2}}\right)\fint_{\mathcal{Q}_{k}}(H_{\delta}(Du_{\varepsilon}))^{\gamma_{k}+p}\,dz\right]^{1+\frac{2}{n}}
\end{align*}
for any $k\in\mathbb{N}_{0}$, where $C\equiv C(N,n,\alpha,p,\delta,C_{1},K)>1$ and $\Theta\equiv\Theta(n,\alpha)>0$.
To shorten our notation, we now set 
\begin{equation}
c_{\varepsilon} := C\left[\frac{1}{(1-s)^{2}} \left(1+ \frac{\Vert u_{\varepsilon}\Vert_{L^{\infty}(Q_{r}(z_{1}))}^{2}}{(1-s)^{2} r^{2}}\right)\right]^{1+\frac{2}{n}}.
\label{eq:c_eps}
\end{equation}
Thus, recalling the definition (\ref{eq:mean_value}), the preceding
inequality turns into 
\[
M_{k+1} \leq c_{\varepsilon}(\gamma_{k}+p)^{4+\frac{6}{n}} 16^{k (1+\frac{2}{n})} M_{k}^{1+\frac{2}{n}}
\]
for $k\geq0$. Note that the constant $c_{\varepsilon}$ depends on
$\varepsilon$ and $r$ through the quotient 
\[
\frac{\Vert u_{\varepsilon}\Vert_{L^{\infty}(Q_{r}(z_{1}))}^{2}}{r^{2}} .
\]
Iterating the above estimate, we obtain 
\[
M_{k} \le \prod_{j=0}^{k-1}\left[c_{\varepsilon}(\gamma_{j}+p)^{4+\frac{6}{n}} 16^{j (1+\frac{2}{n})}\right]^{(1+\frac{2}{n})^{k-1-j}}M_{0}^{(1+\frac{2}{n})^{k}}
\]
for any $k\in\mathbb{N}$. This inequality can be rewritten as follows:
\begin{equation}
M_{k}^{\frac{1}{\gamma_{k}+p}} \le \prod_{j=0}^{k-1}\left[c_{\varepsilon} (\gamma_{j}+p)^{4+\frac{6}{n}} 16^{j (1+\frac{2}{n})}\right]^{\frac{\left(1+\frac{2}{n}\right)^{k-1-j}}{\gamma_{k}+p}}M_{0}^{\frac{\left(1+\frac{2}{n}\right)^{k}}{\gamma_{k}+p}}.\label{eq:recursive_ineq-1}
\end{equation}
Next, we observe that 
\begin{equation}
\lim_{k\rightarrow\infty}\frac{\left(1+\frac{2}{n}\right)^{k}}{\gamma_{k}+p} = \lim_{k\rightarrow\infty} \frac{\left(1+\frac{2}{n}\right)^{k}}{p+p\left[\left(1+\frac{2}{n}\right)^{k}-1\right]} = \frac{1}{p} ,\label{eq:iitt3}
\end{equation}
so that 
\begin{equation}
\lim_{k\rightarrow\infty} M_{0}^{\frac{\left(1+\frac{2}{n}\right)^{k}}{\gamma_{k}+p}}=\left[\fint_{Q_{r}(z_{1})}(H_{\delta}(Du_{\varepsilon}))^{p}\,dz\right]^{\frac{1}{p}}.\label{eq:iitt4}
\end{equation}
Since $c_{\varepsilon}>1$, we can apply inequality (\ref{eqit1}) to obtain the following estimate: 
\begin{equation}
\prod_{j=0}^{k-1}c_{\varepsilon}^{\frac{\left(1+\frac{2}{n}\right)^{k-1-j}}{\gamma_{k}+p}} 
\leq
c_{\varepsilon}^{\frac{n}{2 p}}.
\label{eq:prod1bis}
\end{equation}
Similarly, by inequality (\ref{eqit2}) we get
\begin{equation}
\prod_{j=0}^{k-1}16^{j \frac{\left(1+\frac{2}{n}\right)^{k-j}}{\gamma_{k}+p}} 
\leq
2^{\frac{n (n+2)}{p}}.
\label{eq:prod2bis}
\end{equation}
Now we use that $1<\gamma_{j}+p= p (\frac{n+2}{n})^{j}$ for any $j\in\mathbb{N}_{0}$ to
estimate the following product by means of (\ref{eqit1}) and (\ref{eqit2}):
\begin{align}\label{eq:prod3bis}
& \prod_{j=0}^{k-1}\left[(\gamma_{j}+p)^{4+\frac{6}{n}}\right]^{\frac{\left(1+\frac{2}{n}\right)^{k-1-j}}{\gamma_{k}+p}}
\leq
\prod_{j=0}^{k-1} \bigg[ p \bigg(\frac{n+2}{n}\bigg)^{j} \bigg] ^{4 \frac{\left(1+\frac{2}{n}\right)^{k-j}}{\gamma_{k}+p}}
\leq
p^{\frac{2 (n+2)}{p}}\left(\frac{n+2}{n}\right)^{\frac{n (n+2)}{p}}.
\end{align}
Combining estimates \eqref{eq:prod1bis}$-$\eqref{eq:prod3bis} and
recalling the definition of $c_{\varepsilon}$ in (\ref{eq:c_eps}), from (\ref{eq:recursive_ineq-1}) we obtain that 
\begin{equation}
M_{k}^{\frac{1}{\gamma_{k}+p}}  \leq \frac{c}{(1-s)^{\frac{n+2}{p}}} \left[1+ \frac{\Vert u_{\varepsilon}\Vert_{L^{\infty}(Q_{r}(z_{1}))}^{2}}{(1-s)^{2} r^{2}}\right]^{\frac{n+2}{2 p}} M_{0}^{\frac{\left(1+\frac{2}{n}\right)^{k}}{\gamma_{k}+p}}
\label{eq:iitt2}
\end{equation}
for any $k\in\mathbb{N}$, with $c\equiv c(N,n,\alpha,p,\delta,C_{1},K)>0$.
Since the constant $c$ is independent of $k\in\mathbb{N}$, recalling
(\ref{eq:mean_value}), (\ref{eq:iitt3}) and (\ref{eq:iitt4}) and
passing to the limit as $k\rightarrow\infty$ in both sides of \eqref{eq:iitt2},
we arrive at
\begin{equation}\label{eq:iitt_5}
\underset{Q_{sr}(z_{1})}{\mathrm{ess}\,\sup} \,H_{\delta}(Du_{\varepsilon})
\leq
\frac{c}{(1-s)^{\frac{n+2}{p}}} \left[1+ \frac{\Vert u_{\varepsilon}\Vert_{L^{\infty}(Q_{r}(z_{1}))}^{2}}{(1-s)^{2} r^{2}}\right]^{\frac{n+2}{2 p}}
\left[\fint_{Q_{r}(z_{1})}(H_{\delta}(Du_{\varepsilon}))^{p}\,dz\right]^{\frac{1}{p}},
\end{equation}
due to the fact that $r_{k}\searrow s r$ and $\gamma_{k}\nearrow\infty$
as $k\rightarrow\infty$. Notice that the right-hand
side of \eqref{eq:iitt_5} depends on $\Vert u_{\varepsilon}\Vert_{L^{\infty}(Q_{r}(z_{1}))}$.
However, since $f=0$, we can use the result from Section
\ref{sec:comparison}, thus obtaining
\[
\underset{Q_{sr}(z_{1})}{\mathrm{ess}\,\sup} \,H_{\delta}(Du_{\varepsilon}) \leq c\,  \frac{\left(1+\Vert u\Vert_{L^{\infty}(Q_{r}(z_{1}))}^{2}\right)^{\frac{n+2}{2 p}}}{[(1-s)^{2} r]^{\frac{n+2}{p}}} \left[\fint_{Q_{r}(z_{1})}(H_{\delta}(Du_{\varepsilon}))^{p}\,dz\right]^{\frac{1}{p}}.
\]
In conclusion, we have so far proved the following theorem, which
ensures the desired local $L^{\infty}$-bound of $Du_{\varepsilon}$
when $f=0$ and $p\in\left(1,\frac{2 n}{n + 2}\right]$:

\begin{thm}\label{thm:ite_theo2} 
Let $f=0$, $1<p\leq\frac{2 n}{n+2}$,
$\delta>0$ and $0<\varepsilon<\min \{\frac{1}{2},(\frac{\delta}{4})^{p-1}\}$. Moreover,
assume that 
\[
u_{\varepsilon}\in C^{0}\left([t_{0}-R^{2},t_{0}];L^{2}(B_{R}(x_{0}),\mathbb{R}^{N})\right)\cap L^{p}\left(t_{0}-R^{2},t_{0};W^{1,p}(B_{R}(x_{0}),\mathbb{R}^{N})\right)
\]
is the unique energy solution of problem
$(\ref{eq:CAUCHYDIR})$ with $Q'=Q_{R}(z_{0})\Subset\Omega_{T}$ and
$u$ a weak solution of $\mathrm{(\ref{eq:syst})}$,
satisfying the additional assumption of Remark
\ref{rmk2}. Then, for any parabolic cylinder
$Q_{r}(z_{1})\Subset Q_{R}(z_{0})$ with $r\in(0,1)$ and any $s\in(0,1)$,
we have 
\begin{equation}
\underset{Q_{sr}(z_{1})}{\mathrm{ess}\,\sup} \,H_{\delta}(Du_{\varepsilon}) \leq c\,  \frac{\left(1+\Vert u\Vert_{L^{\infty}(Q_{R}(z_{0}))}^{2}\right)^{\frac{n+2}{2 p}}}{[(1-s)^{2} \,r]^{\frac{n+2}{p}}} \left[\fint_{Q_{r}(z_{1})}(H_{\delta}(Du_{\varepsilon}))^{p}\,dz\right]^{\frac{1}{p}}\label{eq:ite_theo2_est}
\end{equation}
for some positive constant $c\equiv c(N,n,\alpha,p,\delta,C_{1},K)$.
\end{thm}

\section{Proof of Theorem \ref{thm:theo1}}\label{sec:Final}

We are now in a position to prove Theorem
\ref{thm:theo1}. In fact, using the strong $L^{p}$-convergence of
$G_{\delta}(Du_{\varepsilon})$ to $G_{\delta}(Du)$ established in
Lemma \ref{lem:Lpconv}{} as well as the results from
Theorems \ref{thm:ite_theo1} and \ref{thm:ite_theo2}, we will now
show that estimates (\ref{eq:ite_theo1_est}) and (\ref{eq:ite_theo2_est})
also hold true for the weak solution $u$ of (\ref{eq:syst}) in
place of $u_{\varepsilon}$. From this the desired local $L^{\infty}$-bounds
of $Du$ will immediately follow. Therefore, in the
next proof we will use the same assumptions and notations of Theorems
\ref{thm:ite_theo1} and \ref{thm:ite_theo2}.

\begin{proof}[\bfseries{Proof of Theorem~\ref{thm:theo1}}]
By virtue of Lemma \ref{lem:Lpconv}, we have that $G_{\delta}(Du_{\varepsilon})\rightarrow G_{\delta}(Du)$
strongly in $L^{p}(Q_{R}(z_{0}),\mathbb{R}^{Nn})$ as $\varepsilon\searrow0$.
Thus we conclude that there exists a sequence $\{\varepsilon_{j}\}_{j\in\mathbb{N}}$
such that:\vspace{0.1cm}

\noindent $\hspace*{1em}\bullet$ $  0<\varepsilon_{j}<\min \{\frac{1}{2},(\frac{\delta}{4})^{p-1}\}$
for every $j\in\mathbb{N}$ and $\varepsilon_{j}\searrow0$ as $j\rightarrow+\infty$
;\vspace{0.2cm}

\noindent $\hspace*{1em}\bullet$ $  \vert G_{\delta}(Du_{\varepsilon_{j}}(x,t))\vert\rightarrow\vert G_{\delta}(Du(x,t))\vert$
almost everywhere in $Q_{R}(z_{0})\Subset\Omega_{T}$ as $j\rightarrow+\infty$.\\

If $p>\frac{2 n}{n + 2}$,
then, using the definition of $G_{\delta}$ and Theorem \ref{thm:ite_theo1}, we have 
for almost every $\tilde{z}\in Q_{sr}(z_{1})$ that
\begin{align*}
\vert & Du(\tilde{z})\vert 
\leq
(1+\delta+\vert G_{\delta}(Du(\tilde{z}))\vert)
\leq
\lim_{j\rightarrow\infty} H_{\delta}(Du_{\varepsilon_{j}}(\tilde{z})) 
\leq \limsup_{j\rightarrow\infty}\,\,\,\underset{Q_{sr}(z_{1})}{\mathrm{ess \,sup}} \,H_{\delta}(Du_{\varepsilon_{j}})\\
& \leq
\limsup_{j\rightarrow\infty}  \frac{c}{(1-s)^{\frac{\hat{n}+2}{\kappa}}} \cdot\exp\left(c  \frac{\Vert f\Vert_{L^{\hat{n}+2}\log^{\alpha}L(Q_{r}(z_{1}))}^{\Theta}}{r^{\frac{n-\hat{n}}{\alpha}}}\right)\left[\fint_{Q_{r}(z_{1})}(1+\delta+\vert G_{\delta}(Du_{\varepsilon_{j}})\vert)^{p}\,dz\right]^{\frac{1}{\kappa}}\\
& =
\frac{c}{(1-s)^{\frac{\hat{n}+2}{\kappa}}} \cdot\exp\left(c  \frac{\Vert f\Vert_{L^{\hat{n}+2}\log^{\alpha}L(Q_{r}(z_{1}))}^{\Theta}}{r^{\frac{n-\hat{n}}{\alpha}}}\right)\left[\fint_{Q_{r}(z_{1})}(1+\delta+\vert G_{\delta}(Du)\vert)^{p}\,dz\right]^{\frac{1}{\kappa}}\\
& =
\frac{c}{(1-s)^{\frac{\hat{n}+2}{\kappa}}} \cdot\exp\left(c  \frac{\Vert f\Vert_{L^{\hat{n}+2}\log^{\alpha}L(Q_{r}(z_{1}))}^{\Theta}}{r^{\frac{n-\hat{n}}{\alpha}}}\right)\left[\fint_{Q_{r}(z_{1})}(\max \{\vert Du\vert,1+\delta\})^{p}\,dz\right]^{\frac{1}{\kappa}},
\end{align*}
where $\kappa$ is defined by (\ref{eq:kappa}), while $c$ is a positive
constant depending only on $N$, $n$, $\hat{n}$, $\alpha$, $p$,
$\delta$, $C_{1}$ and $K$. We now choose $\delta=1$. Since the above inequality holds for
almost every $\tilde{z}\in Q_{sr}(z_{1})$, we immediately obtain
\begin{align*}
\underset{Q_{sr}(z_{1})}{\mathrm{ess}\,\sup} \,\vert Du\vert
&\leq
\underset{Q_{sr}(z_{1})}{\mathrm{ess}\,\sup} \,H_{\delta}(Du)\\
&\leq
\frac{c}{(1-s)^{\frac{\hat{n}+2}{\kappa}}} \cdot\exp\left(c  \frac{\Vert f\Vert_{L^{\hat{n}+2}\log^{\alpha}L(Q_{r}(z_{1}))}^{\Theta}}{r^{\frac{n-\hat{n}}{\alpha}}}\right)\left[\fint_{Q_{r}(z_{1})}(\vert Du\vert + 1)^{p}\,dz\right]^{\frac{1}{\kappa}}< +\infty.
\end{align*}

Finally, let us consider the case
$f=0$ and $1<p\leq \frac{2 n}{n + 2}$. Arguing as above, but this time using Theorem \ref{thm:ite_theo2}
instead of Theorem \ref{thm:ite_theo1}, we get
\begin{align*}
\underset{Q_{sr}(z_{1})}{\mathrm{ess}\,\sup} \,\vert Du\vert
&\leq
\underset{Q_{sr}(z_{1})}{\mathrm{ess}\,\sup} \,H_{\delta}(Du)\\
&\leq
c  \frac{\left(1+\Vert u\Vert_{L^{\infty}(Q_{R}(z_{0}))}^{2}\right)^{\frac{n+2}{2 p}}}{[(1-s)^{2} r]^{\frac{n + 2}{p}}}
\left[\fint_{Q_{r}(z_{1})}(\vert Du\vert+1)^{p}\,dz\right]^{\frac{1}{p}}< +\infty.
\end{align*}
This concludes the proof.\end{proof}

\lyxaddress{\noindent \textbf{$\quad$}\\
 $\hspace*{1em}$\textbf{Pasquale Ambrosio}\\
 Dipartimento di Matematica e Applicazioni ``R. Caccioppoli''\\
 Università degli Studi di Napoli ``Federico II''\\
 Via Cintia, 80126 Napoli, Italy.\\
 \textit{E-mail address}: pasquale.ambrosio2@unina.it}

\lyxaddress{$\hspace*{1em}$\textbf{Fabian Bäuerlein}\\
 Fachbereich Mathematik, Universität Salzburg \\
 Hellbrunner Str. 34, 5020 Salzburg, Austria.\\
 \textit{E-mail address}: fabian.baeuerlein@plus.ac.at}

\end{document}